\documentclass[11pt]{article}
\usepackage[english]{babel}
\usepackage[T1]{fontenc}
\usepackage[utf8]{inputenc}
\usepackage{amsbsy}
\usepackage{amsmath}
\usepackage{amssymb}
\usepackage{amsfonts}
\usepackage{amsthm}
\usepackage{bm}
\usepackage{graphicx}
\usepackage{hyperref}
\usepackage[active]{srcltx}
\usepackage{upgreek}
\usepackage{xcolor}
\usepackage{dsfont}
\usepackage{subfig}

\usepackage{enumitem}
\setenumerate[2]{label=(\alph*)}
\setenumerate[1]{label=\textup{(\roman*)}}
\setenumerate[3]{label=(\arabic*)}

\newcommand{\C}{\mathbb{C}}
\newcommand{\N}{\mathbb{N}}

\newcommand{\R}{\mathbb{R}}
\renewcommand{\H}{\mathbb{H}}

\newcommand{\Z}{\mathbb{Z}}
\newcommand{\boA}{\mathcal{A}}

\newcommand{\boC}{\mathcal{C}}

\newcommand{\boE}{\mathcal{E}}
\newcommand{\boF}{\mathcal{F}}

\newcommand{\boM}{\mathcal{M}}

\newcommand{\boO}{\mathcal{O}}

\newcommand{\boS}{\mathcal{S}}

\newcommand{\boV}{\mathcal{V}}

\newcommand{\gp}{\mathfrak{p}}
\newcommand{\gq}{\mathfrak{q}}
\newcommand{\gr}{\mathfrak{r}}
\newcommand{\gs}{\mathfrak{s}}

\newcommand{\W}{\mathcal{W}}
\newcommand{\ve}{\varepsilon}

\newcommand{\vp}{\varphi}
\renewcommand{\k}{\textup{k}}
\newcommand{\p}{\textup{p}}

\newcommand{\grad}{\nabla}
\newcommand{\wh}{\widehat }

\newcommand{\Emin}{E_\textup{min}}
\newcommand{\ccc}{C_{0}^{\infty}(\mathbb{R})}
\newcommand{\boEE}{\mathcal{E}(\mathbb{R})}
\newcommand{\boEEE}{\mathcal{N}\mathcal{E}(\mathbb{R})}
\newcommand{\q}{\mathfrak{q}}

\newcommand{\Om}{\Omega}
\renewcommand{\d}{\textup{d}}

\renewcommand{\Im}{\mathop{{\rm Im}}\nolimits}
\DeclareMathOperator{\Ker}{{\rm Ker}}

\renewcommand{\Re}{\mathop{{\rm Re}}\nolimits}

\DeclareMathOperator{\supp}{{\rm supp}}
\DeclareMathOperator{\Res}{{\rm Res}}
\newcommand{\loc}{\operatorname{loc}}
\providecommand{\abs}[1]{|#1 |}
\providecommand{\norm}[1]{\lVert#1 \rVert}

\newcommand{\card}{\operatorname{Card}}
\renewcommand{\Re}{\operatorname{Re}}
\renewcommand{\Im}{\operatorname{Im}}
\newcommand{\sech}{\operatorname{sech}}

\newcommand{\wto}{\rightharpoonup}

\newcommand{\lpar}{\left(}
\newcommand{\rpar}{\right)}

\newtheorem{thm:intro}{Theorem}

\newtheorem{thm}{Theorem}[section]
\newtheorem{cor}[thm]{Corollary}
\newtheorem{lem}[thm]{Lemma}
\newtheorem{prop}[thm]{Proposition}
\newtheorem{thrm}[thm]{Theorem}

\newtheorem{lmm}[thm]{Lemma}
\theoremstyle{definition}
\newtheorem{claim}{Claim}

\newtheorem*{merci}{Acknowledgments}
\theoremstyle{definition}
\newtheorem{rem}[thm]{Remark}

\newcommand{\bqq}{\begin{equation*}}
	\newcommand{\eqq}{\end{equation*}}
\newcommand{\bq}{\begin{equation}}
	\newcommand{\eq}{\end{equation}}



\begin{document}
	\title{Traveling waves for some nonlocal 1D Gross--Pitaevskii equations    with nonzero conditions at infinity}                                 
	\author{
		\renewcommand{\thefootnote}{\arabic{footnote}}
		Andr\'e de Laire\footnotemark[1]~ and Pierre Mennuni\footnotemark[2]}
	\footnotetext[1]{
		Univ.\ Lille, CNRS, Inria, UMR 8524, Laboratoire Paul Painlevé, F-59000 Lille, France.\\
		E-mail: {\tt andre.de-laire@univ-lille.fr}}
	\footnotetext[2]{
		Univ.\ Lille, CNRS, Inria, UMR 8524, Laboratoire Paul Painlevé, F-59000 Lille, France.\\
		 E-mail: {\tt pierre.mennuni@univ-lille.fr}}
	\date{}
	\maketitle

	\begin{abstract}
		We consider a nonlocal family of Gross--Pitaevskii equations with nonzero conditions at infinity  in dimension one.
		We provide  conditions on the nonlocal interaction such that there is a branch of
		traveling waves solutions with nonvanishing conditions at infinity. 
		Moreover, we show that the branch is orbitally stable.
		In this manner, this result generalizes known properties for the contact interaction  given by a Dirac delta function. 	Our proof relies on the minimization of the energy at fixed momentum. 
		
		As a by-product of our analysis, we provide a simple condition to ensure that the solution to the Cauchy problem is global in time. 
		\end{abstract}   
	
	\maketitle
	
	\medskip
	\noindent{{\em Keywords:}
		Nonlocal Schr\"odinger equation, Gross--Pitaevskii equation, traveling waves, dark solitons, orbital stability,
		nonzero conditions at infinity
		


		\section{Introduction}\label{intro}
		\subsection{The problem}
		We consider the one-dimensional nonlocal Gross--Pitaevskii equation 
		for $\Psi: \R\times \R\to \C$ introduced by Gross~\cite{gross1963hydrodynamics} and 
		Pitaevskii \cite{pitaevskii1961vortex} to describe a Bose gas
		\bq
		\label{ngp}
		\tag{NGP}
		i\partial_{t}\Psi={\partial_{xx}\Psi}+\Psi(\W*(1-|\Psi|^{2}))\quad \text{ in }~\mathbb{R}\times\mathbb{R},
		\eq
		with the boundary condition at infinity 
		\bq
		\label{nonzero}
		\lim_{\abs{x}\to \infty}\abs{ \Psi}=1.
		\eq
		Here $*$ denotes the convolution in $\R$, and $\W$ is a real-valued even distribution 
		that describes the interaction between particles.
		The nonzero boundary condition \eqref{nonzero} arises as a background density.
		This model appears naturally in several areas of quantum physics, 
		for instance in the description of superfluids  \cite{berloff2008,abid2003}
		and  in optics when dealing with thermo-optic materials because
		the thermal nonlinearity is usually highly
		nonlocal \cite{vocke15}.
		An important property of equation \eqref{ngp}
		with the boundary condition at infinity \eqref{nonzero},
		is that it allows to study \textit{dark solitons}, i.e.\  localized density notches that propagate  without spreading \cite{kartashov07},
		that have been observed for example in Bose-Einstein condensates \cite{denschlag2000,becker2008}.

		%
		There have  been  extensive studies concerning the dynamics of equation \eqref{ngp}, and the 
		existence and stability of  traveling waves in the case of the {\em contact} interaction $\W=\delta_0$ (see
		\cite{bethuel2,bethuel,orlandiun,BGS-2015,chironmaris,chironstab,gerard,maris2013,delaire2009,gustafson2009,gustafson2007,visan2012} and the references therein). However,
		there are very few mathematical results concerning general 
		nonlocal interactions with nonzero conditions at infinity.
		In \cite{de2010global,pecher2012} the authors gave conditions on $\W$
		to get global well-posedness of the equation and in \cite{de2011nonexistence}
		 conditions were established for the nonexistence of traveling waves (in higher dimensions). 
		Nevertheless, to our knowledge,  there is no result concerning the existence of localized solutions to \eqref{ngp} when $\W$ is not given by a Dirac delta. The aim of this paper is to provide  conditions on $\W$ in order to have stable 
		finite energy traveling wave solutions, more commonly refereed to as dark solitons due to the nonzero boundary condition \eqref{nonzero}. 
		More precisely, we look for  a solution of the form
		$$\Psi_c(x,t)=u(x-ct),$$
		representing a traveling wave propagating at speed $c$. Hence, the profile $u$ satisfies
		the nonlocal ODE
		\begin{equation}\label{ntwc}\tag{TW$_{\mathcal W,c}$}
			ic u'+u''+u(\mathcal  \W*(1-\abs u ^2))=0 \quad  \text{ in } \R.
		\end{equation}
		%
		%
		%
		%
		%
		%
		%
		By taking the conjugate of the function, we  assume 
		without loss of generality that $c\geq 0$.

		Let us remark that  when considering vanishing boundary conditions at infinity, 
		this kind of equation has been studied extensively \cite{ginibre1980,cazenave,moroz2013}
		and long-range dipolar interactions in condensates have received recently
		much attention \cite{lahaye2009,carles2008,antonelli2011,jacopo2016,luo2018}. 
		However, the techniques used in these works cannot be adapted to include solutions satisfying \eqref{nonzero}.

		We recall that \eqref{ngp} is Hamiltonian and its energy
		\begin{equation*}
			E(\Psi(t))=\frac12 \int_{\R}\abs{\partial_x\Psi(t)}^2\,dx +\frac 14 \int_{\R}(W*(1-\abs{\Psi(t)}^2))(1-\abs{\Psi(t)}^2)\,dx,
		\end{equation*}
		is formally conserved, as well as the (renormalized) momentum 
		\bqq
		p(\Psi(t))=\int_{\mathbb{R}}\langle i \partial_x \Psi'(t),\Psi(t) \rangle\left(1-\frac{1}{|\Psi(t)|^2}\right)dx,
		\eqq
		at least as $\inf_{x\in\R}\abs{\Psi(x,t)}>0$, where $\langle z_1,z_2 \rangle =\Re(z_1 \bar z_2)$, for $z_1$, $z_2\in \C$ (see \cite{delaire2009,bogdan1989}).
		In this manner, we seek nontrivial solutions of \eqref{ntwc} in the energy space 
		$$\mathcal{E}(\mathbb{R})=
		\{v \in H^{1}_{\loc}(\mathbb{R}) : 1-|v|^{2}\in L^{2}(\mathbb{R}), \ v' \in L^{2}(\mathbb{R})\},$$
		and more precisely in the nonvanishing energy space
		$$\boEEE=
		\{v \in \boEE  : \inf_{\R}\abs{v}>0\},$$
		where the momentum will be well defined. It is simple to check, using the Morrey inequality, that the functions in $\boEE$ are uniformly continuous
		and satisfy  $\lim_{|x|\to \infty}|v(x)|=1$.
		

		%
		%
		%
		%

		When $\W$ is given by a  Dirac delta function,  equation (TW$_{\delta_0,c}$) corresponds to  
		the classical Gross--Pitaevskii equation, which can be solved explicitly. 
		As explained in \cite{bethuel2008existence}, if  $c\geq \sqrt{2}$ the only solutions in $\boE(\R)$
		are the trivial ones 
		(i.e.\ the constant functions 
		of modulus one) and if  \mbox{$0\leq c<\sqrt{2}$}, the nontrivial solutions are given, 
		up to invariances (translations and a multiplications by constants of modulus one), by
		\bq
		\label{sol:1D}
		u_{c}(x)=\sqrt{\frac{2-c^2}{2}}\tanh\left(\frac{\sqrt{2-c^2}}{2}x\right)-i\frac{c}{\sqrt{2}}.
		\eq
		Thus there is a family of dark solitons belonging to $\boEEE$ for $c\in (0,\sqrt 2)$
		and there is one stationary black soliton associated with the speed $c=0$.
		Notice also that the values of $u_c(\infty)$ and $u_c(-\infty)$ are different, 
		and thus we cannot relax the condition \eqref{nonzero} to $\lim_{\abs{x}\to\infty}{\Psi}=1$,
		as is usually done in higher dimensions.

		The study of  equation (TW$_{\delta_0,c}$) can
		be generalized  to other 
		types of {\em local} nonlinearities such as the  cubic-quintic nonlinearity and
		some cubic-quintic-septic nonlinearities as shown in \cite{chironexistence1d,maris-non}. 
		The techniques used by the authors rely on the analysis of a second-order ODE of Newton type, 
		so that the Cauchy--Lipschitz theorem can be invoked and some explicit formulas can be deduced.
		These arguments cannot be applied to \eqref{ntwc} due to the nonlocal interaction. For this reason, 
		our approach to show existence of traveling waves relies on 
		a priori energy estimates 
		and a concentration-compactness argument, that allow us to prove that there are functions 
that minimize  the energy at fixed momentum. These minimizers  are solutions to \eqref{ntwc} and 		
		 we can also establish that  they are orbitally stable (see Theorem~\ref{thm-existence-general}). These kinds of arguments have been used by several authors to establish existence of solitons 
		for the (local) Gross--Pitaevskii equation in higher dimensions and for some related equations  with zero conditions at infinity (see e.g.\ \cite{bethuel,maris2013,chironmaris,lopes-maris,maris2016,audiard2017,lieb77}).
		The main difficulty in our case is to handle the nonvanishing conditions at infinity, the fact that the constraint  given by the momentum is not a homogeneous function along with the  nonlocal interactions.
		\subsection{The critical speed and  assumptions on $\W$}
		\label{subsec:assumptions}
		Linearizing equation \eqref{ngp} around the constant solution equal to 1 and
		imposing $e^{i(\xi x-wt)}$ as a solution of the resulting equation, 
		we obtain the dispersion relation 
		\bq\label{bogo}
		w(\xi)=\sqrt{\xi^4+2\widehat \W(\xi) \xi^2},
		\eq
		where $\wh \W$ denotes the Fourier transform of $\W$.
		Supposing that $\wh \W$ is positive and continuous  at the origin, 
		we get  the so-called speed of sound 
		\bqq
		c_*(\W)= \lim_{\xi\to 0}\frac{w(\xi)}{\xi}=\sqrt{2\wh \W(0)}.
		\eqq
		The dispersion relation  \eqref{bogo} was first observed by Bogoliubov \cite{bogo}
		in the study of a Bose--Einstein gas. He then argued that the gas should move with a speed less than $c_*(\W)$ to preserve its superfluid 
		properties. This leads to the conjecture that there is no nontrivial 
		solution of \eqref{ntwc} with finite energy when $c> c_*(\W)$. 
		Actually, one of the authors  proved  this conjecture  in \cite{de2011nonexistence} in dimensions greater than 	one, under some conditions on $\W$.
		
		In order to simplify our  computations, we can normalize the equation so that the critical speed is fixed.		Indeed, it is easy to verify that the rescaling $x\mapsto x/\widehat{\W}(0)^{1/2}$ and $t\mapsto t/\widehat{\W}(0)$
		allows us to replace $\wh \W(\xi)$ by $\wh \W(\xi)/\wh \W(0)$ in \eqref{ngp}. Therefore, 
		we assume from now on that $\wh \W(0)=1$ and hence that the critical speed is 
		\bqq
		c_*=\sqrt 2.
		\eqq
		Before going any further, let us state the assumptions that we need on $\W$. 
		\begin{enumerate}[label=({H\arabic*}),ref=\textup{({H\arabic*})}]
			\item\label{H0} $\W$ is an even tempered distribution with $\wh \W\in L^\infty(\R)$,
			and $\wh \W \geq 0$ a.e.\ on $\R$. Moreover $\wh \W$ is continuous at the origin and $\wh \W(0)=1$.
			\item\label{H-coer}$\wh \W$ belongs to $C_b^3(\R)$, $(\wh \W)''(0)>-1$   and $\wh \W(\xi)\geq 1-\xi^2/2$, for all  $\abs{\xi}< 2.$
			\item\label{H-residue}
			$\wh \W$ admits a meromorphic extension to the upper half-plane 
			$\H:=\{z \in \C : \Im(z)>0\}$, and 
			the only possible singularities of $\wh \W$ on $\H$ are simple isolated poles
			belonging to the imaginary axis, i.e.\ they are given by
			$\{i \nu_j : j\in J\},$ with $\nu_j>0$, for all $j\in J$,  $0\leq\card{J}\leq\infty$,
			and their residues  $\Res(\wh \W,i\nu_j)$ are purely imaginary numbers satisfying
			\bq
			\label{resW}
			i\Res(\wh \W,i\nu_j)\leq 0, \quad \textup{ for all }j\in J,
			\eq 
			Also, there exists a sequence of  rectifiable curves $(\Gamma_k)_{k\in \N^*}\subset \H$,
			parametrized by $\gamma_k : [a_k,b_k]\to \C$,   such that 
			$\Gamma_k\cup [-k,k]$ is a closed  positively oriented simple curve that does 
			not pass through any poles. Moreover, 
			\bq
			\label{decayW}
			\lim_{k\to \infty} \abs{\gamma_k(t)}=\infty, \textup{ for all }t\in [a_k,b_k],
			\quad \text{ and }\quad 
			\lim_{k\to\infty } \textup{length}(\Gamma_k)\sup_{t\in[a_k,b_k]}\frac{\wh \W(\gamma_k(t))}{\abs{\gamma_k(t)}^4}=0.
			\eq
			
			%
			%
		\end{enumerate}
		Here $C_b^k(\R)$ denotes the bounded functions of class $C^k$ 
		whose first $k$ derivatives are bounded.
		We have also used the convention that the Fourier transform of (an integrable) function is 
		$$\hat f(\xi)=\int_{\R} e^{-ix\xi}f(x)dx.$$
		In particular,  the Fourier transform of the Dirac delta is $\hat \delta_0=1$ and 
		thus assumptions \ref{H0}--\ref{H-residue} are trivially fulfilled by $\W=\delta_0$.
		Let us make some further remarks about these hypotheses. Assumption  \ref{H0} ensures 
		 that the critical speed exists  and that  the energy functional is nonnegative and well defined	in $\boEE$. Indeed, let us consider $v\in\boEE$, set $\eta=1-\abs{v}^2$ and  write the energy in terms of the kinetic and potential energy as
		\begin{align*}
			E(v)=E_\k(v)+E_\p(v),\quad\text{ where }\ E_\k(v):=\frac{1}{2}\int_{\mathbb{R}}|v'|^{2}dx \ \text{ and }\
			E_\p(v):=\frac{1}{4}\int_{\mathbb{R}}(\W*\eta)\eta.
		\end{align*} 
		By hypothesis~\ref{H0} and the Plancherel theorem, we deduce that 
		\bqq
		0\leq E_p(v)=\frac{1}{8\pi}\int_\R \wh \W\abs{\hat \eta}^2\leq \frac14\norm{\wh \W}_{L^\infty}\norm{\eta}_{L^2}^2,
		\eqq
		so that the functions in $\boEE$ have indeed finite energy and their potential energy is nonnegative. 

Let us recall that for a tempered distribution $\boV\in  S'(\R)$, we can define 
the convolution with a function in $L^p(\R)$,  
through the Fourier transform, as 
the bounded extension on $L^p(\R)$ of the operator
$$\boV*f:=\boF^{-1}(\wh \boV \ \hat f), \quad f\in S(\R).$$
In this manner, the set
$$\boM_{p}(\R)=\{\boV\in S'(\R) : \exists C>0, \norm{\boV *f}_{L^p(\R)}\leq 
C\norm{f}_{L^p(\R)},  \forall f\in L^p(\R) \}$$
is a Banach space endowed with the operator norm denoted by $\norm{\cdot}_{\boM_p}$. 
Thus   \ref{H0}
		implies that $\W\in \boM_{2}(\R)$, with
				$$\norm{\wh \W}_{L^\infty(\R)}=\norm{\W}_{\boM_2}.$$
We refer to  \cite{grafakos} for further details about the properties of
$\boM_{p}(\R)$.

		Hypothesis~\ref{H-coer}, combined with~\ref{H0}, imply that $\wh \W(\xi)\geq (1-\xi^2/2)^+$ a.e.,
that can be seen as
a		coercivity property for the energy. In particular, it will allow us to establish the key energy  estimates in Lemmas~\ref{lem-control-energy}
		and \ref{lmm:generalpE}. 
%
The condition $(\wh \W)''(0)>-1$ will be   crucial to show that  the behavior of 
a solution of \eqref{ntwc} can be formally described in terms of 
the solution of the Korteweg--de Vries equation 
		\bqq
		(1+(\wh \W)''(0))A''-6A^2-A=0,
		\eqq
at least for $c$ close to $\sqrt{2}$ (see Section \ref{sec:curve}).
		
		The more technical and restrictive assumption \ref{H-residue} 
		is used only to prove that the curve associated with the minimizing problem  is concave.
		Indeed, we use some ideas introduced by Lopes and
		 Mari\c s~\cite{lopes-maris} to study 
		the minimization of the nonlocal functional
		$$\int_{\R^N} m(\xi)\abs{\hat w(\xi)}^2d\xi+\int_{\R^N}  F(w(x))dx,$$
		under the constraint $\int_{\R^N}G(w)dx=\lambda$, $\lambda\in \R$, 
		for a class of symbols $m$ (see (2.16) in \cite{lopes-maris}). Here $N\geq 2$,  $F$ and $G$ are local functions, and the minimization is over  
		$w\in H^s(\R)$. The results
in		\cite{lopes-maris} cannot be applied to  the symbol $m(\xi)=\wh \W(\xi)$
		nor to the minimization over functions with nonvanishing conditions at infinity (nor $N=1$).
		However, we can still apply the reflexion argument in  \cite{lopes-maris}, which will lead us to show that 
		\bq 
		\label{ineg-W-f}
		\int_\R (\W*f)f\geq  		\int_\R (\W*\tilde f)\tilde f,
		\eq
		for all  odd functions $f\in C_c^\infty(\R)$, where  $\tilde f$ is given by $\tilde f(x)=f(x)$ for $x\in\R^+$, and $\tilde f(x)=-f(x)$ for $x\in\R^-$.
		Using the sine and cosine transforms
		\bqq
		\hat f_s(\xi)=\int_0^\infty \sin(x\xi)f(x)dx,\quad  
		\hat f_c(\xi)=\int_0^\infty \cos(x\xi)f(x)dx, 
		\eqq
		we will see in Section~\ref{sec:curve} that  inequality \eqref{ineg-W-f} is equivalent to the
		following  assumption.
		
		\begin{enumerate}[label=({H\arabic*'}),ref=\textup{({H\arabic*'})}]
			\setcounter{enumi}{2}
			\item
			\label{H-residue-bis} $\W$ satisfies 
			\bqq \int_0^\infty \wh \W(\xi)(\abs{\hat f_s(\xi)}^2-\abs{\hat f_c(\xi)}^2)d\xi\geq 0,
						\eqq
			for all  odd functions  $f\in C_c^\infty(\R)$.
		\end{enumerate}
		Therefore, we can replace \ref{H-residue} by the weaker (but less explicit) condition \ref{H-residue-bis}. Finally, let us notice that if $\W=\delta_0$, we can verify
		that  condition \ref{H-residue-bis} is satisfied by using the Plancherel formula
		\bqq
		\int_0^\infty \abs{\hat f_s(\xi)}^2d\xi =
		\int_0^\infty \abs{\hat f_c(\xi)}^2d\xi=
		\int_0^\infty \abs{ f(x)}^2dx.
		\eqq
		At the end of this section we will give some examples of potentials satisfying \ref{H0}--\ref{H-residue}. 
				\subsection{Main results}
		\label{statement}
		In the classical minimization problems 
associated with  Schr\"odinger equations with vanishing conditions at infinity, 
the constraint in given by the mass. In our case, the momentum is 
the key quantity that we need to take as a constraint to show the existence of  dark solitons.
		Let us verify that the momentum
		\bq
		\label{def:mom}
		p(v)=\frac{1}{2}\int_{\mathbb{R}}\langle iv',v \rangle\left(1-\frac{1}{|v|^2}\right),
		\eq
		is well defined in the nonvanishing energy space.
		Indeed, a function  $v\in \boEEE$ is continuous and admits a lifting $v=\rho e^{i\phi}$,
		where $\rho=\abs{v}$ and $\phi$ are real-valued functions in $H^1_{\loc}(\R)$ (see e.g.~\cite{gerard3}). Since $v\in \boEEE$, we have $\inf_\R \rho>0$,  
and using that
$$\abs{v'}^2=\rho'^2+\rho^2\phi'^2,$$
		we infer that $\abs{\phi'}\leq \abs{v'}/\inf_\R \rho$, so that $\phi'\in L^2(\R)$.
		Hence, setting $\eta=1-\abs{v}^2\in L^2(\R)$, we get that the
		integrand in \eqref{def:mom} is equal to $\eta \phi'$, and therefore 
\eqref{def:mom} is well-defined since $\eta \phi'\in L^1(\R)$.
In conclusion, for any $v\in\boEEE $, the energy and the momentum
		can be written as
		\bqq
		E(v)=\frac12 \int_\R \rho'^2
		+\frac12 \int_\R \rho^2\phi'^2+
		\frac12 \int_\R (\W*\eta)\eta \quad \text{ and }\quad
		p(v)=\frac12\int_\R \eta\phi',
		\eqq
		under the assumption $\wh \W \in L^{\infty}(\R)$.
				
			Let us now describe our minimization approach for the existence problem, 
			assuming that $\W$ satisfies \ref{H0} and \ref{H-coer}. 
		For $\gq\geq 0$, we consider the minimization curve
		\bqq
		\Emin(\gq):=\inf\{E(v) : v\in \boEEE,\ p(v)=\gq \},
		\eqq
		that is well defined in view of Lemma~\ref{lemmeEminborne}. Moreover, this curve is nondecreasing (see Lemma~\ref{lem:increasing}). 
		We also set 
		\bq
		\label{q_*}
		\gq_*=\sup\{\gq>0~|~\forall v \in \mathcal{E}(\mathbb{R}), E(v)\leq E_{\min}(\gq)\Rightarrow \underset{\mathbb{R}}\inf|v|>0\}.
		\eq
If \ref{H-residue} is also fulfilled and $\q\in(0,\q^*)$, 
we will show that minimum  associated with $\Emin(\q)$  is attained 
and that the corresponding 		 Euler--Lagrange equation 
satisfied by the   minimizers 
		is exactly \eqref{ntwc}, where $c$ appears as a Lagrange multiplier (see Section~\ref{sec:Euler}
		for details). More precisely, our first result establishes the existence of a family of solutions of \eqref{ntwc} parametrized by the momentum.
		\begin{thm:intro}
			\label{thm-existence}
			Assume that \ref{H0},  \ref{H-coer} and \ref{H-residue} hold.
			Then  $\gq_*>0.027$ and for all $\gq\in (0,\gq_*)$ there is a nontrivial solution $u\in \boEEE$ to \eqref{ntwc}
			satisfying  $p(u)=\gq$,
			for some $c \in (0,\sqrt 2)$.
		\end{thm:intro}
			It is important to remark that the constant $\q_*$ is not necessarily small. 
				For instance,  in the case  $\W=\delta_0$,
		the explicit solution \eqref{sol:1D} allows us to compute 
		the momentum of $u_c$,  for $c\in (0,\sqrt 2)$, and to deduce that $\q_*=\pi/2$.
		Moreover  $\Emin$ can be determined and its profile is depicted in Figure~\ref{courbeth}. 
		Notice that $\Emin$ is constant on $(\q_*,\infty)$ and that in this interval the minimum  is not  attained  (see e.g.\ \cite{bethuel2008existence}).
\begin{figure}
\begin{center}
	\begin{tabular}{cc}
		\hspace*{-0.6cm}
{\scalebox{1}{\includegraphics[trim={0.15cm 0.2cm 0.05cm 0cm}, clip]{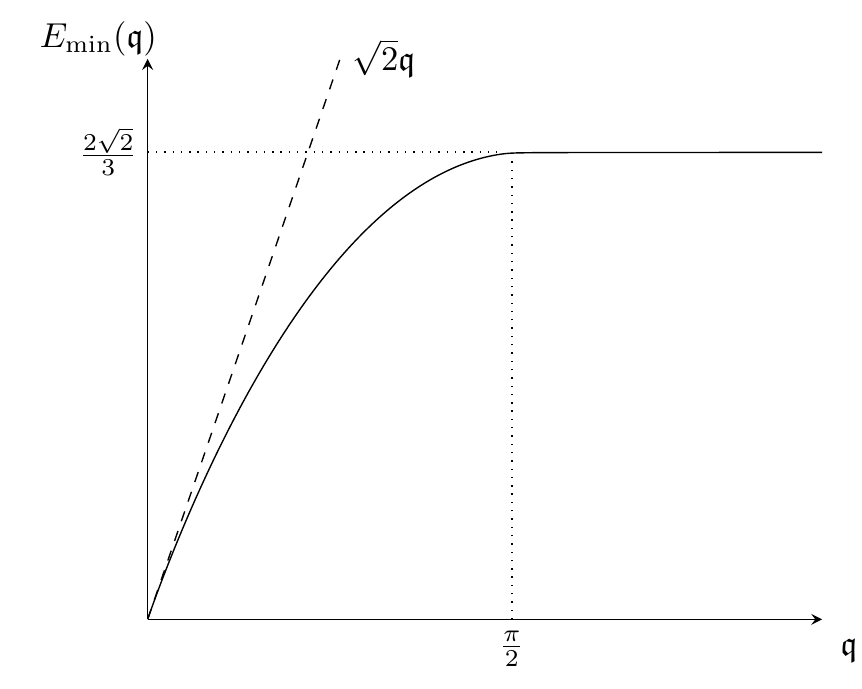}}}    
&
\hspace*{-1.45cm}
{\scalebox{0.93}{\includegraphics[trim={0.2cm 1.1cm 0.3cm 0cm}, clip]{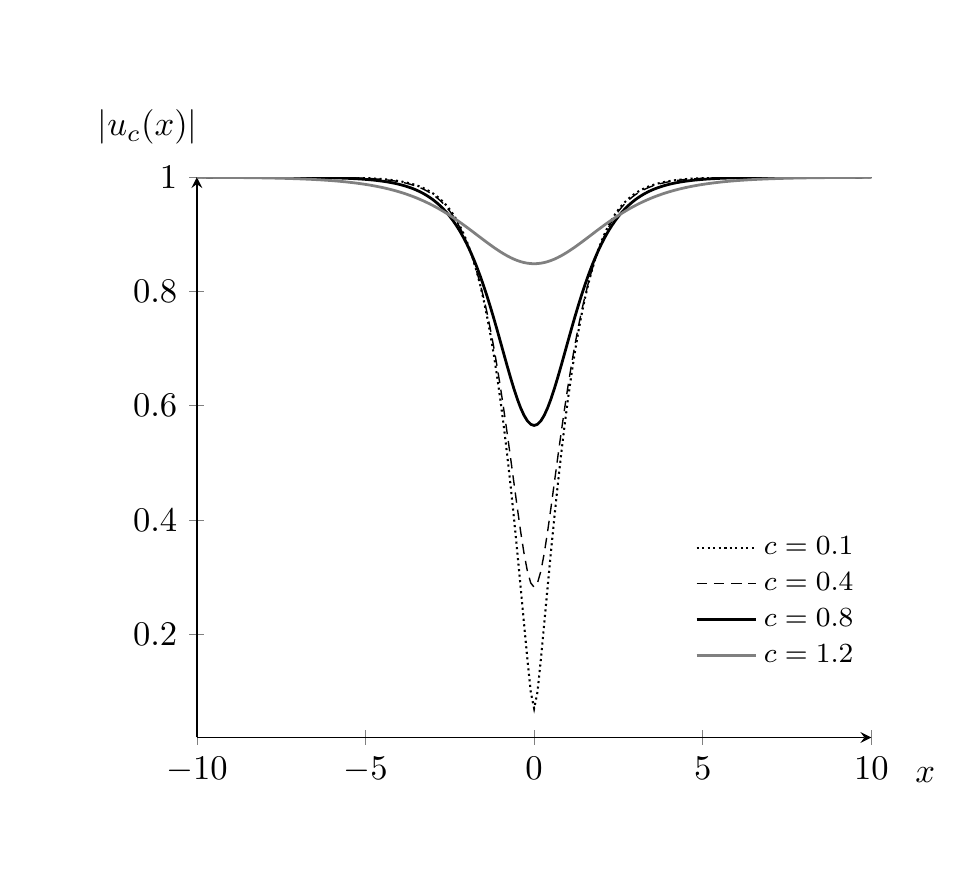}}}
\end{tabular}
\end{center}
\caption{Curve $\Emin$ and solitons in the case $\W=\delta_0$.}
\label{courbeth}
\end{figure}
Since \ref{H0}--\ref{H-residue} are satisfied by $\W=\delta_0$,
and since there is uniqueness (up to invariances) of the solutions to (TW$_{\delta_0,c}$), 
we deduce that the branch of solutions given by Theorem~\ref{thm-existence} corresponds to 
the dark solitons in \eqref{sol:1D},
for $c\in(0,\sqrt2)$.
		In the general case, we do not know if the solution given by Theorem~\ref{thm-existence} is unique (up to invariances).
Actually, the uniqueness for nonlocal equations such as \eqref{ntwc} can be difficult to establish (see e.g.~\cite{albert95,lieb77})
		and goes beyond the scope of this work. Concerning the regularity, the solutions given by
		Theorem~\ref{thm-existence} are smooth and we refer 
		to Lemma~\ref{regularity} for a precise statement.

To establish Theorem~\ref{thm-existence}, we analyze two problems. First, we  provide some general properties 
of the curve $\Emin$. Then, we study the compactness of the minimizing sequences associated with $\Emin$.
The next result summarizes the properties of $\Emin$.
		\begin{thm:intro}
			\label{thm:curve}
			Suppose  that $\W$ satisfies  \ref{H0} and \ref{H-coer}. 
			Then the following statements hold.
			\begin{enumerate}
				\item The function $\Emin$ is even and Lipschitz continuous on $\R$, with 
				$$\abs{\Emin(\gp)-\Emin(\gq)}\leq \sqrt{2}\abs{\gp-\gq},\quad\text{ for all }\gp,\q\in\R.$$
				Moreover, it is nondecreasing and subadditive on $\R^+$.
				\item There exist  constants $\q_1,A_1,A_2,A_3>0$ such that 
				$$\sqrt2\gq- A_1\q^{3/2}\leq 
				\Emin(\q)\leq \sqrt{2}\q -A_2\q^{5/3}+A_3\q^{2},\quad \text{for all }\q\in [0,\q_1].$$
				\item If \ref{H-residue} or \ref{H-residue-bis} is satisfied, then $\Emin$ is concave on $\R^+$.
				\item 	We have $\q_*>0.027$. If  $\Emin$ is concave on $\R^+$, then $\Emin$ is strictly increasing on $[0,\q_*)$,
				and for all $v\in \boEE$ satisfying $E(v)<\Emin(\q_*)$, we have $v\in \boEEE$.
				\item Assume that $\Emin$ is concave on $\R^+$. Then $\Emin(\q)<\sqrt 2\q$, for all $\q>0$,
				$\Emin$ is strictly subadditive  on $\R^+$, and 
				the right and left derivatives of $\Emin$, denoted by $E^{+}_{\min}$ and $E^{-}_{\min}$
				respectively, satisfy
				\bq
				\label{emin:fixed:bounds}
				0 \leq E^{+}_{\min}(\gq)  \leq E^{-}_{\min}(\gq) <\sqrt{2}.
				\eq
				Furthermore, $E^{+}_{\min}(\gq)\to E^{+}_{\min}(0)= \sqrt{2}$, as $\gq \to 0^+$.
			\end{enumerate}
		\end{thm:intro}
		
To prove the existence of solutions we use a concentration-compactness argument.
Applying Theorem~\ref{thm:curve}, we show that the minimum is attained at least for $\q\in(0,\q_*)$,
so that the set 
 $$\boS_\gq=\{ v\in \boEEE : E(v)=\Emin(\gq) \text{ and }p(v)=\gq\}$$
is nonempty, and thus there are nontrivial solutions to \eqref{ntwc} (see Theorem~\ref{thm:euler}).
Hence, we can rely on the Cazenave--Lions~\cite{cazlions} argument to show that the solutions are stable. Let us  remark that the 
Cauchy problem for \eqref{ngp} was studied in \cite{de2010global}. Precisely, using the distance 
\bqq
d_{\boE}(v_1,v_2)=\norm{v_1-v_2}_{L^2(\R)+L^\infty(\R)}+
\norm{v'_1-v'_2}_{L^2(\R)}+
\norm{\abs{v_1}-\abs{v_2}}_{L^2(\R)},
\eqq
 the energy space $\boEE$ is a complete metric space and for every $\Psi_0\in \boEE$
 there is a unique global solution $\Psi\in C(\R,\boEE)$ with initial condition $\Psi_0$,
provided that $\W\in \boM_{3}(\R)$ and that $\W\geq 0$ or that $\inf_{\R} \wh \W>0$ (see Theorem~\ref{global0}). However, these conditions are not 	necessarily fulfilled
by a distribution satisfying \ref{H0}--\ref{H-coer}.  
Nevertheless, using the energy estimates in Section~\ref{sec:estimations}, we can generalize a result 
in \cite{de2010global}  in the following way.
	\begin{thm:intro}
	\label{thm:cauchy0}
	Assume that  $\W\in \boM_{3}(\R)$  is an even distribution, 
	with $\wh \W\geq 0$ a.e.~on $\R$,  and that $\wh \W$ of class $C^2$ in a neighborhood of the origin
	with $\wh \W(0)=1$.	 Then for every $\Psi_0 \in \mathcal{E}(\mathbb{R}),$ there exists a unique $\Psi \in C(\mathbb{R},\mathcal{E}(\mathbb{R}))$  {\em global}  solution to \eqref{ngp} with the initial condition $\Psi_0$. Moreover, the energy is  conserved, as well as the momentum as long as 
	$\inf_{x\in\R}\abs{\Psi(x,t)}>0$.
\end{thm:intro}
\begin{rem}
	As explained before, the condition  $\wh \W(0)=1$ in Theorem~\ref{thm:cauchy0} is due to the normalization, and it can be replaced  by $\wh \W(0)>0$.
\end{rem}
We can also endow $\boE(\R)$ with 
 the pseudometric distance 
\bqq
d(v_1,v_2)=\norm{v'_1-v'_2}_{L^2(\R)}+
\norm{\abs{v_1}-\abs{v_2}}_{L^2(\R)},
\eqq
or with the distance used in \cite{bethuel2008existence}
\bqq
d_A(v_1,v_2)=\norm{v'_1-v'_2}_{L^2(\R)}+
\norm{\abs{v_1}-\abs{v_2}}_{L^2(\R)}+\norm{v_1-v_2}_{L^\infty([-A,A])},
\eqq
for $A>0$. Notice that $d(v_1,v_2)=0$ if and only if $\abs{v_1}=\abs{v_2}$ and $v_1-v_2$ is constant.
We say that the set  $\mathcal{S}_\q$ is orbitally stable in $(\boEE,d)$ if 
for all $\Psi_0\in \boEE$ and for all $\varepsilon>0,$ there exists $\delta >0$ such that if 
 $$ d(\Psi_0,\boS_\q)\leq \delta,$$
then the solution $\Psi(t)$ of \eqref{ngp} associated with the initial condition $\Psi_0$ satisfies 
$$\sup_{t\in\R } d(\Psi(t),\boS_\q)\leq \varepsilon.$$
Similarly, the set  $\mathcal{S}_\q$ is orbitally stable in $(\boEE,d_A)$ if
for all $\Psi_0\in \boEE$ and for all $\varepsilon>0,$ there exists $\delta >0$ such that if 
$ d_A(\Psi_0,\boS_\q)\leq \delta,$ then $\sup_{t\in\R } \inf_{y\in \R}d_A(\Psi(\cdot-y,t),\boS_\q)\leq \varepsilon.$
Here we need to introduce a translation of the flow, since the $d_A$ is not invariant under translations.

Now we can state our main result concerning the existence and stability of traveling waves.		
		\begin{thm:intro}		
			\label{thm-existence-general}
			Suppose  that $\W$ satisfies  \ref{H0} and \ref{H-coer}, and that $\Emin$ is concave on $\R^+$.
			Then the set $\boS_\gq$ is nonempty,  for all $\q\in (0,\q_*)$. 	Moreover, 	every $u\in \boS_\gq$ is a solution of \eqref{ntwc} for some speed $c_\q\in(0,\sqrt 2)$
			satisfying 
			\begin{equation}
				\label{speed0}
				 E^{+}_{\min}(\gq) \leq c_\gq \leq E^{-}_{\min}(\gq).
			\end{equation}
			Also, $c_\gq\to \sqrt{2}$ as $\gq \to 0^+$.
			
			In addition, if  $\W\in \boM_3(\R)$, then  $\boS_\gq$ is orbitally stable in $(\boEE,d)$
			and in   $(\boEE,d_A)$, 	 for all $\q\in(0,\q_*)$. 	Furthermore,
			for all $\Psi_0\in \boEE$ and for all $\varepsilon>0,$ there exists $\delta >0$ such that if 
			$ d(\Psi_0,\boS_\q)\leq \delta,$
			then the solution $\Psi(t)$ of \eqref{ngp} associated with the initial condition $\Psi_0$ satisfies 
			$$\sup_{t\in\R } \inf_{y\in \R}d_A(\Psi(\cdot-y,t),\boS_\q)\leq \varepsilon.$$
					\end{thm:intro}
	In this manner, it is clear that Theorem~\ref{thm-existence} is an immediate corollary of Theorems~\ref{thm:curve} 
	and \ref{thm-existence-general}, and that the branch of solutions given by Theorem~\ref{thm-existence} 
	is orbitally stable provided that $\W\in \boM_3(\R)$. In particular, 
	we recover the orbital stability proved by several authors for the solitons given in~\eqref{sol:1D} 
(see e.g.\ \cite{linbubbles,bethuel2,chironstab} and the references therein).


		We point out that we have not discussed what happens with the minimizing curve for $\q\geq \q_*$. 
As mentioned before, for all $\q>\q_*$, the curve $\Emin(\q)$ is constant for	 $\W=\delta_0$ (see Figure~\ref{courbeth}) and  $\boS_q$ is empty. Moreover,  the critical case $\q=\q_*$ is associated with  the black soliton and its analysis is more involved (see e.g.~\cite{blacksoliton,asymptstabblack}). 
Numerical simulations lead us to conjecture that similar results hold for a potential satisfying \ref{H0}-\ref{H-residue},
i.e.~that $\Emin(\q)$ is constant and that $\boS_q$ is empty on $(\q_*,\infty)$, and that there is a black soliton when $\q=\q_*$.
 		In addition, in the performed simulations the value $\q_*$ 
 		is close to   $\pi/2$ (see Section~\ref{sec:numerics}).
		Furthermore, these simulations also show that \ref{H-coer} and \ref{H-residue-bis}
				are not necessary for the concavity of $\Emin$ nor the existence of solutions of 
				\eqref{ntwc}. We think that \ref{H-coer} could be relaxed, but that the condition 
$(\wh \W)''(0)>-1$ is necessary. As seen from Theorem~\ref{thm:curve}, we have only used
\ref{H-residue-bis} as a sufficient condition to ensure the concavity of $\Emin$.
If for some $\W$ satisfying \ref{H0} and \ref{H-coer}, 
one is capable of showing that $\Emin$ is concave, then the existence and stability of solutions of 
\eqref{ntwc} is a consequence of Theorem~\ref{thm-existence-general}.

		In addition to the smoothness of the obtained solutions (see Lemma~\ref{regularity}), 
		it is possible to study further properties of these solitons such as their 
		decay at infinity and uniqueness (up to invariances).
		Another related open problem is to show the nonexistence of traveling waves for $c>\sqrt{2}$.
		We will study  these questions in a forthcoming paper.
		
We give now some examples of potentials satisfying conditions \ref{H0}, \ref{H-coer}	and \ref{H-residue}
		\begin{enumerate}
			\item 
			For  $\beta >2\alpha>0$, we consider 
$\W_{\alpha,\beta}=\frac{\beta}{\beta-2\alpha}(\delta_0-\alpha e^{-\beta\abs{x}})$,  so its Fourier transform is
			$$\wh \W_{\alpha,\beta}(\xi)=\frac{\beta}{\beta-2\alpha}\Big(1-\frac{2\alpha \beta}{\xi^2+\beta^2}\Big),$$
so that   $\wh \W_{\alpha,\beta}(0)=1$, and it is simple to check that  \ref{H0} and \ref{H-coer}
			are satisfied. To verify \ref{H-residue}, it is enough to notice that 
			the only singularity on $\H$ of the meromorphic function 
			$\wh \W_{\alpha,\beta}$ is the simple pole $\nu_1=i\beta$ and that 
			$$i\Res(\wh \W_{\alpha,\beta},i\beta)=-\frac{\alpha\beta}{\beta-2\alpha}< 0.$$
			Since $\wh \W_{\alpha,\beta}	$	is bounded on $\H$ away from the pole, we conclude that \ref{H-residue} is fulfilled.
			We recall that, by the Young inequality, $L^1(\R)$ is a subset of $\boM_3(\R)$.
			Therefore $\W_{\alpha,\beta}\in  \boM_3(\R)$ and Theorem~\ref{thm-existence-general} applies.
			\item For $\alpha\in [0,1)$, we take the potential 
			$\W_\alpha=\frac{1}{1-\alpha}(\delta_0-\alpha \boV)$,
			where 
			$$\boV(x)=-\frac{3}{\pi}\ln(1-e^{-\pi\abs{x}}), \quad \text{and}\quad \wh \boV(\xi)=\frac{3(\xi\coth(\xi)-1)}{\xi^2}.$$
			It can be seen that $\wh \boV$ is a smooth even positive function on $\R$, 
			decreasing on $\R^+$, with $\wh \boV(0)=1$ and decaying at infinity as $3/\xi$.
			Thus the conditions  \ref{H0} and \ref{H-coer}
			are satisfied. As a function on the complex plane, $\wh \boV$ is a meromorphic function 
			whose only singularities on $\H$ are given by the simple poles
			$\{i  \pi \ell\}_{\ell\in \N^*}$, and 
			$$i\Res(\wh \W_\alpha,i\pi \ell)=i\Res(-\wh \boV,i\pi \ell)=-\frac{3}{\pi \ell}.$$
			To check \ref{H-residue}, we define for $k\geq 2$,
			the functions 
			$\gamma_{1,k}(t)=(k+1/2)\pi+it$, $t\in~[0,(k+1/2)\pi]$, 
			$\gamma_{2,k}(t)=t+i(k+1/2)\pi$, $t\in~[(k+1/2)\pi,-(k+1/2)\pi]$,
			and 
			$\gamma_{3,k}(t)=-(k+1/2)\pi+it$, $t\in~[(k+1/2)\pi,0]$, 
			so that the corresponding curve $\Gamma_k$ is given by the three sides of a square
			and $\Gamma_k$  does not pass through any poles. 
			Using that for $x,y\in\R$ (see e.g.\ \cite{abram})
			$$\abs{\coth(x+iy)}=\Big|
			\frac{\cosh(2x)+\cos(2y)}{\cosh(2x)-\cos(2y)}\Big|^{1/2},$$
			we can obtain a constant $C>0$, independent of $k$, such that 
			$\abs{\wh \boV (\gamma_{j,k}(t) )}\leq C $, for all $t\in[a_{j,k},b_{j,k}]$, 
			for $j\in\{1,2,3\}$, where $[a_{j,k},b_{j,k}]$ is the domain of definition of $\gamma_{j,k}$.
						As a  conclusion, \ref{H-residue} is fulfilled. 
Since $V\in L^{1}(\R)$, we conclude that $\W_{\alpha}\in  \boM_3(\R)$
and therefore we can apply Theorem~\ref{thm-existence-general} to this potential.
			\item We can also construct perturbations of previous examples. For instance, using the function 
			$\boV$ defined above, we 
						set 
			$$\wh\W_{\sigma,m}(\xi)=\frac{2m^2\pi^2}{m^2\pi^2+2\sigma}
			\Big(
			1-\frac{\wh \boV(\xi)}{2}+\frac{\sigma}{\xi^2+m^2\pi^2}
			\Big),$$
for $\sigma\in \R$ and $m\in \N^*$, so that the poles on $\H$ are still $i\pi\N^*$. It follows that for \mbox{$\sigma>-\pi^2m^2/2$}, the potential 
satisfies  \ref{H0}, and  that \ref{H-residue} holds if 
$\sigma\leq 3$. We can also check that for $\sigma\in(-\pi^2m^2/2,3]$, $\wh\W_{\sigma,m}$ satisfies \ref{H-coer}, and therefore Theorem~\ref{thm-existence-general} applies.
				\end{enumerate}
		%
		%
		%
		%
		%
		%
		%
		%
		%
		In Section~\ref{sec:numerics} we perform some numerical simulations
		to illustrate the shape of the solitons and the minimization curves associated with these and other examples. 
		The rest of the paper is organized as follows: we give some energy estimates  in Section~\ref{sec:estimations}. In Section~\ref{sec:curve},
		we establish the properties of the minimizing curve and the proof of Theorem~\ref{thm:curve}, and
in Section~\ref{sec:suites} we show the compactness of the sequences associated with the minimization problem. The orbital stability of the solutions  and Theorem~\ref{thm:cauchy0}  are proved in Section~\ref{sec:stability}.
		We finally complete the proof of Theorem~\ref{thm-existence-general} in Section~\ref{sec:Euler}.
		\numberwithin{equation}{section}
		\section{Some a priori estimates}
		\label{sec:estimations} 
		We start by establishing an $L^\infty$-estimate for the functions in the energy space in terms 
		of their energy. 
				
				\begin{lem}
			\label{lem-control-energy}
	Assume that $\W\in \boM_2(\R)$ satisfies
		\bq\label{cond:kappa}
	\wh	\W(\xi)\geq (1-\kappa\xi^2)^+, \quad \text{a.e.~on } \R,
	\eq
	for some $\kappa\geq 0$.
			Let $v\in \boEE$ and set $\eta:=1-\abs{v}^2$. Then 
			\bq
			\label{est-eta}
			\norm{\eta}_{L^\infty}^2\leq 8\tilde \kappa E(v)(1+8 \tilde\kappa E(v)+2\sqrt{ 2\tilde\kappa E(v)})
			\eq
			and 
			\bq
			\label{est-eta-L2}
			\norm{\eta}_{L^2}^2\leq 8\tilde \kappa E(v)(1+8\tilde \kappa E(v)+2\sqrt{ 2\tilde \kappa E(v)}),
			\eq
			with $\tilde \kappa=\kappa+1$.
		\end{lem}
		\begin{proof}
			Let $\W\in \boM_2(\R)$ and  
			 $v\in \boEE$, and set $\rho=\abs{v}$, $\eta=1-\rho^2$ and $x\in \R$. By Plancherel's identity
			\bq
			\label{dem-ener}
			\eta^2(x)=2\int_{-\infty}^x \eta \eta' \leq \int_\R (\eta^2 +\eta'^2)=\frac{1}{2\pi}\int_\R(1+\xi^2)\abs{\hat \eta}^2d\xi.
			\eq
			By \eqref{cond:kappa}, we have $1\leq \wh \W(\xi)+\kappa\xi^2$ a.e.~on $\R$,
			so that the term on the right-hand side of \eqref{dem-ener}
			can be bounded by 
			\bq
			\label{dem-ener2}
			\frac{1}{2\pi}\int_\R(1+\xi^2)\abs{\hat \eta}^2\leq \frac{1}{2\pi}\int_\R(\wh \W(\xi)+\tilde \kappa \xi^2)\abs{\hat \eta}^2
			=4E_\p(v)+  \tilde \kappa \int_\R\eta'^2,
			\eq
with $\tilde \kappa=\kappa+1$.
			Now we notice that $\eta'=-2\rho \rho'$, so that $\eta'^2\leq 4 \norm{v}_{L^\infty}^2\rho'^2$.
			Also, if  $\abs{v}\neq 0$ in some open set, then we can write $v=\rho e^{i\theta}$ and 
			$\abs{v'}^2=\rho'^2+\rho^2\theta'^2$. On the other hand, 
			 the set $\tilde\Omega:=\{v=0\}$ coincides with the set $\{\eta=1\}$, and $v'=0$ and $\eta'=0$ a.e.~on $\tilde \Omega$. Therefore, we conclude that 
			\bq
			\label{dem-good-ineq}
			\eta'^2\leq 4 \norm{v}_{L^\infty}^2 \abs{v'}^2 \quad \text{ a.e.~on }\R.
			\eq
			Combining \eqref{dem-ener}, \eqref{dem-ener2} and \eqref{dem-good-ineq}, we have
			\bq
			\label{dem-ener3}
			\eta^2(x)\leq 4E_\p(v)+8\tilde \kappa \norm{v}_{L^\infty}^2E_\k(v)\leq \max(4,8\tilde \kappa \norm{v}_{L^\infty}^2)E(v).
			\eq
			If $\norm{v}^2_{L^\infty}\leq1$, 
			inequality \eqref{est-eta} follows, since $\max(4,8\tilde \kappa)=8\tilde \kappa$. Thus we suppose now that 
			\bq
			\label{dem-ener4}
			\norm{v}^2_{L^\infty}>1.
			\eq
			Bearing in mind that 
			$\eta(\pm \infty)=0$, we deduce that there is some $x_0\in \R$ such that 
			$$
			a:=\min_{\R} \eta=\eta(x_0)=1-\norm{v}^2_{L^\infty}.
			$$
			Therefore, using \eqref{dem-ener3} for $x_0$ and \eqref{dem-ener4}, we get
			\bqq
			a^2\leq 8\tilde \kappa(1-a)E(v).
			\eqq
			Solving the associated quadratic equation and using that $\sqrt{a+b}\leq \sqrt{a}+\sqrt{b}$, we conclude that 
			\bqq
			a\geq \frac12(-8\tilde \kappa  E(v)-\sqrt{
				64\tilde \kappa^2 E(v)^2+32\tilde \kappa   E(v) 
			}
			)
			\geq -8\tilde \kappa E(v)-2\sqrt{ 2\tilde \kappa E(v)},
			\eqq
			which implies that 
			\bq
			\label{dem-ener5}
			\norm{v}^2_{L^\infty}\leq 1+8\tilde \kappa E(v)+2\sqrt{ 2\tilde \kappa E(v)}.
			\eq
			By putting together \eqref{dem-ener3}, \eqref{dem-ener4} and \eqref{dem-ener5}, we obtain \eqref{est-eta}.
			
			To prove \eqref{est-eta-L2}, we use 			the Plancherel identity
			and argue as before to get 
			\bqq
			\int_\R\eta^2\leq 
			\frac{1}{2\pi}\int_\R(\wh \W(\xi)+ \kappa\xi^2)\abs{\hat \eta}^2\leq 4E_\p(v)+\kappa\int_\R\eta'^2
			\leq 
			4E_\p(v)+8\kappa \norm{v}_{L^\infty}^2E_\k(v).
			\eqq
			Therefore, using \eqref{dem-ener5},  inequality \eqref{est-eta-L2} is established.
		\end{proof}
		\begin{rem}\label{rem:estimation}
Let us suppose that $\W\in \boM_2(\R)$ is even and that also $\wh \W$ 
is of class $C^2$ in some interval $[-r,r]$, with $r>0$. 
Then $(\wh \W)'(0)=0$, and by the Taylor theorem we deduce that
for any $\xi\in (-r,r)$, there exists $\tilde \xi\in(-r,r)$ such that 
\bqq
\wh \W(\xi)=1+(\wh \W)''(\tilde \xi)\frac{\xi^2}{2}\geq 1-\mu \xi^2,
\eqq
where $\mu=\max_{[-r,r]} \abs{(\wh \W)''}/2$. 
If $1/\mu\leq r^2$, we set $\kappa=\mu$.
If $1/\mu> r^2$, we take $\kappa=1/r^2$.
Assuming also that 
 $\wh \W\geq 0$ a.e.~on $\R$,
 we conclude that in both cases   condition \eqref{cond:kappa} is fulfilled. 
		\end{rem}
%
		
						From now on until the end of this paper,  we assume that  \ref{H0} and \ref{H-coer} are satisfied, 
		so in particular Lemma~\ref{lem-control-energy} holds true with $\kappa=1/2$. 	
		In the sequel, we also use the identity 
		\begin{equation}
		\label{sym}
		\int_\R (\W*f)g=\int_\R (\W*g)f, \quad\text{ for all }f,g\in L^2(\R),
		\end{equation}
		that is a consequence of parity of $\W$ stated in \ref{H0}.

		A key point to obtain the compactness of the sequences in Section~\ref{sec:suites}
		is that the momentum can be controlled by the energy. This kind of inequality is crucial in the arguments 
		when proving the existence of solitons by variational techniques in the case $\W=\delta_0$ (see \cite{bethuel,chironmaris}).
		Moreover, for an open set  $\Omega\subset \R$ and $u={\rho}e^{i\theta}\in \boEEE$,
		we need to be able to control the localized momentum
		$$p_{\Omega}(u):=\frac{1}{2}\int_{\Omega}\eta\theta',$$
		by some localized version of the energy.
		By the Cauchy inequality, setting as usual $\eta=1-\abs{u}^2$, we have 
		\begin{align}
			\label{mom-young}
			\sqrt{2}|p_{\Om}(u)|&\leq \ \frac{1}{4}\int_{\Om}\eta^2+\frac{1}{2}\int_{\Om}\theta'^2 
			\leq \frac{1}{4}\int_{\Om}\eta^2+\frac{1}{2\underset{\Om}\inf{\rho}^2}\int_{\Om}\rho^2\theta'^2, 
		\end{align}
but it is not clear how to define a localized version of energy, due the to the nonlocal 
interactions. We propose to introduce the localized energy
		$$ E_{\Omega}(u):=\frac{1}{2}\int_{\Omega}|u'|^2+\frac{1}{4}\int_{\Omega}(\W\ast{\eta_\Omega}) {\eta_\Omega}, 
		\quad \text{ with } \eta_{\Omega}:=\eta\mathds{1}_{\Omega}.$$
				Notice that if $\Omega=\R$, then $E_{\Omega}(u)=E(u)$ and $p_{\Omega}(u)=p(u)$.
		Since $\eta_\Omega$ can be discontinuous (and thus not  weakly differentiable)
		when $\Omega$ is bounded, 
		we  also need to introduce a smooth cut-off function as follows: 
		for $\Omega_0$ an open set  compactly contained in $\Omega$, i.e.\
		$\Om_0 \subset\subset \Om$, we set a function $\chi_{\Om,\Om_0}\in C^{\infty}(\R)$   taking values in 
		$[0,1]$  and satisfying
		\bq
		\chi_{\Om,\Om_0}(x)=\begin{cases}
			1 &\text{ if }x\in \Om_0, \\
			0 &\text{ if }x \in \R\setminus\Om. 
		\end{cases}
		\label{def:fonctionreg}
		\eq 
		In the case $\Omega=\Omega_0=\R$, we simply set $\chi_{\Om,\Om_0}\equiv 1$.
		\begin{lem}
	\label{lmm:generalpE}
	Let  $\Omega,\Omega_0\subset \R$ be two smooth open sets with  $\Om_0 \subset\subset \Om$
	and let $\chi_{\Om,\Om_0}\in C^{\infty}(\R)$ as above.
	Let   $u \in \mathcal{E}(\mathbb{R})$  and
	assume that there is some $\ve\in  (0,1)$  such that
	$1-\ve\leq \abs{u}^2\leq 1+ \ve$ on $\Omega$. Then
	\bq
	\sqrt{2}|p_{\Omega}(u)|\leq \frac{E_{\Omega}(u)}{1-\ve}+\Delta_{\Om}(u),
	\label{ctrlEPsuromega}
	\eq
	where the remainder term $\Delta_{\Omega}(u)$ satisfies the estimate
	\begin{align}
	|\Delta_{\Omega}(u)|\leq C(
	\Vert \eta \Vert_{L^{2}(\Om\setminus\Om_0)}+ \Vert \eta\chi_{\Om,\Om_0}' \Vert_{L^{2}(\Om\setminus\Om_0)}
	+\Vert \eta\chi_{\Om,\Om_0}' \Vert_{L^{2}(\Om\setminus\Om_0)}^2
	).
	\label{bornedeltaomega}
	\end{align}Here $C=C(E(u),\ve)$ is a constant depending on $E(u)$ and  $\ve$, but not on $\Omega$ nor $\Om_0$.
	In particular, in the case $\Om=\Om_0=\R,$ we have
	\bq
	|p(u)|\leq \frac{E(u)}{\sqrt{2}(1-\ve)}.
	\label{ctrlEPsurR}
	\eq
\end{lem}
\begin{proof}
	As usual, we write $u=\rho e^{i\theta}$ on $\Omega$.
	As in \eqref{mom-young}, using the Cauchy inequality and that $1-\ve\leq \rho^2\leq 1+\ve_2$ on $\Omega$, we have
	\bq
	\label{dem-mom}
	\sqrt{2}\abs{p_\Omega(u)}\leq \frac{\sigma}{4}\int_{\Om}\eta^2+\frac{1}{2\sigma(1-\ve)}\int_{\Om}\rho^2\theta'^2,
	\eq
	with $\sigma>0$ to be fixed later. Now,  we write
	\begin{align*}
	\frac{\sigma}{4}\int_{\Om}\eta^2+\frac{1}{2\sigma(1-\ve)}\int_{\Om}\rho^2\theta'^2
	=\frac{\sigma}{4}\int_{\Om}\Big(
	\eta_{\Om}^2-{(\W\ast\eta_{\Om})\eta_{\Om}}\Big)+
	{R_{\Om}(u)},
	\end{align*}
	where $$
	R_{\Om}(u):=\frac{\sigma}4\int_{\Om}
	(\W\ast\eta_{\Om})\eta_{\Om}+\frac{1}{2\sigma(1-\ve)}\int_{\Om}\rho^2\theta'^2.
	$$
	Let $\tilde\eta_{\Omega}=\eta\chi_{\Om,\Om_0}$ and 
	\bq
	\label{def-Delta}
	\Delta_{1,\Om}(u):=\frac{\sigma}4\int_{\R}
	\Big(
	\big(
	\eta_{\Omega}^2-\tilde\eta_{\Omega}^2
	\big)
	- (\W\ast\eta_{\Omega})\eta_{\Omega}+(\W\ast\tilde\eta_{\Omega})\tilde\eta_{\Omega} 
	\Big).
	\eq
	Using the  Plancherel theorem and  \ref{H-coer}, we have
	\begin{align*}
	\frac\sigma{4}\int_{\Om}\Big(
	\eta_{\Om}^2-(\W\ast\eta_{\Om})\eta_{\Om}
	\Big)
	&=\frac\sigma{4}\int_{\R}
	\Big(
	\tilde\eta_{\Om}^2-(\W\ast\tilde\eta_{\Om})\tilde\eta_{\Om}\Big)
	+\Delta_{1,\Omega}(u)\\&=
	\frac{\sigma}{8\pi}\int_{\R}|\widehat{\tilde\eta_\Om}|^2\Big( 1-\widehat{\W}(\xi)
	\Big)  +\Delta_{1,\Omega}(u)
	\\&\leq
	\frac{\sigma}{16\pi}\int_{\R}\xi^2|\widehat{\tilde\eta_\Om}|^2
	+\Delta_{1,\Omega}(u)
	\\&=		  \frac{\sigma}{8}\int_{\R}({\tilde\eta_\Om}')^2
	+\Delta_{1,\Omega}(u).
	\end{align*}
	Noticing that 
	$$\tilde\eta_{\Om}'^2=({\eta}'\chi_{\Om,\Om_0})^2+
	2\eta{\eta}'\chi_{\Om,\Om_0}\chi_{\Om,\Om_0}'+({\eta}\chi_{\Om,\Om_0}')^2,$$
	and that $0\leq \chi_{\Om,\Om_0}\leq 1$, by putting together the estimates above, 			we conclude that 
	\bqq
	\sqrt{2}\abs{p_\Omega(u)}\leq
	\frac\sigma{8}\int_{\Omega}\eta'^2+{R_{\Omega}(u)}+\Delta_{\Omega}(u),
	\eqq							
	where the remainder term is given by
	\bqq
	\Delta_{\Omega}(u)=\Delta_{1,\Omega}(u)+\Delta_{2,\Omega}(u), \quad 
	\Delta_{2,\Omega}(u):=\frac{\sigma}8\int_{\Omega}
	\big( 2\eta{\eta}'\chi_{\Om,\Om_0}\chi_{\Om,\Om_0}'+({\eta}\chi_{\Om,\Om_0}')^2\big).
	\eqq
	Therefore, since $\eta'^2\leq 4(1+\ve)\rho'^2$,  taking $\sigma=1/\sqrt{1-\ve^2}$, we obtain
	$$\sqrt{2}\abs{p_\Omega(u)}\leq 
	\frac{\sqrt{1-\ve^2}}{1-\ve}\int_{\Omega}\Big(
	\frac{\rho'^2}{2}+
	\frac{\rho^2\theta'^2}{2}
	\Big)
	+	\frac1{4\sqrt{1-\ve^2}}
	\int_{\Om}(\W\ast\eta_{\Om})\eta_{\Om}
	+					\Delta_{\Omega}(u),$$
	which gives us \eqref{ctrlEPsuromega}.
	It remains to show the estimate in \eqref{bornedeltaomega}.
	For the first term in $\Delta_{1,\Om}$, we see that
	\bq
	\int_{\Omega}\left|\eta_{\Omega}^2-\tilde\eta_{\Omega}^2\right|=
	\int_{\Omega\setminus\Om_0}\eta^2\left|\mathds{1}_{\Om}^2-\chi_{\Om,\Om_0}^2\right|
	\leq \Vert \eta \Vert_{L^{2}(\Om\setminus\Om_0)}^2.
	\label{1stbornedelta}
	\eq
	For the other term in $\Delta_{1,\Om}$, using \eqref{sym},  we have
	\begin{align}
	\left| \int_{\R}(\W\ast\eta_{\Omega})\eta_{\Omega}-(\W\ast\tilde\eta_{\Omega})\tilde\eta_{\Omega} \right|
	&	=
	\left|
	\int_{\R} (\W\ast(\eta_{\Omega}+\tilde\eta_{\Omega}))( \eta_{\Omega}-\tilde\eta_{\Omega} )
	\right|
	\nonumber	\\&	\leq 4\Vert \W \Vert_{\boM_2} \Vert \eta \Vert_{L^{2}(\mathbb{R})}\Vert \eta \Vert_{L^{2}(\Om\setminus\Om_0)}.
	\label{2ndbornedelta}
	\end{align}
	Concerning in $\Delta_{2,\Om}$, we have
	\begin{equation}
	\label{3rdbornedelta}
	\abs{\Delta_{2,\Om}}
	\leq\frac{\sigma}{8}
	\big(
	4\Vert u \Vert_{L^{\infty}(\R)}\Vert u'\Vert_{L^{2}(\R)}\Vert \eta\chi_{\Om,\Om_0}'\Vert_{L^{2}(\Om\setminus\Om_0)} +
	\Vert \eta\chi_{\Om,\Om_0}'\Vert_{L^{2}(\Om\setminus\Om_0)}^2
	\big).
	\end{equation}
	By putting together \eqref{1stbornedelta}, \eqref{2ndbornedelta} and \eqref{3rdbornedelta}, and invoking Lemma~\ref{lem-control-energy}, we obtain \eqref{bornedeltaomega}.
\end{proof}
			From now on, we set for $\q>0,$
		\bq
		\Sigma_\q:=1-\frac{E_{\min}(\q)}{\sqrt{2}\q}.
		\label{defsigmaq}
		\eq
		In this manner, the condition $E_{\min}(\q)<\sqrt{2}\q$ is equivalent to $\Sigma_\q>0$. We also define for $\q>0$ and $\delta>0$, the set
		\bq
		X_{\q,\delta}:=\{v \in \boEEE : |p(v)-\q|\leq \delta~\textup{and}~|E(v)-E_{\min}(\q)|\leq \delta \}.
		\eq
		\begin{lem}
			\label{lmm:evanes2}
			Let  $\q>0$, $L>1$ and suppose that   $\Sigma_\gq>0$.
			Then there is  $\delta_0>0$ such that for all  $\delta \in [0,\delta_0]$ and for all  $v \in X_{q,\delta}$,
			there exists $\bar{x}\in \mathbb{R}$ such that
			\bqq
			\left| 1-|v(\bar{x})|^2\right|\geq \frac{\Sigma_\gq}{L}.
			\eqq
		\end{lem}
		\begin{proof}
			We argue by contradiction and  suppose that the statement is false. Hence, for all $\delta_0>0,$ there exists $\delta\in[0,\delta_0]$ and $v \in X_{\q,\delta}$ such that
			\bqq
			\Vert 1-|v|^2 \Vert_{L^{\infty}(\mathbb{R})}<\Sigma_\gq/L.
			\eqq
			Then, taking $\delta_0=1/n,$ there is  $\delta_n \in [0,\frac{1}{n}]$ and $v_n \in X_{\q,\delta_n}$ such that
			\bqq
			\Vert 1-|v_n|^2 \Vert_{L^{\infty}(\mathbb{R})}<\Sigma_\gq/L.
			\eqq
			Since $\Sigma_\gq\in(0,1]$, considering $\ve=\Sigma_\gq/L$, we have $\ve\in (0,1)$. Therefore
			we can apply Lemma~\ref{lmm:generalpE} to conclude that
			\bqq
			\sqrt{2}|p(v_n)|\leq \frac{1}{(1-{\Sigma_\gq}/{L})}E(v_n),
			\label{eq:evanesabsurde}
			\eqq
			and letting $n \to \infty$, we get
			$$ {\sqrt{2}\q\Big(1-\frac{\Sigma_\gq}{L}\Big)}\leq E_{\min}(\q),$$
			which is equivalent to 	$ \Sigma_\gq\leq {\Sigma_\gq}/{L},$	contradicting the fact that $L>1$.
		\end{proof}
		\begin{lem}
			\label{lemmabgs}
			Let $E>0$ and $0<m_0<1$ be two constants. There is $l_{0}\in \N$, depending on $E$ and $m_{0}$, such that 
			for any function $v \in \mathcal{E}(\mathbb{R})$ satisfying $E(v)\leq E$, one of the following
			holds: 
			\begin{enumerate}
				\item\label{cas1} For all $x\in \R$, $|1-|v(x)|^2|<m_{0}$.
				\item\label{cas2} There exist $l$ points $x_1,x_2,\dots,x_l$, with $l\leq l_0$, such that
				$$|1-|v(x_j)|^2|\geq m_{0},\ \forall 1\leq j\leq l,\quad \text{ and }\quad |1-|v(x)|^2|\leq m_{0},\ \forall x \in \mathbb{R}\setminus \bigcup_{j=1}^{l}[x_j-1,x_j+1].$$
			\end{enumerate}
		\end{lem}
		\begin{proof}
The proof is a rather standard consequence of the energy estimates. 
For the sake of completeness, we give a proof similar to the one given 
 in~\cite{bethuel2008existence}.
 
			Let us suppose that \ref{cas1} does not hold.  Then 
			the set  $$ \mathcal{C}=\{z \in \mathbb{R} : |\eta(z)|\geq m_{0}\},$$ 
			is nonempty, where $\eta=1-\abs{v}^2$ as usual. Setting $I_j=[j-1/2,j+1/2]$, for $j\in\Z$, 
			the assertion in \ref{cas2} will follow if we show 
			that 
			 $l:=\text{Card}\{j \in \mathbb{Z},~I_j\cap \mathcal{C}\neq \emptyset \}$
			 can be bounded by some $l_0$, depending only on $E$ and $m_{0}$.
			
Using that $\abs{\abs{v}'}=\abs{v'}$ (see Lemma 7.6 in \cite{gilbarg}), the Cauchy--Schwarz inequality 
and \eqref{est-eta}, we deduce that there exists a constant $C$, 
			depending on $E$, such that for all $x,y\in \R$,
			\begin{align*}
			\abs{|v(x)|^2-|v(y)|^2}=2\left|\int_{x}^{y}|v|'|v|\right|\leq 2\Vert v \Vert_{L^{\infty}(\mathbb{R})}\Vert v' \Vert_{L^{2}(\mathbb{R})}|x-y|^{\frac{1}{2}}\leq C|x-y|^{1/2}.
			\end{align*}
			Thus, setting  $r={m_{0}^2}/(4C^2)$, we deduce that for any 		
			 $z \in \mathcal{C}$ and for any $y \in [z-r,z+r]$, 
			 \begin{align*}
			|\eta(y)|  \geq  m_0 -||v(y)|^2-|v(z)|^2|  \geq \frac{m_0}{2}.
			\end{align*}
			Taking $r_0=\min( r,1/2)$ and integrating this inequality, we get,
			for any 		
			$z \in \mathcal{C}$, 
			\begin{align*}
			\int_{z-r_{0}}^{z+r_0}\eta^2(y)dy\geq \frac{{m_{0}^2}r_0}{4}.
			\end{align*} 
			Noticing that $[z-r_0,z+r_0]\subset \tilde{I_{j}}:=[j-1,j+1]$,  
			if $z \in I_{j}\cap{\mathcal{C}}$, we conclude that 
			
			$$ \frac{\tilde l {m_{0}^2}r_0}{4}\leq \sum_{j \in \mathbb{\Z}, \tilde I_j\cap \boC\neq \emptyset }\int_{\tilde{I_j}}\eta^2
			\leq 2\norm{\eta}_{L^2(\R)}^2,$$
			where  $\tilde l:=\text{Card}\{j \in \mathbb{Z} : \tilde I_j\cap \mathcal{C}\neq \emptyset \}$.
			The conclusion follows from \eqref{est-eta-L2}, taking $l_0=2\tilde l$, since 
			$l \leq 2\tilde l$.
					\end{proof}
		
		\section{Properties of the minimizing curve}
		\label{sec:curve}
		For the study of the minimizing curve, it will be useful to use finite energy smooth functions  that are constant far away from the origin. For this purpose we introduce the set
		$$\boE^\infty_0(\R)=\{v\in \boEEE\cap C^\infty(\R) : \exists R>0\text{ s.t.\ }  v \text{ is constant on }B(0,R)^c\}.$$
		Notice that in the functions in the space $\boE^\infty_0(\R)$ can have different values near  $+\infty$ and near $-\infty$. 
		Bearing in mind that the solitons $u_c$ in \eqref{sol:1D} satisfy
		$u_c(+\infty)\neq u_c(-\infty)$,
		we will see that these   kinds of functions are
		 well-adapted to approximate the solutions of \eqref{ntwc}.

		The next result shows that $\Emin$ is well defined and that its graph lies under the line $y=\sqrt{2}x$ on $\R^+$.
		\begin{lem}
			\label{lemmeEminborne}
			For all  $\gq\in\R$, there exists a sequence  
			$v_{n}\in\boE^\infty_0(\mathbb{R})$ satisfying
			\bq
			p(v_n)=\gq \quad \text{ and }\quad E(v_n)\rightarrow \sqrt{2}\abs{\gq}, \quad \text{ as }n\to \infty.
			\label{inegtangente}
			\eq 
			In particular the function $\Emin :\R\to\R$ is well defined, and for all $\gq\geq 0$  
			\bq
			E_{\min}(\gq)\leq \sqrt{2}\gq.
			\label{droiteaudessus}
			\eq 
		\end{lem}
		\begin{proof}
			The case $\gq=0$ is trivial since it is enough to take $v\equiv 1$. Let us assume that $\gq>0$
			and consider   $\chi \in \ccc$ such that $\int_{\mathbb{R}}\chi'^2={\q}/{\sqrt{2}}$.
			Let us define 
			\bqq
			 a=\frac{\sqrt{2}}{2}\int_{\mathbb{R}}\chi'(y)^3dy,\quad
			 \alpha_{n}=\frac{1}{n} \quad
			 \text{ and }\quad
			\beta_{n}=\frac{1}{n^2}-\frac{a}{\q n^3}.
			\eqq
			Then it is enough to consider 
			\bqq
			v_n=\rho_n e^{i \theta_n}, \quad \text{ where }\rho_{n}(x)=1-\alpha_n\chi'(\beta_n x) \text{ and } \theta_{n}(x)=\sqrt{2}\frac{\alpha_n}{\beta_n}\chi(\beta_n x).
			\eqq
			We can assume that $v_n$ does not vanish since $|v_n|=|\rho_{n}|\geq 1-|\alpha_{n}|\Vert \chi' \Vert_{L^{\infty}(\mathbb{R})}.$
			Thus the momentum of $v_n$ is well defined and we have
			\begin{align*}
				p(v_n)&=\frac{1}{2}\int_{\mathbb{R}}(1-\rho_n^2)\theta_n'=\frac{\sqrt{2}}{2\beta_n}\int_{\mathbb{R}}(2\alpha_n\chi'(y)-\alpha_n^2\chi'(y)^2)\alpha_n\chi'(y)dy\\
				&=\frac{\alpha_{n}^2}{\beta_n}\q-\frac{\alpha_{n}^3}{\beta_n}a=\gq.
			\end{align*}
			It remains to show that $E(v_n)\to\sqrt{2}\q$. For the kinetic part, we have
			\begin{align*}
				E_{\k}(v_n)&= \int_{\mathbb{R}}(1-\alpha_n\chi'(\beta_n x))\alpha_n\chi'(\beta_n x)^2 dx +
				\frac{1}{2}\int_{\mathbb{R}}(\alpha_n\beta_n \chi''(\beta_n x))^2dx\\
				&=\frac{\alpha_n^2}{\beta_n}\int_{\mathbb{R}}(1-\alpha_n\chi'(y))\chi'(y)^2 dy+
				\frac{\alpha_n^2 \beta_n}{2}\int_{\mathbb{R}}\chi''(y)^2dy\\
				&\to \int_{\mathbb{R}}\chi'(y)^2dy=\frac{\gq}{\sqrt2},
			\end{align*}
			since $\alpha_n,\beta_n\to0$ and $\alpha_n^2/\beta_n\to1$.
			For the potential energy, using Plancherel's theorem, the dominated convergence theorem and the continuity of $\widehat{\W}$ at $0$, we get
			\begin{align*}
				E_{\p}(v_n)&= \frac{1}{8\pi}\int_{\mathbb{R}}\widehat{\W}(\xi)|\mathcal{F}(1-\rho_n^2)|^{2}(\xi)d\xi
				= \frac{\alpha_n^2}{8\beta_n\pi}\int_{\mathbb{R}}\widehat{\W}(\beta_n\xi)|\mathcal{F}(2\chi'-\alpha_n\chi'^2)|^{2}\left({\xi}\right)d\xi\\
				&\to  \int_{\mathbb{R}}\chi'^2(y)dy=\frac{\gq}{\sqrt{2}}.
			\end{align*}
			Therefore we conclude that \eqref{inegtangente} holds true for $\gq\geq 0$. In the situation $\gq<0$,
			it is enough to proceed as above taking 
			\bqq
			\int_{\mathbb{R}}\chi'^2=\frac{\abs{\q}}{\sqrt{2}}=-\frac{{\q}}{\sqrt{2}} \quad \text{ and}\quad v_n=\rho_n e^{-i \theta_n}.
			\eqq
			This concludes the proof of \eqref{inegtangente}. By the definition of $\Emin$, we also have $\Emin(\q)\leq E(v_n).$
			Letting $n\to\infty$, we obtain \eqref{droiteaudessus}.
		\end{proof}
		\begin{lmm}
			\label{pariteetsousadd}
			The curve $E_{\min}$ is even on $\R$.
		\end{lmm}
		\begin{proof}
			Let $\gq\in \R$ and  $u_n=\rho_n e^{i\phi_n} \in \boEEE$ be such that 
			$ E(u_n)\to E_{\min}(\q)$ and $ p(u_n)=\q.$
			Setting ${v_n}=\rho_n e^{-i\phi_n}$, it is immediate to verify that 
			$E(v_n)=E(u_n)$ and that $p(v_n)=-p(u_n)=-\q$. Therefore 
			\bqq E(v_n)\geq E_{\min}(-\q),
			\eqq
			and letting $n\to\infty$ we conclude that $E_{\min}(\q)\geq E_{\min}(-\q).$
			Replacing $\q$ by $-\gq$, we deduce that $E_{\min}(-\q)= E_{\min}(\q)$, i.e.\ that $\Emin$ is even. 
		\end{proof}
		\begin{cor}\label{cor:q*}
			The constant defined in \eqref{q_*} satisfies $\q_*>0.027$.
		\end{cor}
				\begin{proof}
					Let $v\in \boEE$, with $E(v)\leq\Emin(\q)$. Then, 
					by combining \eqref{est-eta} and \eqref{droiteaudessus}, with $\tilde \kappa=3/2$,
					we have
														\bqq
		\norm{1-\abs{v}^2}_{L^\infty}^2\leq 12\sqrt{2}\q (1+12\sqrt{2}\q+2( 3\sqrt{2}\q)^\frac12).
		\eqq
		Since the right-hand is an increasing function of $\q$, and  
since the solution of the equation $12\sqrt{2}z (1+12\sqrt{2}z+2( 3\sqrt{2}z)^\frac12)=1$ is 
$$
z=\frac{\sqrt 2}{288}\frac{ \big(    (12\sqrt 3+4\sqrt{31})^{2/3}-4   \big)^2 }{(12\sqrt 3+4\sqrt{31})^{2/3}}\approx 0.0274,
$$
the conclusion follows from the definition of $\q_*$.
		\end{proof}
		In view of Lemma~\ref{pariteetsousadd}, it is enough to study $\Emin$ on $\R^+$.
		Concerning the density of the space $\boE_0^\infty(\R)$ in $\boEEE$, we have the following result.
		\begin{lem}
			\label{lem-density}
			Let $v=\rho e^{i\theta}\in \boEEE$. Then there exists a sequence  functions 
			$v_n=\rho_n  e^{i\theta}$ in $\boE^\infty_0(\R)$,
			with $\rho_n-1, \theta_n'\in C_c^\infty(\R)$, such that 
			\bq\label{dense1}
			\norm{\rho_n-\rho}_{H^1(\R)}+\norm{\theta'_n-\theta'}_{L^2(\R)}\to 0, \quad \text{ as }n\to \infty.
			\eq
			In particular
			\bq
			\label{dense2}
			E(v_n)\to E(v) \quad\text{ and } \quad p(v_n)\to p(v),  \quad \text{ as }n\to \infty.
			\eq
		\end{lem}
		\begin{proof}
			Since $v=\rho e^{i\theta}\in \boEEE$, we deduce that  $v\in L^\infty(\R)$ and that $\abs{v(x)}\to 1$, as $\abs{x}\to \infty$.
			Let $$g(x):=\rho(x)-1=\abs{v(x)}-1=\frac{\abs{v(x)}^2-1}{\abs{v(x)}+1}.$$
			Then $g\in L^2(\R)$ and since $g'=\langle v',v\rangle/\abs{v}$,
			we conclude that $g\in H^1(\R)$. Therefore, there exists $g_n\in C_c^\infty(\R)$ such that 
			$g_n \to g$ in $H^1(\R)$. Setting $\rho_n=g_n+1$, we deduce that  $\norm{\rho_n-\rho}_{H^1}\to 0$, as $n\to \infty.$
			
			Concerning $\theta$, using the   density of $C_c^\infty(\R)$ in $L^2(\R)$, we get the existence of a sequence $\phi_n\in C_c^\infty(\R)$
			converging to $\theta'$ in $L^2(\R)$. Hence, taking 
			\bq
			\label{def-theta}
			\theta_n(x)=\int_{-\infty}^x \phi_{n},
			\eq
			we conclude that $\theta_n'-\theta'\to0$ in $L^2(\R)$ and that $v_n:=\rho_n  e^{i\theta_n}$ belongs to $\boE^\infty_0(\R)$.
			The convergences in  \eqref{dense2} are a direct consequence of the convergences in \eqref{dense1} and the Sobolev injection 
			$H^1(\R)\hookrightarrow L^\infty(\R)$.
		\end{proof}
		\begin{rem}
			\label{rem:theta}
			If $v\in \boE^\infty_0(\R)$, then we can write $v=\rho e^{i\theta}$, with $\rho,\theta\in C^\infty(\R)$
			and such that $\rho-1,\theta'\in C_c^\infty(\R)$. Hence the function $\theta$ is constant outside $\supp(\theta')$
			and without loss of generality we can assume that there is $R>0$ such that $\theta(x)\equiv 0$ for all  $x\leq -R$, 
			or  that $\theta(x)\equiv 0$ for all  $x\geq R$ (but we cannot assume  that $\theta(x)\equiv 0$ for all  $\abs{x}\geq R$).
			Therefore, w.l.o.g.\  we can suppose that $v(x)\equiv 1$ for all  $x\leq -R$   or  that $v(x)\equiv 1$ for all  $x\geq R$,
			for some $R>0$ large enough.
		\end{rem}
		To handle the nonlocal interaction term in the  energy in the construction 
		of comparison sequences, we use introduce the functional
		$$B(f):=\int_\R(\W*f)f,$$ 
		for $f\in  L^2(\R;\R).$
		It is clear that if $u\in \boE(\R)$, then $B(1-\abs{u}^2)=4E_\p(u).$
		The following elementary lemma will be useful.
		\begin{lem}
			\label{lem-decompW}
			For all  $f,g\in L^2(\R)$ we have
			\bq
			\label{Wdecomp}
			B(f+g)=B(f)+B(g)+2 \int_{\R}(\W*f)g.
			\eq
			Assume further that $g\in C_c^{\infty}(\R)$ and that there is a sequence of numbers $(y_n)$
			such that $y_n\to \infty$, as $n\to \infty$. 
			Then, setting 	set $g_n(x)=g(x-y_n)$, we have
			\bq
			\label{Wconverge}
			B(f+g_n)-B(f)-B(g_n)=2 \int_{\R}(\W*f)g_n\to0, \quad \text{ as }n\to \infty.
			\eq
		\end{lem}
		\begin{proof}
			The identity \eqref{Wdecomp} is a direct consequence of \eqref{sym}.
			The convergence in  \eqref{Wconverge} follows from the fact that 
			$g_n\wto 0$ in $L^2(\R)$.
		\end{proof}
		We finally conclude that we can modify a function with energy close to $\Emin(\gq)$
		such that it is constant far away, but the momentum remains unchanged.
		\begin{cor}
			\label{cor:suppcte}
			Let $u=\rho e^{i\theta}\in \boEEE$. There exists a sequence
			$u_n \in  \boE^\infty_0(\R)$ such that 
			\bq
			\label{cor-conv-E-p}
			p(u_n) = p(u)\quad\text{ and } \quad E(u_n) \to E(u),  \quad \text{ as }n\to  \infty.
			\eq
		\end{cor}
		\begin{proof}
			Let $v_n=\rho_n  e^{i\theta} \in  \boE^\infty_0(\R)$ be the sequence given by Lemma~\ref{lem-density} 
			such that 
			\bq
			\label{proof-dense}
			E(v_n)\to E(u) \quad\text{ and } \quad p(v_n)\to p(u),  \quad \text{ as }n\to \infty.
			\eq
			If $p(u)\neq 0$, we set $\alpha_n=p(u)/p(v_n)$. Therefore $\alpha_n\to 1$ and it is straightforward to verify that the sequence 
			$u_n=\rho_n  e^{i \alpha_n\theta_n}$ satisfies \eqref{cor-conv-E-p}.
			
			The case $p(u)=0$ is more involved. 
			In this instance, we may assume that $\delta_n:=p(v_n) \neq 0$ for $n$ sufficiently large. Otherwise, up to a subsequence, the conclusion holds with $u_n =v_n$. 
			By Lemma~\ref{lemmeEminborne}, we get the existence of a sequence $w_n\in \boE_0^\infty(\R)$ such that 
			\bq
			\label{proof-pn}
			p(w_n)=-\delta_n \quad \text{ and }\quad E(w_n)\rightarrow 0, \quad \text{ as }n\to \infty.
			\eq 
			Let $R_n,r_n>0$ be such that the functions 
			$$
			f_{n}:=1-\abs{v_n}^2 \quad \text{ and }\quad g_n:=1-\abs{w_n}^2
			$$
			are supported in the balls $B(0,R_n)$ and $B(0,r_n)$, respectively. 
			Taking into account Remark~\ref{rem:theta}, without loss of generality, we can assume that the following function is continuous and belongs to $\boE^\infty_0(\R)$
			\bq\label{def-u_n}
			u_n =
			\begin{cases}
				v_n, & \text{ on } \ (-\infty,R_n),\\
				1, & \text{ on }  [R_n, -r_n+y_n],\\
				w_{n}(\cdot - y_n), & \text{ on } (-r_n+y_n,\infty),
			\end{cases}
			\eq
			where $y_n$ is a sequence of points such that $R_n< -r_n+y_n$. For simplicity, we set $\tilde w_{n}=w_{n}(\cdot - y_n)$ and 
			$\tilde g_n:=1-\abs{\tilde w_n}^2$.
			It follows that 
			\bq
			\label{proof-dense2}
			p(u_n)=p(v_n)+p(\tilde w_n)=0 \quad \text{ and } \quad E_\k(u_n)= E_\k(v_n)+ E_\k(w_n).
			\eq
			In particular, combining with \eqref{proof-dense} and \eqref{proof-pn}, we infer that 
			$E_\k(u_n)\to E_\k(u)$. In addition, $1-\abs{u_n}^2= f_n+\tilde g_n$, so that \eqref{Wdecomp} leads to
			\begin{align*}
				E_\p(u_n)&=\frac{1}4B(f_n)+\frac{1}4B(g_n)+\frac{1}2 \int_{\R}(\W*f_n)\tilde g_n
				=E_\p(v_n)+E_\p( w_n)+\frac{1}2 \int_{\R}(\W*f_n)\tilde g_n.
			\end{align*}
			Therefore
			\bq
			\label{proof-estimate}
			\abs{E_\p(u_n)-E_\p(v_n)}\leq E_\p( w_n)+\norm{\W}_{\boM_2}\norm{f_n}_{L^2}\norm{g_n}_{L^2}.
			\eq
			Using the estimate \eqref{est-eta-L2}, \eqref{proof-dense} and \eqref{proof-pn}, we conclude 
			that $\norm{f_n}_{L^2}$ is bounded and that  $\norm{g_n}_{L^2}\to~0$, so that 
			$E_\p(u_n)\to E_\p(u)$, which completes the proof of the corollary.
		\end{proof}
		\begin{cor}
			\label{cor:supp2}
			For all $\gq \geq 0$ and $\ve>0$, there is $v\in \boE_0^\infty(\R)$ such that 
			$$p(v)=\gq \quad \text{ and }\quad E(v)<\Emin(\gq)+\ve.$$
			In particular
			\bqq
			\Emin(\gq)=\inf\{E(v) : v\in \boE_0^\infty(\R),\  p(v)=\gq \}.
			\eqq
		\end{cor}
		\begin{proof}
			Let $\gq\geq 0$ and $\ve>0$. By definition of $\Emin$, there is a sequence $v_m\in \boEEE$ such that 
			$p(v_m)=\gq$ and $E(v_m)\to \Emin(\gq)$, as $m\to \infty$.
			Hence there is $m_0$ such that 
			\bq
			\label{proof-cor-emin}
			E(v_{m_0})<\Emin(\gq)+{\ve}/{2}.
			\eq
			By Corollary~\ref{cor:suppcte}, we deduce the existence of $v\in \boE_0^\infty(\R)$ such that 
			$p(v)=p(v_{m_0})=\gq$ and $\abs{E(v_{m_0})-E(v)}\leq \ve/2$. 
			Combining with \eqref{proof-cor-emin}, the conclusion follows.
		\end{proof}

		\begin{prop} 
			\label{Lipschitz}
			$E_{\min}$ is continuous  and 
			\begin{equation}
				\label{E:lipschitz}
				|E_{\min}(\gp) - E_{\min}(\gq)| \leq \sqrt{2}|\gp - \gq|, \quad\text{ for all } \gp, \gq\in\R.
			\end{equation}
		\end{prop}
		\begin{proof} 
			We assume without loss of generality that $\gq \geq \gp\geq 0$. It is enough to show that
			\begin{equation}
				\label{dem:lip}
				E_{\min}(\gq) \leq E_{\min}(\gp) +  \sqrt{2}(\gq - \gp).
			\end{equation}
			Let $\delta > 0$. By Corollary~\ref{cor:supp2} and Remark~\ref{rem:theta}, 
			there is $v_\delta\in \boE_0^\infty(\R)$ such that
			for some $R_\delta>0$, the function  $1-\abs{v_\delta}^2$ is supported on $B(0,R_\delta)$, $v_\delta=1$ on $[R_\delta,\infty)$,
			\bq\label{lip-1}
			p(v_\delta) = \gp \quad{\rm and } \quad E (v_\delta) \leq \Emin(\gp) + {\delta}/{3}.
			\eq
			Now, setting $\gs = \gq - \gp$ and invoking Lemma~\ref{lemmeEminborne}, we deduce that there is  $w_\delta\in\boE_0^\infty(\R)$ such that 
			for some $r_\delta>0$, $1-\abs{w_\delta}^2$ is supported on $B(0,r_\delta)$,  $w_\delta=1$ on $(-\infty, r_\delta]$,
			\bq\label{lip-2}
			p(w_\delta) = \gs \quad\text{ and }\quad E(w_\delta) \leq  \sqrt{2}\gs+ {\delta}/{3}.
			\eq 
			Let $f_\delta=1-\abs{v_\delta}^2$ and $g_\delta=1-\abs{w_\delta}^2$.  
			Then $f_\delta$ and $g_\delta$ have compact supports and applying  Lemma~\ref{lem-decompW} we can choose $y_\delta \in \R$, 
			large enough, such that their  supports do not intersect.
			Finally, we infer that the function 
			\bq\label{def-v}
			u_\delta =
			\begin{cases}
				v_\delta, & \text{ on } \ (-\infty,R_\delta),\\
				1, & \text{ on }  [R_\delta, -r_\delta+y_\delta],\\
				w_{\delta}(\cdot - y_\delta), & \text{ on } (-r_\delta+y_\delta,\infty),
			\end{cases}
			\eq
			satisfies 
			\bq\label{dem:p-E}
			p(u_\delta)=p(v_\delta)+p(w_{\delta}(\cdot - y_\delta))=\gq \quad \text{ and }\quad
			E_\k(u_\delta)= E_\k(v_\delta)+E_\k(w_{\delta}).
			\eq
			Moreover, since
			$$1-\abs{u_\delta}^2=f_\delta+g_\delta(\cdot-y_\delta),$$
			applying Lemma~\ref{lem-decompW} and increasing $y_\delta$ if necessary, we conclude that
			\bq
			\label{dem:Ep}
			E_\p(u_\delta)\leq E_\p(v_\delta)+E_\p(w_{\delta})+{\delta}/{3}.
			\eq
			Therefore, combining \eqref{lip-1}, \eqref{lip-2}, \eqref{dem:p-E} and \eqref{dem:Ep}, we get
			$$\Emin(\gq)\leq E(u_\delta)\leq \Emin(\gp) +  \sqrt{2}(\gq-\gp)+ \delta. $$
			Letting $\delta\to 0$, we obtain \eqref{dem:lip}.
		\end{proof}
		As noticed by Lions~\cite{lions84}, the  properties established above  are usually sufficient to check   that the minimizing curve is subadditive, as stated in the following result.
		\begin{lem}
			\label{lem:subadd}
			$E_{\min}$ is subadditive  on $\R_+$, i.e.
			\begin{equation}
				\label{E:sub}
				E_{\min}(\gp+\gq)\leq E_{\min}(\gp)+E_{\min}(\gq), \quad\text{ for all } \gp, \gq\geq 0.
			\end{equation}
		\end{lem}
		\begin{proof}
			Let $\gp, \gq\geq 0$ and $\delta>0$. By using Corollary~\ref{cor:supp2} and arguing as in the proof of Proposition~\ref{Lipschitz}, 
			we get the existence of $v,w\in \boE_0^\infty(\R)$ such that 
			\bqq
			p(v) =\gp, \quad p(w) =\gq,  \quad E(v) \leq  \Emin(\gp)+ {\delta}/{3}\quad \text{ and } 
			E(w) \leq  \Emin(\gq)+ {\delta}/{3},
			\eqq 
			with  $v$ and $w$ constant on $B(0,R)^c$ and $B(0,r)^c$, respectively, for some $R,r>0$. 
			As in previous proofs, we define
			\bqq
			u =
			\begin{cases}
				v, & \text{ on } \ (-\infty,R),\\
				1, & \text{ on }  [R, -r+y],\\
				w(\cdot - y), & \text{ on } (-r+y,\infty),
			\end{cases}
			\eqq
			with $y$ large enough  such that 
			\bqq
			E_\p(u)\leq E_\p(v)+E_\p(w)+{\delta}/{3}.
			\eqq
			Since $E_\k(u)=E_\k(v)+E_\k(w)$ and $p(u)=p(v)+p(w)=\gp+\q$, we conclude that 
			$$\Emin(\gp+\q)\leq E(u)\leq E(v)+E(w)+\frac{\delta}{3}\leq  \Emin(\gp)+\Emin(\gq)+\delta.$$
			Letting $\delta\to 0$, inequality \eqref{E:sub} is established.
		\end{proof}
		
		In some minimization problems, there is some kind of homogeneity in the functionals 
		that allows to obtain the strict subadditive property. In our case, the homogeneity give us 
		only the monotonicity of the curve.
		\begin{lem}
			\label{lem:increasing}
			$\Emin$ is nondecreasing on $\R^+$.
		\end{lem}
		\begin{proof}
			Let $0<\gp<\gq$ and $\lambda=\gp/\gq\in(0,1)$.
			As in previous proofs, for $\delta>0$ we take $v=\rho e^{i\theta}$ in $\boEEE$
			such that $E(v)<\Emin(\q)+\delta$ and $p(v)=\q.$ Then we verify that 
			the function $v_\lambda=\rho e^{i\lambda\theta}$ satisfies $p(v_\lambda)=\lambda \q$
			and $E(v_\lambda)\leq E(v)$. Therefore
			$$\Emin(\lambda \q)\leq E(v_\lambda)\leq E(v)<\Emin(\q)+\delta,$$
			so that the conclusion follows letting  $\delta\to 0$.
		\end{proof}
		
Hypothesis \ref{H-residue-bis} provides
a sufficient condition to  ensure the concavity of the function $\Emin$.
As mentioned in the introduction,   the proof  relies  some identities developed  by Lopes and Mari\c s in~\cite{lopes-maris}. 
		\begin{prop}
			\label{prop-concave}
			Assume that \ref{H-residue-bis} holds. Then for all $\gp,\gq\geq 0$, 
			\bq\label{ineq-conc}
			\frac{E_{\min}(\gp) + E_{\min}(\gq)}{2} \leq  E_{\min} \Big( \frac{\gp + \gq}{2} \Big).
			\eq
			In particular $\Emin$ is concave  on $\R^+$.
		\end{prop}
	\begin{proof}
		Let $\gp,\gq>0$ and $\delta>0$. By Corollary~\ref{cor:supp2}, there is $u=\rho e^{i\theta}\in \boE_0^\infty(\R)$ such that
		\bq\label{sum-p}
		p(u) = \frac{\gp + \gq}{2} \quad {\rm and} \quad E(u) \leq E_{\min} \Big( \frac{\gp + \gq}{2} \Big) + \frac{\delta}{2}.
		\eq
		By the  dominated convergence theorem, it follows that the map $G : \R\to\R$
		given by 
		$$G(a):=\frac12\int_a^\infty (1-\rho^2)\theta'$$
		is continuous, with $\lim_{a\to\infty}G(a)=0$ and $\lim_{a\to-\infty}G(a)=p(u)=(\gp+\gq)/2$.
		Hence, by the mean value theorem, there is $a_0$ such that $G(a_0)=\gp/2$. Thus the translation
		$\tilde u(x):=\tilde \rho(x) e^{i\tilde \theta(x)}=\rho(x-a_0) e^{i\theta(x-a_0)}$ satisfies
		\bq\label{sum-moment2}
		\frac12\int_0^\infty (1-\tilde \rho^2)\tilde \theta'=\frac{\gp}{2}\quad \text{ and }\quad
		\frac12\int_{-\infty}^{0} (1-\tilde \rho^2)\tilde \theta'=\frac{\gq}{2}.
		\eq
		For notational simplicity,  we continue to write $u$, $\rho$ and $\theta$ for $\tilde u$, $\tilde \rho$ and $\tilde \theta$.
		Now we introduce the reflexion operators
		$$(T^+\rho)(x)=\begin{cases}
		\rho(x),  & \text{ if } x\geq 0,\\
		\rho(-x), & \text{ if } x<0,
		\end{cases}
		\qquad 
		(T^-\rho)(x)=\begin{cases}
		\rho(-x),& \text{ if } x\geq 0,\\
		\rho(x),  &   \text{ if } x<0,
		\end{cases}
		$$
		and 
		$$(S^+\theta)(x)=\begin{cases}
		\theta(x)-\theta(0), & \text{ if } x\geq 0,\\
		\theta(0)-\theta(-x), & \text{ if } x<0,
		\end{cases}
		\qquad 
		(S^-\theta)(x)=\begin{cases}
		\theta(0)-\theta(-x), & \text{ if } x\geq 0,\\
		\theta(x)-\theta(0), & \text{ if } x<0.
		\end{cases}
		$$
		Since $\rho$ and $\theta$ are continuous and belong to 
		$H^{1}_{\loc}(\R)$, we can check that
		the functions $(T^\pm\rho)$ and $(S^\pm\rho)$ are  continuous on $\R$ and also belong to $H^{1}_{\loc}(\R)$.
		Then it is simple to verify that the functions 
		$$u^{\pm}=(T^\pm\rho)e^{iS^\pm\theta}$$
		belong to $\boEEE$. Bearing in mind \eqref{sum-moment2}, we obtain
		\bqq
		p(u^{+})=\gp \quad\text{ and }\quad p(u^{-})=\gq,
		\eqq
		which implies that 
		\bq\label{Eu+}
		\Emin(\gp)\leq E(u^+) \quad\text{ and }\quad \Emin(\gq)\leq E(u^-).
		\eq
		In addition 
		\bq\label{sum-Es}
		E(u^+)+E(u^-)=2E_\k(u)+E_\p(u^+)+E_\p(u^-).
		\eq
		We claim that 
		\bq\label{ineq-E-p}
		E_\p(u^+)+E_\p(u^-)\leq 2E_\p(u),
		\eq
		which combined with \eqref{sum-Es}, allows us to conclude that $E(u^+)+E(u^-)\leq 2E(u)$. By putting together this inequality, 
		\eqref{sum-p} and \eqref{Eu+}, we get 
		$$E_{\min}(\gp) + E_{\min}(\gq) \leq 2 E(u) \leq 2 E_{\min} \Big( \frac{\gp + \gq}{2} \Big) + \delta,$$
		so that \eqref{ineq-conc} is proved. Since $\Emin$ is a continuous function by Proposition~\ref{Lipschitz}, we conclude that $E$ 
		is concave on $\R^+$.
		
		It remains to prove \eqref{ineq-E-p}.
		Let us set $\eta=1-\abs{u}^2$, $\eta_1=1-\abs{u^+}^2$, $\eta_2=1-\abs{u^-}^2$, 
		$$g(x)=\frac12 (\eta(x)+\eta(-x))\quad\text{ and }\quad f(x)=\frac12 (\eta(x)-\eta(-x)).$$
		Hence $g$ is even, $f$ is odd,
		$$\eta=f+g, \quad \eta_1=g+\tilde f \quad\text{ and }\quad \eta_2=g-\tilde f,$$
		where $\tilde f(x)=f(x)$ for $x\in\R^+$ and 
		$\tilde f(x)=-f(x)$ for $x\in\R^-$.
		By Plancherel's identity, we then can write
		\begin{align*}
		8\pi  (2E_\p(u)-E_\p(u^+)-E_\p(u^-) )
		&=\int_{\R}\wh \W(\xi)(2\abs{\hat\eta}^2 -\abs{\hat\eta_1}^2 -\abs{\hat\eta_2}^2) \\
		&=\int_{\R}\wh \W(\xi)(2\abs{\hat g +\hat f}^2 -\abs{\hat g +\hat {\tilde f}}^2 -\abs{\hat g -\hat {\tilde f}}^2) \\
		&=2\int_{\R}\wh \W(\xi)(\abs{\hat f}^2 -\abs{\hat {\tilde f}}^2 +4\int_{\R}\wh \W(\xi)\langle \hat g, \hat f\rangle \\
		&=4\pi(B(f)-B(\tilde f)),
		\end{align*}
		where we have used the parity of $\wh \W$ to check that $\int_{\R}\wh \W(\xi)\langle \hat g, \hat f\rangle =0.$
		To conclude, we only need to show that 
		$B(f)- B(\tilde f)\geq 0.$
Indeed, since $f$ is odd and $\tilde f$ is even,  we have
		$\hat f(\xi)=-2i\hat f_s(\xi)$ and $\hat{ \tilde f}(\xi)=2\hat f_c(\xi)$.
Therefore, by Plancherel's theorem, \ref{H-residue-bis},
and using that $\wh \W(\xi)(\abs{\hat f_s(\xi)}^2-\abs{\hat  f_c(\xi)}^2)$ is an even function, 
				\bqq
		(2\pi)(	B(f)- B(\tilde f))
		=4\int_\R \wh \W(\xi)(\abs{\hat f_s(\xi)}^2-\abs{\hat  f_c(\xi)}^2)d\xi=
		8\int_0^\infty \wh \W(\xi)(\abs{\hat f_s(\xi)}^2-\abs{\hat  f_c(\xi)}^2)d\xi \geq 0,
		\eqq
		which completes the proof.
	\end{proof}
		%


		The following lemma shows 
		that assumption \ref{H-residue} is stronger than \ref{H-residue-bis}, 
and is a reminiscent of Lemmas~2.1 and 2.6 in \cite{lopes-maris}.

		\begin{lem}
			\label{lem-Bf}
			Assume that \ref{H-residue} holds. Then  \ref{H-residue-bis} is satisfied.
%
		\end{lem}
		\begin{proof}
We notice that by Fubini's theorem, we have
			\begin{align*}
				\abs{\hat f_s(\xi)}^2=\int_0^\infty\int_0^\infty\sin(x\xi)\sin(y\xi)f(x)f(y)dx dy, \\
				|\hat f_c(\xi)|^2=\int_0^\infty\int_0^\infty\cos(x\xi)\cos(y\xi)f(x)f(y)dx dy.
			\end{align*}
Thus, 			introducing the complex-valued function 
			$$h(\xi)=\int_0^\infty\int_0^\infty e^{i(x+y)\xi}f(x)f(y)dx dy=\left(\int_0^\infty e^{ix\xi}f(x)dx\right)^2,$$
			we conclude that
			\bq
			\label{lemm-res}
						\int_\R \wh \W(\xi)(\abs{\hat f_s(\xi)}^2-\abs{\hat { f_c}(\xi)}^2)d\xi
			=	-\int_\R \wh \W(\xi) h(\xi) d\xi.
			\eq
			Then, using that $\bar h(\xi)=h(-\xi)$ and that $\wh \W$ is even, we conclude that
			\bq
			\label{dem-intW}
	\int_\R \wh \W(\xi)(\abs{\hat f_s(\xi)}^2-\abs{\hat { f_c}(\xi)}^2)d\xi=	-\int_\R \wh \W(\xi)\Re(h(\xi))d\xi=-\int_\R \wh \W(\xi)h(\xi)d\xi.
			\eq
			We will compute the integral in the right-hand side of \eqref{dem-intW} by using Cauchy's residue theorem. First we notice that $h$ is real-valued and nonnegative on the imaginary line since
			\bqq
			h(it)=\left(\int_0^\infty e^{-tx}f(x)dx\right)^2 \geq 0, \quad \text{ for all }t\in \R.
			\eqq
			Also, since $f\in C_c^\infty(\R)$, $h$ is a holomorphic  function on $\C$. 
			To establish the decay of $h$ on the upper half-plane, 
			we use that $h(z)=H(z)^2$, where
			$$H(z)=\int_0^\infty e^{ix z}f(x)dx.$$
			Using the fact that $e^{ixz}=\frac{1}{iz}\frac{d}{dx} e^{ixz}$ and integrating by parts, we get for $z\neq0$,
			\bqq
			H(z)=-\frac{f(0)}{iz}-\frac1{iz}\int_0^\infty e^{ixz}f'(x)dx.
			\eqq
			Since $f$  is odd, $f(0)=0$, so that integrating by parts once more, we have
			\bqq
			H(z)=-\frac{f'(0)}{z^2}-\frac{1}{z^2}\int_0^\infty e^{ixz}f''(x)dx.
			\eqq
			Therefore, 
			\bq
			\label{borne-h}
			\abs{h(z)}\leq \frac{C}{\abs{z}^4}, \quad \text{ for all }z\neq 0,\  \Im(z)\geq 0,
			\eq
			where $C=\left(\abs{f'(0)}+\norm{f''}_{L^1}\right)^2$. 
			Using the curves $\gamma_k$,   Cauchy's residue theorem yields
			\bq
			\label{int-R-W}
			\int_{-k}^{k} \wh \W(\xi)h(\xi)d\xi+
			\int_{a_k}^{b_k} \wh \W(\gamma_k(t))h(\gamma_k(t))\gamma'_k(t)dt=2\pi i \sum_{j\in J_k}h(i \nu_j){\Res}(\wh \W,i\nu_j)\leq 0,
			\eq
			where $J_k$ refers to the poles enclosed by $\Gamma_k$. 
			Taking into account  \eqref{borne-h}, we see that 
			\bqq
			\Big| \int_{a_k}^{b_k} \wh \W(\gamma_k(t))h(\gamma_k(t))\gamma'_k(t)dt \Big| \leq
			C \textup{ length}(\Gamma_k)	 \sup_{t\in[a_k,b_k]}\frac{\abs{\wh \W(\gamma_k(t))}}{\abs{\gamma_k(t)}^4}, 
			\eqq
			so that the decay in \eqref{decayW} gives that the integral goes to $0$ as $k\to\infty$. Therefore,  using the  dominated convergence theorem, we can pass to the limit in \eqref{int-R-W}, and using 
			\eqref{dem-intW}, we conclude that  condition \ref{H-residue-bis} is satisfied.
		\end{proof}

		The following propositions provide estimates for the curve $\Emin$  near the origin.
		\begin{prop}
			\label{prop:minoration}
			There are constants  $\gq_{0}>0$ and $K_0>0$  
			such that 
			\bq
			\label{Emin-inf}
			\sqrt 2\q-K_0 \q^{3/2}\leq \Emin(\gq), \quad \text{for all }\gq \in [0,\gq_{0}).
			\eq
		\end{prop}
		\begin{proof}
			Invoking 
			Corollary~\ref{cor:supp2} 
			and  \eqref{droiteaudessus}, for $\delta\in(0,1/2)$, we have the existence of a function $v\in \boEEE$ such that $p(v)=\gq$
			and $E(v)<\Emin(\gq)+\delta\leq \sqrt{2}\q+\delta$.
			Then, using the estimate $\eqref{est-eta}$, we conclude that there is some $\q_0>0$
			small and a constant $K>0$, such that if $\q\leq \q_0$, then $E(v)\leq 1$ and also
			\bq 
			\label{cond-satis}
			\abs{1-\abs{v}^2}\leq K( \sqrt{2}\q+\delta).
			\eq
			Since we can assume that $K( \sqrt{2}\q_0+\delta)<1$, 
			we can apply  the inequality  \eqref{ctrlEPsurR} in Lemma~\ref{lmm:generalpE} to  conclude that
			$\sqrt2(1-(K( \sqrt{2}\q+\delta)^{1/2}) p(v)\leq E(v)$. 
			Inequality \eqref{Emin-inf} follows letting $\delta\to0$.
				\end{proof}
		The rest of the section is devoted to establish the following upper bound for $\Emin$. So far, we have assumed that \ref{H0} and \ref{H-coer}
		hold, but we have not used the $C^3$ regularity nor the condition 
		$(\wh \W)''(0)>-1$. These hypotheses are going to be essential to prove the following proposition.
				\begin{prop}
			\label{prop-kdv}
			There exist
			constants  $\gq_1,K_1, K_2>0$,  depending on $\norm{\wh \W}_{C^3}$, such that
			\bq
			\Emin(\gq)\leq \sqrt{2}\gq-K_1 \gq^{5/3}+K_2\gq^{2}, \quad\text{ for all }\gq \in [0,\gq_1],
			\label{eminqkdv}
			\eq
		\end{prop}
		As an immediate consequence of Propositions \ref{prop:minoration}
		and \ref{prop-kdv}, is  that  $\Emin$ is right differentiable at the origin, with  $\Emin^+(0)=\sqrt{2}$. 
		Moreover, if $\Emin$ is concave we also deduce  that 
		$\Emin$ is strictly subadditive as a consequence of the following  elementary
		lemma  (see e.g.~\cite{bethuel,chironmaris}).
		\begin{lmm}
			\label{lmm:sousad}
			Let $f:[0,\infty)\to\R$ be continuous concave function, with $f(0)=0$, 
			and with  right derivative at the origin $a:=f^+(0)$. Then 
			for any $\gs>0$, the following alternative holds: 
			\begin{enumerate}
				\item $f$ is linear on $[0,\gs]$, with $f(\gp)=a\gp$, for all $\gp \in [0,\gs]$, or
				\item $f$ is strictly subadditive on $[0,\gs]$.
			\end{enumerate}
		\end{lmm}
		\begin{cor}
			\label{coro:strictadd}
			The right derivative of $\Emin$ at the origin exists and $\Emin^+(0)=\sqrt{2}$. 
			In particular, if $\Emin$ is concave on $\R^+$, then  
			$\Emin$ is strictly subadditive on $\R^+$. 
		\end{cor}
		The proof of Proposition \ref{prop-kdv} is inspired  on the fact that 
		the Korteweg--de Vries (KdV) equation provides a good approximation of solutions of the 
		Gross--Pitaveskii equation when $\W=\delta_0$ in the 
		long-wave regime  \cite{zakharov86,bethuel-kdv-2009,chiron-rousset}.
		Our aim  is to extend this idea to the nonlocal  equation \eqref{ngp}.
		Let us explain how this works in the case of solitons, performing first some formal computations.
		We are looking to describe a solution of \eqref{ntwc} with $c\sim \sqrt{2}$, so we consider 
		$$c=\sqrt{2-\ve^2},$$
		and use the ansatz 
		$$u_\ve(x)=(1+\ve^2A_\ve(\ve x))e^{i\ve \vp_\ve(\ve x)}.$$
		Therefore, setting 
		\bq
		\label{def:Weps}
		\wh \W_\ve(\xi):=\wh \W(\ve \xi),
		\eq
		i.e.\ $\W_{\ve}(x)=\W( x/\ve)/\ve$ in the sense of distributions, we deduce that $u_\ve$ is a solution to \eqref{ntwc} if $(A_\ve,\vp_\ve)$ 
		satisfies
		\begin{gather}
			\label{eq-kdv1}
			\ve^2A_\ve''-\ve^2(1+\ve^2 A_\ve)\vp_\ve'^2-c\vp_\ve'(1+\ve^2A_\ve)-(1+\ve^2A_\ve)\big(\W_\ve*(2A_\ve+\ve^2A_\ve^2)\big)=0,\\
			\label{eq-kdv2}
			2\ve^2A_\ve'\vp_\ve'+(1+\ve^2A_\ve)\vp_\ve''+cA_\ve'=0.
		\end{gather}
		To handle the nonlocal term, we use the following lemma.
		\begin{lem}
			\label{lem-W-ve}
			For all $f\in H^3(\R)$, we have 
			\bq
			\label{exp-W-ve}
			\W_\ve*f=f-\frac{\ve^2}{2}(\wh \W)''(0)f''+\ve^3 R_\ve(f),
			\eq
			where 
			\bqq
			\norm{R_\ve(f)}_{L^2(\R)}\leq \frac{1}{6}\norm{\wh \W'''}_{L^\infty(\R)}\norm{f'''}_{L^2(\R)}.
			\eqq
		\end{lem}
		\begin{proof}
			Let us set $$R_\ve(f):=\frac{1}{\ve^3}
			\big(\W_\ve*f-f+\frac{\ve^2}{2}(\wh \W)''(0)f'' \big).$$
			By Plancherel's theorem, we have
			\begin{align}
				\label{dem-R}
				2\pi\norm{R_\ve(f)}^2_{L^2(\R)}=\norm{\boF( R_\ve(f))}^2_{L^2(\R)}=
				&\frac{1}{\ve^6}\int_{\R}
				\Big| \wh \W(\ve\xi)-1 -\frac{\ve^2 \xi^2}2(\wh \W)''(0) \Big|^2 \abs{\hat f(\xi)}^2   d\xi.
			\end{align}
			Now, by Taylor's theorem and the fact that $(\wh \W)'(0)=0$, we deduce  that  for all $\xi \in \R$ and $\ve>0$,
			there exists $z_{\ve,\xi}\in\R$ such that
			$$\wh \W(\ve \xi)=1+\frac{\ve^2\xi^2}{2}
			(\wh \W)''(0)+\frac{\ve^3\xi^3}{6}(\wh \W''')(z_{\ve,\xi}).$$
			Replacing this equality into \eqref{dem-R}, we conclude that
			$$\sqrt{2\pi}\norm{R_\ve(f)}_{L^2(\R)}\leq \frac{1}{6}\norm{\wh \W'''}_{L^\infty(\R)}\norm{\boF(f''')}_{L^2(\R)}=
			\frac{\sqrt{2\pi}}{6}\norm{\wh \W'''}_{L^\infty(\R)}\norm{f'''}_{L^2(\R)},$$
			which completes the proof of the lemma.
		\end{proof}
		In this manner, applying Lemma \ref{lem-W-ve}, we formally deduce from \eqref{eq-kdv1}--\eqref{eq-kdv2} that
		\begin{gather}
			\label{eq-kdv3}
			-c\vp_\ve'-2A_\ve+\ve^2(-c\vp_\ve' A_\ve-3A_\ve^2+(1+\wh \W''(0)) A_\ve''-\vp_\ve'^2)=\boO(\ve^3),\\
			\label{eq-kdv4}
			\vp_\ve''+cA_\ve'+\ve^2(2\vp_\ve' A_\ve'+A_\ve\vp_\ve'')=0.
		\end{gather}
		Therefore for the speed $c=\sqrt{2-\ve^2}$, \eqref{eq-kdv3} implies that 
		\bq
		\label{phi-A}
		\vp_\ve'=-2A_\ve+\boO(\ve^2).
		\eq
		Differentiating \eqref{eq-kdv3}, adding \eqref{eq-kdv4} multiplied by $c$, using \eqref{phi-A}, 
		and supposing that $A_\ve$ and $\vp_\ve$ converge to some functions $A$ and $\vp$, respectively, 
		as  $\ve\to 0$,  we obtain the limit equation
		\bqq
		-A'-12AA'+(1+\wh \W''(0))A'''=0.
		\eqq
		Thus, imposing that $A,A',A''\to 0$ as $\abs{x}\to \infty$, by integration, we get 
		\bq
		\label{eq-kdv}
		(1+\wh \W''(0))A''-6A^2-A=0.
		\eq
		By hypothesis \ref{H-coer}, we have $(\wh \W)''(0)>-1$, so that  setting 
		\bqq
		\omega:= (1+(\wh \W)''(0))^{1/2},
		\eqq
		so that the solution to \eqref{eq-kdv} (up to translations) corresponds to a soliton for the KdV equation given explicitly by
		\bq
		\label{def-A}
		A(x):=-\frac{1}{4}\sech^2\big(\frac{x}{2 \omega} \big).
		\eq
		Moreover, \eqref{phi-A} reads in the limit $\vp'=-\sqrt{2}A$, so that we choose $\vp$ as 
		\bq
		\label{def-vp}
		\vp(x):=\frac{ \omega}{\sqrt{2}}\tanh\big(\frac{x}{2 \omega} \big).
		\eq
		In this manner, we should expect that $u_\ve(x)\sim(1+\ve^2A(\ve x))e^{i\ve \vp(\ve x)}$.
		This is the motivation of the following result.
		\begin{lem}\label{lem:kdv}
			Let $v_\ve(x)=(1+\ve^2A(\ve x))e^{i\ve \vp(\ve x)}$, where $A$ and $\vp$ are given by \eqref{def-A} and 
			\eqref{def-vp}. Then 
			\bq
			\label{kdvmomentum}
			E(v_\ve)=\frac{\omega}{3}\left(\ve^3-\frac{\ve^5}{4}\right)+\mathcal{O}(\varepsilon^6)
			\quad\text{ and }\quad 
			p(v_\ve)=\frac{\sqrt{2} \omega}{6}\left(\ve^3-\frac{\ve^5}{10} \right),
			\eq
			where $\mathcal{O}(\varepsilon^7)/\ve^7$ is a function that is bounded in terms of $\norm{\wh \W}_{W^{3,\infty}}$, uniformly  for all $\ve\in (0,1]$.
		\end{lem}
		\begin{proof}
			Let us first compute the  momentum. Bearing in mind that  $\vp'=-\sqrt{2}A$, we have
			\begin{align*}
				p(v_\varepsilon)&=-\frac{1}{2}\int_{\mathbb{R}}\left(2\varepsilon^2 A(\varepsilon x)+\varepsilon^4A(\varepsilon x)^2\right)
				\varepsilon^2\varphi'(\varepsilon x)dx \\
				&= \frac{\sqrt{2}\varepsilon^3}{2}\int_{\mathbb{R}}\left(2 A( x)^2+\varepsilon^2A( x)^3\right)dx \\
				&= {\sqrt{2}\omega\varepsilon^3}
				\int_{\mathbb{R}}\left(\frac{1}{8}\sech(x)^{4}-
				\frac{\varepsilon^2}{64}\sech(x)^{6}\right)dx,
			\end{align*} 
			so  using that $\int_{\mathbb{R}}\sech^{4}(x)dx=4/3$ and that $\int_{\mathbb{R}}\sech^{6}(x)dx=16/15$,
			we obtain the expression  for $p(v_\ve)$ in \eqref{kdvmomentum}.
			For the kinetic energy we can proceed in the same manner. Indeed, using that 
			$$A'(x)=\frac{1}{4\omega}\tanh\left(\frac{x}{2\omega}\right)\sech^{2}\left( \frac{x}{2\omega} \right),\quad \text{ and }
			\int_{\mathbb{R}}\sech(x)^4\tanh(x)^2=\frac{4}{15},$$
			we get 
			\begin{align*}
				E_{\k}(v_\varepsilon)&=\frac{1}{2}\int_{\mathbb{R}}\left( \varepsilon^6A'(\varepsilon x)^2 + \varepsilon^4(1+\varepsilon^2A(\varepsilon x))^2\varphi'(\varepsilon x)^2\right)dx \\
				&= {\varepsilon^3}\int_{\mathbb{R}}A(x)^2 dx+\frac{\varepsilon^5}{2}\int_{\mathbb{R}}A'(x)^2+4A(x)^3 dx\\
				&=\frac{\varepsilon^3\omega}{8}\int_{\mathbb{R}}\sech(x)^4 dx+\frac{\varepsilon^5}{16\omega}\int_{\mathbb{R}}\sech^{4}(x)\tanh^2(x)dx-\frac{\varepsilon^{5}\omega}{16}\int_{\mathbb{R}}\sech^6(x) dx\\
				&=\frac{\varepsilon^3\omega}{6}+\varepsilon^5\left(\frac{1}{60\omega}-\frac{\omega}{15}\right).
			\end{align*}
			Now, for the potential energy, invoking  Lemma \ref{lem-W-ve} and \eqref{eq-kdv},  we have
			\begin{align*}
				E_{\p}(v_{\varepsilon})&=
				\frac{1}{4\varepsilon}
				\int_{\mathbb{R}}\big(
				\W_{\varepsilon}\ast(2\varepsilon^2A
				+\varepsilon^4A^2) 
				\big)(x)
				(2\varepsilon^2A(x)+\varepsilon^4A(x)^2)dx\\
				&=\varepsilon^3\int_{\mathbb{R}}A(x)^2dx+\varepsilon^5
				\int_{\mathbb{R}}\left(A(x)^3-\frac{\widehat{\W}''(0)}{2}A(x)A''(x)\right)dx+\mathcal{O}(\varepsilon^6) \\
				&= \varepsilon^3\int_{\mathbb{R}}A(x)^2dx+\varepsilon^5
				\int_{\mathbb{R}}\left(A(x)^3-\frac{\widehat{\W}''(0)}{2\omega^2}\left(A(x)^2+6A(x)^3\right)\right)dx+\mathcal{O}(\varepsilon^6)
				\\
								&= \frac{\varepsilon^3\omega}{6}-\frac{\varepsilon^5}{60}\left(\omega+\frac{1}{\omega}\right)+\mathcal{O}(\varepsilon^6),
			\end{align*}
			where we have also used that $\widehat{\W}''(0)=\omega^2-1.$ 
			Adding the expressions for $E_\k$ and $E_\p$, we obtain the estimate for the energy in \eqref{kdvmomentum}.
		\end{proof}
		\begin{proof}[Proof of Proposition~\ref{prop-kdv}]
			For $\q$ small, we can parametrize $\q$ as a function of $\ve$ as 
			$$
			\q_\ve=\frac{\sqrt{2}\omega}{6}{\left(\ve^3-\frac{\ve^5}{10}\right)},
			$$
			so $\q_\ve$ is a strictly increasing  function  of $\ve \in [0,1]$.
			The idea is to express $\ve$ in terms of $\gq_\ve$  
			in order to obtain $E(v_\ve)$ in  \eqref{kdvmomentum} as a function  of $\gq_\ve.$
			Then  \eqref{eminqkdv} will follow from the facts that $p(v_\ve)=\gq_\ve$ and that
			$E_{\min}(\gq_\ve)\leq E(v_\ve)$.
			For notational simplicity, we set
			\bq
			\label{def:s}
			\gs_\ve:=\frac{3\sqrt{2}}{\omega}\gq_\ve=\ve^3-\frac{\ve^5}{10},
			\eq
			so that 
			\bq
			\label{est:ve}
			\ve^3/2\leq \gs_\ve\leq \ve^3\leq 1,\quad\text{ for all }\ve\in[0,1].
			\eq
			Applying Taylor's theorem and noticing that $\ve^5/10\leq \gs_\ve$, 
			we infer that there is some $\gp_\ve \in (\gs_\ve,2\gs_\ve)$ such that 
			\begin{align*}
				\ve^5&=\left( \gs_\ve+\frac{\ve^5}{10}\right)^{5/3} 
				= \gs_\ve^{5/3}+\frac{5\ve^5}{30}\gp_\ve^{2/3}.
			\end{align*}
			Using again \eqref{est:ve}, we conclude that 
			\begin{align*} \ve^5&=\gs_\ve^{5/3}+\mathcal{O}(\gs_\ve^{7/3})=\left(\frac{3\sqrt{2}}{\omega}\right)^{5/3} \q_\ve^{5/3}+\boO(\q_\ve^{7/3}). 
			\end{align*}
			Combining this  asymptotics with  \eqref{kdvmomentum}, \eqref{def:s} and \eqref{est:ve}, we get
			\begin{align*}
				E(v_\ve)=
				\frac{\omega}{3}\left(
				\frac{3\sqrt{2}}{\omega}\gq_\ve-\frac{3\ve^5}{20}\right)+\boO(\ve^{6}) 				=\sqrt{2}\q_\ve-K_1\q_\ve^{5/3}+\boO(\q_\ve^{2}),
			\end{align*}
			where $K_1=(3\sqrt{2}/\omega)^{5/3}\omega/20$. Since $E_{\min}(\gq_\ve)\leq E(v_\ve)$, 
			we conclude that \eqref{eminqkdv} holds true.
		\end{proof}
		We are now in position to prove  Theorem \ref{thm:curve}.
		\begin{proof}[Proof of Theorem \ref{thm:curve}]
			Statement (i) follows from 
			Lemma~\ref{pariteetsousadd},  Proposition~\ref{Lipschitz} and 
			Lemmas~\ref{lem:subadd} and \ref{lem:increasing}. From Propositions~\ref{prop:minoration} and \ref{prop-kdv}, we obtain (ii). Proposition~\ref{prop-concave} and Lemma~\ref{lem-Bf} establish  (iii). 
			 
By Corollary~\ref{cor:q*}, $\q_*>0.027$.  			 
Let us proof now the rest  of the statement in (iv).
Since $\Emin$ is nondecreasing on $[0,\q_*)$, if we suppose that $\Emin$ is not strictly increasing,
then  $\Emin$ is constant in some interval $[a,b]$, with $0\leq a<b< \q_*$. Since $\Emin$ is concave, 
this implies that $\Emin$ is constant on $[a,\infty)$ and therefore $\Emin(a)=\Emin(\q_*)$, which 
contradicts the definition of $\q_*$ in \eqref{q_*}. Finally, we remark that if 
$E(v)< \Emin(\q_*)$, for some $v\in \boEE$, using the fact that $\Emin(0)=0$, the 
intermediate value theorem gives us the existence of some $\tilde \q\in [0,\q_*)$ such that 
$E(v)=\Emin(\tilde \q)$. Since $\tilde \q<\q_*$, the definition of $\q_*$ implies that $\abs{v}$ does not
vanish.
 
We now establish (v). Arguing by contradiction, we show that  $\Emin(\q)<\sqrt2 {\q}$, for all $\q>0$. Indeed, in view of \eqref{droiteaudessus}, let us suppose that for some $\gp>0$ we have $\Emin(\gp)=\sqrt2 {\gp}$. Since $\Emin$ is concave, 
the function $\q\mapsto \Emin(\q)/\q$ nonincreasing, thus
\bqq
\sqrt 2=\frac{\Emin(\gp)}{\gp}\leq \frac{\Emin(\q)}{\q}\leq \sqrt2, \quad\text{ for all }\q\in (0,\gp).
\eqq
Therefore $\Emin(\q)= \sqrt2\q$, for all $\q\in (0,\gp)$, which contradicts (ii).

At this point, we recall that the concavity of $\Emin$ implies that $\Emin^+$
is right-continuous, so that, by Corollary~\ref{coro:strictadd}, 
we have  $E^+_{\min}(\q) \to E^+_{\min}(0)=\sqrt{2}$, as $\q \to 0^+$.  
 Using also  that $\Emin$ is nondecreasing, \eqref{droiteaudessus} and Corollary~\ref{coro:strictadd}, we deduce the other statements in~(v). 
		\end{proof}

\section{Compactness of the minimizing sequences}
\label{sec:suites}
We start now the study of the minimizing sequences associated with the curve $\Emin$.
The following result shows that the set $\boS_\q$ in Theorem~\ref{thm-existence-general}
is nonempty, and also allows us to establish the orbital stability in the next section.
\begin{thm}
	\label{thm:compacite}
	Assume that  $\W$ satisfies \ref{H0} and \ref{H-coer}, and that 
 $E_{\min}$ is concave on $\mathbb{R}^+$.
 	Let  $\gq\in (0,\gq_*)$ and  $(u_n)$ in $\boEEE$ be a sequence  satisfying
	\bq
	\label{mseq}
	p(u_n)\to \gq\quad \textup{ and }\quad E(u_n)\to E_{\min}(\gq),
	\eq
	as $n\to\infty$.
	Then there exists
	$v \in \boEEE$,
	a sequence of points $(x_n)$
	such that, up to a subsequence that we still denote by  $u_n$,
	the following convergences hold
	\begin{alignat}{2}
		\label{cvuniforme} 
		u_{n}(\cdot+x_n)&\to  v(\cdot),  &\quad \text{ in }&L^\infty_{\loc}(\R),\\
		\label{cvforteepot}
		1-|u_{n}(\cdot+x_n)|^{2} &\to 1-|v(\cdot)|^{2},   &\quad\text{ in }&L^2(\R),\\
		\label{cvfortegradient}
		u_{n}'(\cdot+x_n)&\to v'(\cdot),  &\quad\text{ in }&L^2(\R),
	\end{alignat}
	as $n\to\infty$. In addition, there is a constant $\nu>0$ such that 
	\bq\label{inf-u_n}
	\inf_{\R}\abs{u_n(\cdot +x_n)}\geq \nu, \quad \text{ for all }n.
	\eq
		In particular 	$p(v)=\gq$, $E(v)=E_{\min}(\gq),$ and  $v\in\boS_\q$.
\end{thm}
In the rest of the section  we will assume that the hypotheses in Theorem~\ref{thm:compacite}
are satisfied and therefore the conclusion in Theorem~\ref{thm:curve}-(v) holds. Thus, in the sequel, $\Emin$
is strictly subadditive and $\Emin(\q)<\sqrt{2}\q$, for all
$\q>0$.

For the sake of clarity, we state first the following elementary lemma.
\begin{lem}
	\label{Thsolitonw}
	Let $(u_n)$ be a sequence as in Theorem~\ref{thm:compacite}.
	Then there is function  $u \in \boEEE$
	such that, up to a subsequence, 
	\begin{alignat}{2}
		\label{cvl2}
		u_n &\to u,&\quad \text{ in }&L^{\infty}_{\loc}(\R),\\
		\label{cv:grad:weak}
		u'_n &\rightharpoonup u', &\quad \text{ in }&L^2(\mathbb{R}),\\
		\label{weak-eta}
		\eta_n:=1-|u_n|^2&\rightharpoonup \eta:=1-|u|^2, &\quad\text{ in }&L^2(\mathbb{R}).
	\end{alignat}
In addition,
$E(u)\leq E_{\min}(\gq)$, 
and writing  $u=\rho e^{i\phi}$ and $u_n=\rho_n e^{i\phi_n}$,
 the following relations hold, up to a subsequence,
 for all $A>0$,
	\begin{align}
		\label{convergencespropcin}
		&\int_{-A}^{A}|u'|^2\leq \underset{n \to \infty}\liminf\int_{-A}^{A}|u_{n}'|^2,\\
		\label{convergencesproppot}
		&\int_{-A}^{A}(\W\ast\eta)\eta=
		\lim_{n\to\infty}\int_{-A}^{A}(\W\ast\eta_n)\eta_n ,\\
		\label{convergencesprop2}
		&\int_{-A}^{A}\eta\phi'=\underset{n \to \infty}\lim\int_{-A}^{A}\eta_n\phi_{n}'.
	\end{align}

\end{lem}

\begin{proof}
In view of \eqref{mseq},   $E(u_n)$ is bounded, 
so that, using also Lemma~\ref{lem-control-energy}, 
	we deduce that  $u'_n$ and that $\eta_n:=1-\abs{u_n}^2$ are bounded in $L^2(\R)$ and that $u_n$ is bounded in $L^\infty(\R)$. 
	Therefore,  by weak compactness in Hilbert spaces and the Rellich--Kondrachov theorem,  there is a function $u \in H^{1}_{\loc}(\mathbb{R})$ 
	such that, up to a subsequence, the convergences in \eqref{cvl2}--\eqref{weak-eta} hold, 
as well as  \eqref{convergencespropcin}, and also
	\bq
	\label{liminfsurtoutrkin}
	\Vert u'\Vert_{L^{2}(\mathbb{R})}\leq \underset{n\to\infty}\liminf \Vert u_n'\Vert_{L^{2}(\mathbb{R})}.
		\eq 
		At this point we remark that the function $B(f)=\int_{\R} (\W*f)f$ is continuous and convex in $L^2(\R)$, since $\wh \W\geq 0$ a.e.
	Thus it is  weakly lower semi-continuous, so that 
	\bq
	\label{for-contrad}
B(u) \leq \underset{n\to\infty}\liminf B(u_n).
	\eq 
Combing with \eqref{liminfsurtoutrkin}, we deduce  that  $E(u)\leq E_{\min}(\gq)$. 
Using \eqref{weak-eta} and the fact that  $\W \in \mathcal{M}_{2}(\mathbb{R})$, 
we get 
	\bq
	\label{weak-W}
	\W\ast \eta_n\rightharpoonup \W\ast \eta \quad\textup{ in } L^{2}(\mathbb{R}), 
	\eq 
which together with  \eqref{cvl2} lead to \eqref{convergencesproppot}.
	
	Since $\q\in (0,\q_*)$, Theorem~\ref{thm:curve} and the fact that $E(u)\leq E_{\min}(\gq)<\Emin(\gq_*)$ imply that  $u\in \boEEE$, so that we can write $u=\rho e^{i\phi}$.
	Then, setting $u_n=\rho_n e^{i\phi_n}$ and by using that $E_{\k}(u_n)$ is bounded and \eqref{cvl2}, we get for $A>0$,
	$$ \int_{-A}^A\phi_{n}'^2\leq \frac{1}{\underset{[-A,A]}\inf|u_n|^2}\int_{\mathbb{R}}\rho_{n}^2\phi_{n}'^2 \leq  \frac{4}{\underset{[-A,A]}\inf|u|^2}E_{\k}(u_n),$$
	 so that, up to a subsequence, 
	$ \phi_n' \rightharpoonup \phi'$ in $L^{2}([-A,A]).$
	Using again  \eqref{cvl2}, we then establish \eqref{convergencesprop2}.
	\end{proof}
\begin{proof}[Proof of Theorem \ref{thm:compacite}] 
	By hypothesis, we can assume that 
	\bq\label{borneEn}
	E(u_n)\leq 2E_{\min}(\q).
	\eq
	 Since   $\Emin(\q)<\sqrt{2}\q$,
	we have 	$\Sigma_\gq\in (0,1),$ so that applying Lemma~\ref{lmm:evanes2} 
	with $L=1+\Sigma_\gq$, and Lemma~\ref{lemmabgs}
	with $E=2E_{\min}(\q)$ and $m_0=\tilde \Sigma_\gq:=\Sigma_\gq/L$, we deduce that 
	there exist an integer $l_\q$, depending on $E$ and $\q$, but not on $n$,
	and  points  $x_1^n,x_2^n,\dots,x_{l_n}^n$, with $l_n \leq l_\q$
	such that 
	\bq
	\label{compacteq} 
	|1-|u_{n}(x_j^n)|^2|\geq \tilde\Sigma_\gq,\quad \forall 1\leq j\leq l_n 
	\eq
	and
	\bq
	|1-|u_{n}(x)|^2|\leq \tilde\Sigma_\gq,\quad \forall x\in \mathbb{R}\setminus\bigcup_{j=1}^{l_n}[x_j^n-1,x_j^n+1]. 
	\label{compacteq2}
	\eq
	Since the sequence $(l_n)$ is bounded, we can assume that, up to a subsequence, $l_n$ does not depend on $n$ and set $l_*=l_n$. Passing again to a
	further subsequence and relabeling the points 
	$(x_j^n)_{j}$ if necessary, there exist some integer $\ell$,  with $1\leq  \ell\leq l_*$, and some number $R>0$ such that
	
	\bq
	\label{xi:xk:inf}
	|x_{k}^{n}-x_{j}^{n}|\underset{n\to \infty}\longrightarrow \infty,~\forall 1\leq k\neq j\leq \ell
	\eq
	and
	\bqq
	x_{j}^{n} \in \displaystyle\mathop{\cup}_{k=1}^{\ell}B(x_{k}^n,R),~\forall \ell<j \leq l_*.
	\eqq
	Hence, by \eqref{compacteq2}, we deduce that
	\bq 1-\tilde \Sigma_\gq\leq  |u_{n}|^2\leq 1+\tilde \Sigma_\gq,\quad\text{on } \mathbb{R}\setminus\bigcup_{j=1}^{\ell}B(x_{j}^n,R+1). 
	\label{compaceq3}
	\eq
	Applying Lemma~\ref{Thsolitonw} to the translated sequence $u_{n,j}(\cdot)=u_{n}(\cdot+x_{j}^{n})$, 
	we infer that there exist functions $v_{j}=\rho_j e^{i\phi_j}\in \boEEE$, $j\in\{1,\dots,\ell\}$,  satisfying
	the following convergences 
	\begin{alignat}{2}
		\label{cvuniformedemo}
		u_{n,j} &\to v_j,\quad &\text{ in }&L^{\infty}_{\loc}(\R),\\
		\label{cv:weak:grad}
		u'_{n,j} &\rightharpoonup v_j', \quad &\text{ in }&L^2(\mathbb{R}),\\
		\label{weak:eta:j}
		\eta_{n,j}:=1-|u_{n,j}|^2&\rightharpoonup \eta_j:=1-|v_j|^2, &\ \text{ in }&L^2(\mathbb{R}),
	\end{alignat}
	as $n\to\infty$, and also 
	\begin{gather}
		\label{evjeminq}
\Emin(\q_j)\leq 		E(v_j)\leq E_{\min}(\gq),
		\\
		\label{cvfaibleenergiekinsegment}
		\int_{-A}^{A}|v_j'|^2\leq \underset{n\to \infty}\liminf\int_{-A}^{A}|u_{n,j}'|^2,
		\\
		\underset{n\to \infty}\lim\int_{-A}^{A}(\W\ast\eta_{n,j})\eta_{n,j}=\int_{-A}^{A}(\W\ast\eta_j)\eta_j,
		\label{cvfaibleenergiepotsegment}
		\\
		\label{cvfaiblemomentsegment}
\underset{n \to \infty}\lim	\int_{-A}^{A}\eta_{n,j}\phi_{n,j}'=\int_{-A}^{A}\eta_j\phi_{j}',
	\end{gather}
	where $u_{n,j}=\rho_{n,j}e^{i\phi_{n,j}}$ and $\q_j=p(v_j)$.
%
%
	Moreover, using  \eqref{compacteq} and 
	\eqref{cvuniformedemo}, we infer that
		\bq |1-|v_{j}(0)|^2| \geq \tilde\Sigma_{q}.
	\label{ineqpas0u2}
	\eq
		In particular, $v_j$ cannot be a constant function of modulus one.
%
Now we focus on proving the following claim.
	\begin{claim}
		\label{step2}
		There exist $\tilde{\gq}\in \R$ and $\tilde{E}\geq 0$ such that
		\begin{align}
		E_{\min}(\gq)&\geq \sum_{j=1}^{\ell}E_{\min}(\gq_j)+\tilde{E}\quad \text{ and}
		\label{inegEqfinal}
		\\
		\gq&=\sum_{j=1}^{\ell}\gq_j+\tilde{\gq}.
		\label{inegqfinal}
		\end{align}
	\end{claim}
For this purpose, we fix $\mu>0$. By the dominated convergence theorem,
there exists
\bq
R_{\mu}\geq \max\left(R+1,\frac{1}{\mu}\right),
\label{defrayonmu}
\eq  
such that, for $1 \leq j\leq \ell,$ 
\bq 
\frac{1}{2}\int_{-R_\mu}^{R_\mu}|v_j'|^2\geq E_{\textup{kin}}(v_j)-\frac{\mu}{2\ell}.
\label{ekinvjsurboules}
\eq 
By \eqref{xi:xk:inf}, we can assume that $B(x_{k}^n,R_{\mu})\cap B(x_{j}^{n},R_{\mu})=\emptyset$, for all 
$1 \leq k\neq j\leq \ell.$ Hence, using \eqref{cvfaibleenergiekinsegment} and \eqref{ekinvjsurboules}, we deduce that there exists $N_{\mu}\geq 1$, such that for all $n\geq N_{\mu}$ and for all $1 \leq k\neq j \leq \ell,$ 
\bq 
\frac{1}{2}\int_{-R_\mu}^{R_\mu}|u_{n,j}'|^2\geq E_{\textup{kin}}(v_j)-\frac{\mu}{\ell}.
\label{ekinunversusekinvj}
\eq
By adding the inequality \eqref{ekinunversusekinvj} from $j=1$ to $j=\ell,$ we conclude that 
\bq
			\label{inegenergkinetape2}
\frac{1}{2}\sum_{j=1}^{\ell}\int_{-R_\mu}^{R_\mu}|u_{n,j}'|^2\geq \sum\limits_{j=1}^{\ell}E_{\k}(v_{j})-\mu, 
\quad \text{ for all }n\geq N_\mu.
\eq
 Similarly, using again the dominated convergence theorem and possibly increasing  $R_\mu$, we obtain for all $1\leq j\leq \ell,$ 
\bq
\left| \frac{1}{4}\int_{-R_{\mu}}^{R_{\mu}}(\W\ast \eta_j)\eta_j-E_{\p}(v_j)\right| \leq \frac{\mu}{2\ell}. 
\label{epotvjsurboules}
\eq 
By \eqref{cvfaibleenergiepotsegment}, and increasing  $N_\mu$ if necessary, we have for $n\geq N_{\mu},$ 
\bq
\left| \frac{1}{4}\int_{-R_{\mu}}^{R_{\mu}}(\W\ast \eta_j)\eta_j-\frac{1}{4}\int_{-R_{\mu}}^{R_{\mu}}(\W\ast\eta_{n,j})\eta_{n,j}\right| \leq \frac{\mu}{2\ell}.
\label{epotvjunsurboules}
\eq 
Combining \eqref{epotvjsurboules}, \eqref{epotvjunsurboules}
and adding from $j=1$ to $j=\ell,$ we deduce that
\bq
\label{inegenergpotetape2}
\left| \frac{1}{4}\sum_{j=1}^{\ell}\int_{-R_\mu}^{R_\mu} (\W\ast \eta_{n,j})\eta_{n,j}-\sum_{j=1}^{\ell}E_{\p}(v_j)\right|\leq \mu,
\quad \text{ for all }n\geq N_\nu.  
\eq
Applying the same argument to $\eta_{n,j}\phi_{n,j}'$ and $\eta_{j}\phi_{j}'$
instead of $(\W\ast\eta_{n,j})\eta_{n,j}$ and $(\W\ast\eta_j)\eta_j$, we get
\bq
			\label{inegmomentetape2}
			\left| \frac{1}{2}\sum_{j=1}^{\ell}\int_{-R_{\mu}}^{R_{\mu}}\eta_{n,j}\phi_{n,j}'-\sum\limits_{j=1}^{\ell}\gq_j \right| \leq \mu.
\eq	
	Now we handle the integrals on
	\bqq
	\boA_\mu:=\mathbb{R}\setminus\bigcup_{j=1}^{\ell}B(x_{j}^n,R_\mu).
	\eqq
		 Let us start with the momentum.	
We split $p(u_n)$ as
	\begin{align}
		p(u_n) =\frac{1}{2}\sum_{j=1}^{\ell}\int_{-R_\mu}^{R_\mu}\eta_{n,j}\phi_{n,j}'+
		{p_{\boA_{\mu}}(u_n)},\quad  \text{ with  }\
{p_{\boA_{\mu}}(u_n)}:=\frac{1}{2}\int_{\boA_{\mu}}\eta_n\phi_{n}'. 
		\label{decoupagepun}
	\end{align}
	By \eqref{est-eta-L2}, \eqref{mom-young}, \eqref{borneEn} and \eqref{compaceq3}, we obtain
	\begin{align*}
\sqrt{2}\abs{p_{\boA_{\mu}}(u_n)}
		 &\leq  \frac{1}{4}\int_{\boA_{\mu}}\eta_n^2+\frac{1}{2(1-{\tilde\Sigma_\gq})}\int_{\boA_{\mu}}\rho_n^2
		 \phi'^2_n \leq  C(\gq).
	\end{align*}  
	Hence, $p_{\boA_{\mu}}(u_n)$ is is uniformly bounded with respect to  $n$ and $\mu,$  so that,
	passing possibly to a subsequence (in $n$ and $\mu$), we infer that there exists $\tilde{\gq} \in \mathbb{R}$ such that
	\bq 
	\underset{\mu\to 0}\lim~\underset{n\to\infty}\lim  
	p_{\boA_{\mu}}(u_n) =\tilde{\q}. \label{def:qtilde}
	\eq 
	Hence, passing to the limit $n\to \infty$ and  then letting $\mu\to 0$ in \eqref{inegmomentetape2}, and
	using  \eqref{decoupagepun}, we obtain \eqref{inegqfinal}.
To prove \eqref{inegEqfinal}, we first remark that 
since $E_{\k}(u_n)$ and $E_{\p}(u_n)$ are bounded, passing possibly to a subsequence, 
there are constants $E_{k}, \tilde E_{k},E_p\geq  0$ such that $E_{\min}(\gq)=E_k+E_p$, 
\begin{align*}
 E_{\k}(u_n) \to E_{k},\quad 
 E_{\p}(u_n) \to E_{p},
 \end{align*}
 and 
 \bqq
\lim_{\mu\to 0} \lim_{n\to\infty}  \frac{1}{2}\int_{\boA_{\mu}}|u_n'|^2= \tilde E_k.
  \eqq
Thus, decomposing  the kinetic part as
	\begin{align*}
		E_{\k}(u_n)&
		=\frac{1}{2}\int_{\boA_{\mu}}|u_n'|^2+\frac{1}{2}\sum_{j=1}^{\ell}\int_{-R_{\mu}}^{R_{\mu}}|u_{n,j}'|^2,
	\end{align*}
and using  \eqref{inegenergkinetape2}, we deduce as before that
	\bq
	E_{k}\geq \sum_{j=1}^{\ell}E_{\k}(v_j)+\tilde{E_{k}}. 
	\label{ineqekinenmu}
	\eq 
To prove \eqref{inegEqfinal}, it remains to study the potential energy.
However, $E_\p(u)$ is more involved  because of the nonlocal interactions.
To make the decomposition,  we introduce the functions 
 $$g_{n,\mu}(x):=\eta_n(x)\mathds{1}_{\mathop{\cup}\limits_{j=1}^{\ell}B(x_{j}^n,R_{\mu})}(x)\quad \text{  and }\quad f_{n,\mu}(x):=\eta_n(x)\mathds{1}_{\boA_{\mu}}(x),$$
so that 
	\begin{align}
		E_{\p}(u_n)&=\frac{1}{4}\int_{\mathbb{R}}(\W\ast\eta_n)(f_{n,\mu}+g_{n,\mu})\nonumber  =\frac{1}{4}\int_{\mathbb{R}}(\W\ast\eta_n)f_{n,\mu}+\frac{1}{4}\int_{\mathop{\cup}\limits_{j=1}^{\ell}B(x_{j}^n,R_\mu)}(\W\ast\eta_n)\eta_n \\&= \frac{1}{4}\int_{\mathbb{R}}(\W\ast g_{n,\mu})f_{n,\mu}+\frac{1}{4}\int_{\mathbb{R}}(\W\ast f_{n,\mu})f_{n,\mu}+\frac{1}{4}\sum_{j=1}^{\ell}\int_{-R_{\mu}}^{R_{\mu}}(\W\ast\eta_{n,j})\eta_{n,j}.
		\label{decoupageepot}
	\end{align}
	Using Plancherel's identity, the Cauchy--Schwarz inequality and \eqref{est-eta-L2}, we deduce that
	\begin{align*}
		\left| \int_{\mathbb{R}}(\W\ast g_{n,\mu})f_{n,\mu} \right| &\leq {\Vert \widehat{\W} \Vert_{L^{\infty}(\mathbb{R})}}\Vert g_{n,\mu} \Vert_{L^{2}(\mathbb{R})} \Vert f_{n,\mu} \Vert_{L^{2}(\mathbb{R})} \leq C(E_{\min}(\gq)),
	\end{align*}
	and the same argument shows that  $\int_{\mathbb{R}}(\W\ast f_{n,\mu})f_{n,\mu}$
can also be bounded in terms of $E_{\min}(\gq)$.
	Passing possibly to a subsequence, we conclude that there exists $\tilde{E_p}\geq 0$ such that
	\bq
	{\underset{\mu\to 0}\lim}~\underset{n\to \infty}\lim \int_{\mathbb{R}}(\W\ast f_{n,\mu})f_{n,\mu}=4\tilde{E_p}.
	\label{boutrestantpositifepot}
	\eq
We will show   that 
	\bq
	{\underset{\mu\to 0}\lim}~\underset{n\to \infty}\lim \int_{\mathbb{R}}(\W\ast g_{n,\mu})f_{n,\mu}=0.
	\label{boutrestantpotence}
	\eq 
Assuming \eqref{boutrestantpotence}, we can now establish inequality \eqref{inegEqfinal}.
Indeed, letting $n\to\infty$ and then $\mu\to 0$ in \eqref{decoupageepot}, and using  \eqref{boutrestantpositifepot} and \eqref{boutrestantpotence}, we obtain
		$${\underset{\mu\to 0}\lim}~\underset{n\to \infty}\lim\left(\frac{1}{4}\sum_{j=1}^{\ell}\int_{-R_{\mu}}^{R_{\mu}}\left(\W\ast\eta_{n,j}\right)\eta_{n,j}\right)=E_{p}-\tilde{E_p}.$$
	Combining with \eqref{inegenergpotetape2}, we have
	\bq
	E_{p}= \sum_{j=1}^{\ell}E_{\p}(v_j)+\tilde{E_{p}}. 
	\label{dichoenepot}
	\eq 
Therefore, setting
\bq
\label{def:tildeE}
\tilde{E}:=\tilde{E_{k}}+\tilde{E_{p}}=\lim_{\mu\to 0} \lim_{n\to \infty}E_{\boA_{\mu}}(u_n),
\eq
	and bearing in mind that  $E_{\min}(\gq)=E_k+E_p$ and that $E(v_j)\geq \Emin(\q_j)$, 
	inequality \eqref{inegEqfinal} follows by adding \eqref{ineqekinenmu} and \eqref{dichoenepot}.

		It remains to show \eqref{boutrestantpotence}.
	By definition of $g_{n,\mu}$, we obtain
	\begin{align*}
		\int_{\mathbb{R}}(\W\ast f_{n,\mu})(x)g_{n,\mu}(x)dx &=
		  \sum_{j=1}^{\ell}\int_{B(x_j^n,R_{\mu})}\lpar \W\ast f_{n,\mu}\rpar (x)\eta_n(x)dx \\&= \sum_{j=1}^{\ell}\int_{B(0,R_{\mu})} \lpar \W\ast f_{n,\mu}\rpar (x+x_{j}^n)\eta_{n,j}(x)dx.
	\end{align*}
	Using also  \eqref{sym} and the fact that convolution commutes with translations, we get
	$$\int_{\mathbb{R}}(\W\ast g_{n,\mu})(x)f_{n,\mu}(x)dx=\sum_{j=1}^{\ell}\int_{\mathbb{R}\setminus\mathop{\cup}\limits_{k=1}^{\ell}B(x_{k}^n-x_{j}^n,R_{\mu})}\left(\W\ast(\eta_{n,j}\mathds{1}_{B(0,R_{\mu})})\right)(x)\eta_{n,j}(x)dx.$$
Noticing that $B(0,R_{\mu})$ is a subset of $\cup_{k=1}^{\ell}B(x_{k}^n-x_{j}^n,R_{\mu})$, 
we conclude that
	\begin{align}
	\label{ineq-w-n} \left|\int_{\mathbb{R}}(\W\ast g_{n,\mu})f_{n,\mu} \right| & 
\leq \sum_{j=1}^{\ell} \int_{\mathbb{R}\setminus B(0,R_{\mu})}\left|\W\ast(\eta_{n,j}\mathds{1}_{B(0,R_{\mu})})\right||\eta_{n,j}|.
		 \end{align}
To study the limit of the right-hand side of \eqref{ineq-w-n}, we  first remark 
that \eqref{cvuniformedemo} and the fact that 
 $\W\in \boM_2(\R)$ imply that 
\bq
\label{w-inl2}
 \W\ast(\eta_{n,j}\mathds{1}_{B(0,R_{\mu})})\to  \W\ast (\eta_j\mathds{1}_{B(0,R_{\mu})})\quad\text{~in~}L^{2}(\R),
 \eq
as $n\to\infty$. At this point we also notice that
 \eqref{cvuniformedemo} and  the same argument  leading to \eqref{weak:eta:j}, 
also give us  that $\abs{\eta_{n,j}}\wto\abs{\eta_j}$ in $L^2(\R)$. 
Combining with \eqref{w-inl2}, we thus get
	$$  \int_{\mathbb{R}\setminus B(0,R_{\mu})}\left|\W\ast(\eta_{n,j}\mathds{1}_{B(0,R_{\mu}))}\right||\eta_{n,j}|\to\int_{\mathbb{R}\setminus B(0,R_{\mu})}\left|\W\ast (\eta_j\mathds{1}_{B(0,R_{\mu})})\right||\eta_j|,$$
	as $n\to\infty$.
Finally, by the Cauchy--Schwarz inequality, 
	\begin{align}
	\label{proof-compa}
		\int_{\mathbb{R}\setminus B(0,R_{\mu})}\left|\W\ast \eta_j\mathds{1}_{B(0,R_{\mu})}\right||\eta_j|\leq \Vert \W \Vert_{2} \Vert \eta_j \Vert_{L^{2}(\mathbb{R})}\norm{\eta_j}_{L^2(\mathbb{R}\setminus B(0,R_{\mu})},
	\end{align}
so that the definition of $R_\mu$ in \eqref{defrayonmu} and the dominated convergence theorem 
allow us to conclude that the right-hand side of \eqref{proof-compa} goes to $0$ as $\mu\to0$. 
In view of \eqref{ineq-w-n} and \eqref{proof-compa}, this proves \eqref{boutrestantpotence}, completing the proof
of Claim~\ref{step2}.

Now we establish an inequality between $\tilde{\gq}$ and $\tilde{E}$ 
that will be key to conclude that both quantities are equal to zero. 
\begin{claim}
	\label{claim:ineq}
	We have
	\bq
	\label{inegqEfinal}
	\sqrt{2}\left(1-\tilde\Sigma_\gq\right)|\tilde{\gq}|\leq \tilde{E}.
	\eq
\end{claim}
This inequality is  a consequence of Lemma~\ref{lmm:generalpE}.
To choose our cut-off function, we take the sequence $\mu_m=1/m$, and we notice that since
$\lim|v_j(x)|\to1$  as $\abs{x}\to\infty$,
there exists $R_{j}>0$ such that, for every $|x|\geq R_{j},$ we have
\bq
\label{bornevjexpmu}
|\eta_j(x)|\leq e^{-2/{\mu_m}}.
\eq
Moreover, 	without loss of generality we can assume that $R_m:=R_{\mu_m}\geq  R_j$, for all $1\leq j\leq \ell.$
Now we use the function $\chi$ given by Lemma~\ref{lmm:cutoff} to define
$$ \chi_{j,n}(x):=\chi(x-x_j^n)=\begin{cases}
1 &\text{ if }|x-x_j^n|\leq R_m,  \\
0 &\text{ if } |x-x_j^n| \geq R_m+\mu_m,
\end{cases}
\quad
\text{ and }\quad
\tilde\chi_{n,m}:=1-\sum_{j=1}^{\ell}\chi_{j,n}.
$$
To establish \eqref{inegqEfinal}, we apply Lemma \ref{lmm:generalpE} with
$u=u_n$, $\Om=\boA_{\mu_m}$, $\ve=\tilde\Sigma_\gq$ and $\chi_{\Om,\Om_0}=\tilde{\chi}_{n,m}$,  where $\Omega_0$ is given by
$$\Om\setminus\Om_0= \bigcup_{j=1}^{\ell}[x_j^n-R_m-\mu_m,x_j^n-R_m]\cup [x_j^n+R_m,x_j^n+R_m+\mu_m].$$
Using \eqref{compaceq3}, the definitions of $\tilde{\gq}$ and $\tilde{E}$ in \eqref{def:qtilde} and \eqref{def:tildeE}, and letting $n\to\infty$ and $m\to \infty$ in \eqref{ctrlEPsuromega}, we obtain
\[
\sqrt{2}|\tilde{\gq}|\leq \frac{\tilde{E}}{1-\tilde\Sigma_\gq}+\lim_{m\to \infty}\limsup_{n\to\infty}\Delta_{n,m},
\]
with
{\bq
	|\Delta_{n,m}|\leq C(\q)\lpar \Vert \eta_n \Vert_{L^{2}(\Om\setminus\Om_0)}+ \Vert \eta_n\tilde\chi_{n,\mu_m}' \Vert_{L^{2}(\Om\setminus\Om_0)}
	+\Vert \eta_n\tilde\chi_{n,\mu_m}' \Vert_{L^{2}(\Om\setminus\Om_0)}^2
	\rpar.
	\label{bornedeltanmu}
	\eq}
Notice that we omit the dependence on $m$ and $n$ in $\Om\setminus\Om_0$
for notational simplicity.
Therefore, to prove \eqref{inegqEfinal} we only need to show that 	the right-hand side of \eqref{bornedeltanmu} goes to zero.
For the first term,  we have
\bqq
\Vert \eta_n \Vert_{L^{2}(\Om\setminus\Om_0)}^2=\sum_{j=1}^{\ell}\lpar \int_{-R_{m}-\mu_m}^{-R_m}\eta_{n,j}^2+\int_{R_{m}}^{R_{m}+\mu_m}\eta_{n,j}^2\rpar.
\eqq
Using \eqref{cvl2} and the dominated convergence theorem, 
we get
\bq
\underset{m\to \infty}\lim~\underset{n\to\infty}\limsup{\Vert \eta_n \Vert^2_{L^{2}(\Om\setminus\Om_0)}}=
\lim_{m\to\infty}	\sum_{j=1}^{\ell}\lpar \int_{-R_{m}-\mu_m}^{-R_m}\eta_{j}^2+\int_{R_{m}}^{R_{m}+\mu_m}\eta_{j}^2\rpar
=0.
\label{deltamuto01}
\eq

To bound the  term $\Vert\eta_n\tilde\chi_{n,\mu_m}'\Vert_{L^{2}(\Om\setminus\Om_0)}$
in \eqref{bornedeltanmu}, we notice that
$$ \left(\tilde{\chi}_{n,\mu_m}'\right)^2=\Big(\sum_{j=1}^{\ell}\chi_{j,n}'\Big)^2=\sum_{j=1}^{\ell}(\chi_{j,n}')^2,$$
since $\chi_{j,n}'\chi_{k,n}'=0$ for all $j\neq k.$ Hence,
\begin{align*}
\Vert \eta_n\tilde\chi_{n,\mu_m}' \Vert_{L^{2}(\Om\setminus\Om_0)}^2 
&\leq \sum_{j=1}^{\ell}\int_{\mathbb{R}}\eta_n^{2}|\chi_{j,n}'|^2 \\
&\leq \sum_{j=1}^{\ell}\lpar \int_{-{R_m}-\mu_m}^{-{R_m}}\eta_{n,j}^{2}|\chi'|^2+\int_{R_m}^{R_m+\mu_m}\eta_{n,j}^{2}|\chi'|^2\rpar.
\end{align*}
Invoking again \eqref{cvl2}, we obtain
\begin{align*}
\underset{n\to \infty}\limsup\Vert \eta_n\tilde\chi_{n,\mu_m}' \Vert_{L^{2}(\Om\setminus\Om_0)}^2 
&\leq \sum_{j=1}^{\ell}\lpar\int_{-{R_m}-\mu_m}^{-{R_m}}\eta_j^{2}|\chi'|^2+\int_{R_m}^{R_m+\mu_m}\eta_j^{2}|\chi'|^2\rpar
\\&\leq 32\ell e^{-4}\mu_m,
\end{align*}
where we have used  \eqref{bornevjexpmu} and
that $	|\chi'(x)|\leq 4e^{-2}e^{\frac{2}{\mu_m}}$
for the last inequality. Then, we conclude that
\bq
\lim_{m\to \infty}\limsup_{n\to\infty}\Vert \eta_n\chi_{n,\mu_m}'\Vert_{L^{2}(\Om\setminus\Om_0)}=0.
\label{deltamuto03}
\eq
Combining \eqref{deltamuto01} and \eqref{deltamuto03}, we obtain 	\bqq
\lim_{m\to \infty} \limsup_{n\to\infty}\Delta_{n,\mu_m}=0,
\eqq which completes the proof of Claim~\ref{claim:ineq}.

\begin{claim}
	We have $\tilde{E}=\tilde{\q}=0$ and $\ell=1$.
	\label{pasdichoto}
\end{claim}

	We suppose first that $\tilde{\gq}> 0.$ By definition of $\Sigma_\gq$ in \eqref{defsigmaq}, and using that $\tilde \Sigma_\q=\Sigma_\q/L<\Sigma_\q$,
	we have
	\bq
	\frac{E_{\min}(\q)}{\q}=\sqrt{2}(1-\Sigma_\q)<\sqrt{2}\left(1-\tilde\Sigma_\q\right).
	\label{rapportconcave1}
	\eq
In addition,  since $\Emin$ is concave, we obtain for all $0<\gp<\q$,
	\bq
	E_{\min}(\gp)\geq \gp\frac{E_{\min}(\q)}{\q}=\gp\sqrt{2}(1-\Sigma_\q).
	\label{rapportconcave2}
	\eq
Then, setting $\gs:=\gq-\tilde{\gq}=\sum\limits_{j=1}^{\ell}\gq_j$, the assumption
 $\tilde{\gq}>0$ implies that $\gs<\q$, and combining with 
 \eqref{inegqEfinal},  \eqref{rapportconcave1} and \eqref{rapportconcave2}, we also obtain
	\begin{align*}
		E_{\min}(\gs)&\geq \gs\frac{E_{\min}(\gq)}{\gq}=E_{\min}(\gq)-\tilde{\gq}\frac{E_{\min}(\gq)}{\gq}> E_{\min}(\gq)-\sqrt{2}\tilde{\gq}\left(1-{\tilde \Sigma_\gq}\right)\geq E_{\min}(\gq)-\tilde{E}.
	\end{align*}
Hence, using  \eqref{inegEqfinal}, we get
	\bq
	E_{\min}( \gs)>\sum\limits_{j=1}^{\ell}E_{\min}(\gq_j).
	\label{contradictionwidehatq}
	\eq 
		Since $E_{\min}$ is even, nondecreasing and subadditive, 
the inequality		 $\gs\leq \sum_{j=1}^{\ell}|\gq_j|$ yields
\bqq
	E_{\min}(\gs)\leq E_{\min}\Big(\sum_{j=1}^{\ell}|\gq_j|\Big)
	\leq \sum\limits_{j=1}^{\ell}{E}_{\min}(\gq_j).
	\eqq
	which contradicts \eqref{contradictionwidehatq}. Thus $\tilde{\gq}\leq 0$ 
	 and  \eqref{inegqfinal} gives
$	\gq \leq \sum_{j=1}^{\ell}|\gq_j|.
$
As before, this implies that 
	\bqq
	E_{\min}(\gq)\leq E_{\min}\Big(\sum_{j=1}^{\ell}|\gq_j|\Big)
	\leq \sum\limits_{j=1}^{\ell}{E}_{\min}(\gq_j).
	\eqq
	On the other hand, since  $\tilde{E}\geq 0$, we see from  \eqref{inegEqfinal} that
	\bqq
	{E}_{\min}(\q)\geq \sum\limits_{j=1}^{\ell}{E}_{\min}(\gq_j).
	\eqq
	Therefore 
	\bq
	{E}_{\min}(\q)=\sum\limits_{j=1}^{\ell}{E}_{\min}(\gq_j).
	\label{egaliteinegEEj}
	\eq 
In view of  \eqref{inegEqfinal} and \eqref{inegqEfinal},
\eqref{egaliteinegEEj}
	yields  $\tilde{E}=0$  and $\tilde{\gq}=0$. 
	Finally, if there are at least two nonzero values ${\gq_{k}}$
	and ${\gq_{m}}$, with $1\leq k\neq m\leq \ell$, then the strictly subadditivity of $\Emin$  implies that 
	$$ E_{\min}(\gq)= E_{\min}\Big(\sum_{j=1}^{\ell}|\gq_j|\Big)<\sum_{j=1}^{\ell}E_{\min}(\gq_j),$$ 
	contradicting \eqref{egaliteinegEEj}. Therefore we can suppose without loss of generality that $\ell=1$, 
which finishes the proof of Claim~\ref{pasdichoto}.

Setting $v=v_1$, the convergence in \eqref{cvuniforme} and the estimate in \eqref{inf-u_n} follow from \eqref{cvuniformedemo} and \eqref{compaceq3} (with $\ell=1$).
We now show the convergences in \eqref{cvforteepot} and \eqref{cvfortegradient} (with $v=v_1$) to   complete the proof of the theorem. Indeed, since $\ell=1$ and $\tilde{\gq}=0$, by Claim~\ref{pasdichoto}, \eqref{inegqfinal} shows that $\gq=\gq_1$,  and using also \eqref{mseq} and  \eqref{evjeminq}, we get
\bq
p(u_{n,1})\to \q=p(v) \quad \text{ and } \quad E(u_{n,1}) \to \Emin(\q)=E(v).
\label{q:pv1}
\eq 

We now establish \eqref{cvfortegradient}.
Since $u'_{n,1}\wto v'$ in $L^2(\R)$, 
it is enough to prove that 	\bq 
\label{cv:norme:grad}
\limsup_{n\to\infty }\Vert u_{n,1}'\Vert_{L^2(\mathbb{R})}\leq  \Vert v'\Vert_{L^2(\mathbb{R})}. 
\eq 
Arguing by contradiction, taking a subsequence that we still denote by $u_{n,1}$,
we suppose that 
$$M:=\lim_{n\to\infty }\norm{u_{n,1}}_{L^2(\R)}^2=2E_{\k}(u_{n,1}), \quad \text{ with } \quad M>\Vert v'\Vert_{L^2(\mathbb{R})}^2=2E_{\k}(v).$$
Hence, using \eqref{q:pv1},
\begin{align*}
\underset{n \to \infty}\lim~ E_{\p}\left(u_{n,1}\right)  &= \underset{n \to \infty}\lim  ~\big( E(u_{n,1} )-E_{\k}(u_{n,1}) \big) 
= E(v)-\frac{M}{2} 
< E(v)-E_{\k}(v)
= E_{\p}(v),
\end{align*}
which contradicts \eqref{for-contrad}. Therefore  $u'_{n,1}\to v'$ in $L^2(\R)$. 
In particular $E_k(u_{n,1})\to E_k(v)$, so that 
\eqref{q:pv1} implies that 
	\bq
	\label{cv:normeW}
	\lim_{n\to\infty}\int_{\R} \left(\W\ast \eta_{n,1}\right)\eta_{n,1}= \int_{\R} \left(\W\ast \eta\right)\eta,
	\eq
where $\eta=1-\abs{v}^2$ as usual.	Using Plancherel's identity and \ref{H-coer}, we have
	\begin{align}
\Vert \eta_{n,1}-\eta \Vert_{L^2(\R)}^{2}  
		 &\leq \frac{1}{2\pi}\int_{\R}\widehat{\W}(\xi)|\widehat{\eta_{n,1}}-\widehat{\eta}|^2
		+\frac{1}{4\pi}\int_{\R}\xi^2|\widehat{\eta_{n,1}}-\widehat{\eta}|^2 
		\nonumber
		\\
		\label{dernier1} &= \int_{\R}\W\ast(\eta_{n,1}-\eta)(\eta_{n,1}-\eta)
		+ \frac14\norm{\eta_n'-\eta'}_{L^2(\R)}^2.
	\end{align}
Since $\W \in \mathcal{M}_{2}(\R)$, it follows from \eqref{weak:eta:j} and \eqref{cv:normeW} that
	\bq
	\label{dernier2} \int_{\R}\W\ast(\eta_{n,1}-\eta)(\eta_{n,1}-\eta)\to 0.\eq
	It remains to prove that 
\bq
\label{eta1-etan1}
\norm{\eta_{n,1}'-\eta'}_{L^2(\R)}\to 0.
\eq 
Noticing that  $\eta'-\eta_{n,1}'=2(\langle v,v' \rangle- \langle u_{n,1},u_{n,1}' \rangle)$,
we have
\begin{align}
\label{dernier3}
\norm{\eta_{n,1}'-\eta'}_{L^2(\R)} 
	&\leq 2 \Vert (v-u_{n,1})v_1'\Vert_{L^2(\R)}
	+ 2\Vert (v'-u_{n,1}')u_{n,1}\Vert_{L^2(\R)} .
\end{align}
From inequality \eqref{est-eta}, we obtain 
\bq\label{bound-un}
\Vert u_{n,1} \Vert_{L^{\infty}(\R)}\leq C(\q).
\eq 
Thus, using \eqref{cvfortegradient}, we deduce that
 $$\Vert (v'-u_{n,1}')u_{n,1}\Vert_{L^2(\R)}
\leq  C(\q)\Vert v'-u_{n,1}'\Vert_{L^2(\R)} \to 0,$$
Moreover, \eqref{bound-un}  allows us to use the dominated convergence theorem to infer 
that the other term in the right-side of \eqref{dernier3} also converges to zero. Therefore, combining with 
\eqref{dernier1} and  \eqref{dernier2}, we obtain
\eqref{cvforteepot}, which finishes the proof
of the theorem.
\end{proof} 
\section{Stability}
	\label{sec:stability}
	We start recalling the following result concerning the Cauchy problem.
	\begin{thm}[\cite{de2010global}]\label{global0}
Let $\phi_0\in \boEE$, with $\grad \phi\in H^2(\R)\cap C(\R)$.
Let $\W\in \boM_3(\R)$ be an even distribution.
 Assume that one of the following is satisfied.
		\begin{itemize}
				\item[$(i)$]  $\W\in\boM_{1}(\R)$\textrm{ and }$\W\geq 0$\textrm{ in a distributional sense}.
			\item[$(ii)$] There exists $\sigma>0$ such that $\widehat{\W}\geq\sigma$ a.e.\ on $\R$.
							\end{itemize}
			Then,  for every $w_0\in H^1(\R)$  there
			exists a unique solution  $\Psi\in  C(\R,\phi_0+H^1(\R))$
			to \eqref{ngp} with the initial condition $\Psi_0=\phi_0+w_0$.
			 Moreover, the energy is  conserved, as well as the momentum 
			as long as $\inf_{x\in\R}\abs{\Psi(x,t)}>0$.
			
			In the case (ii), we also have 	the growth estimate
			\bq\label{linear} \norm{\Psi(t)-\phi_0}_{L^2(\R)}\leq C \abs{t}+\norm{\Psi_0-\phi_0}_{L^2(\R)},\eq
			for any $t\in \R$, where $C$ is a positive constant that depends only on $E(\Psi_0)$,
			$\norm{\wh \W}_{L^\infty},$ $\phi_0$ and $\sigma$.
	\end{thm}
Let us remark that the author in \cite{de2010global} uses a sightly different definition of the momentum to allow a possible vanishing of $\Psi(t)$. However, the proof of the 
conservation of momentum in \cite{de2010global} also applies to our renormalized momentum  as long as  $\Psi(t)\in \boEEE$.
We also notice that other statements for Cauchy problem for the Gross--Pitaevskii equation have been established  in different topologies when $\W=\delta_0$ (see e.g.~\cite{zhidkov2001korteweg,gerard,gallo,bethuel2008existence,delaire-gravejat-sine,deLaGra1} and the reference there in), and  these results
can probably be adapted to our nonlocal framework.

For the proof of Theorem~\ref{global0}, the author proves first 
a local well-posedness result for $\W\in \boM_3(\R)$.
Then  conditions (i) and (ii)  are used to show that the solution is global.
In \cite{de2010global}, it is also established that the solution is global in dimensions
greater than 1, provided that $\wh\W\geq \sigma>0$ a.e. However, the proof given by the author 
does not apply in the one-dimensional case.  Using Lemma~\ref{lem-control-energy}, we can partially fill
this gap.
	\begin{thm}\label{global}
Let $\phi_0$ and $\W$ as in Theorem~\ref{global0}, but 
instead of (i) or (ii), we 	
assume that there exists $\kappa\geq 0$ such that
	\bq\label{cond:global}
\wh	\W(\xi)\geq (1-\kappa\xi^2)^+, \quad \text{a.e.~on } \R.
	\eq
Then we have the same conclusion as in  Theorem~\ref{global}, 
including the growth estimate \eqref{linear}, 
with a constant $C$  depending only on $E(\Psi_0)$,
$\norm{\wh \W}_{L^\infty}$, $\phi_0$ and $\kappa$.
\end{thm}
\begin{proof}
In view of the local well-posedness established in Theorem~1.10 in~\cite{de2010global}, 
to prove that the solution is global, 
we only need to show that the solution $\Psi(t)=\phi_0+w(t)$ defined $(T_{\min},T_{\max})$, 
satisfies $T_{\max}=\infty$ and $T_{\min}=-\infty$. In view of the blow-up alternative in 
 the mentioned theorem, it is sufficient to prove that  
$\norm{w(t)}_{L^2(\R)}$ remains bounded in any bounded interval of  $(T_{\min},T_{\max})$.
Indeed, from \eqref{ngp}, we have (see equation (63) in \cite{de2010global})
\begin{align*}
	\frac12 {\left|\frac{d}{dt}\norm{w(t)}_{L^2(\R)}^2\right|} 
	&\leq \norm{ \phi_0''}_{L^2(\R)}\norm{w(t)}_{L^2(\R)}+
	\norm{\phi_0}_{L^\infty(\R)}\int_{\R} \abs{\W*(1-\abs{u(t)}^2)} \, \abs{w(t)}\,dx\\
		&\leq \norm{ \phi_0''}_{L^2(\R)}\norm{w(t)}_{L^2(\R)}+
	\norm{\phi_0}_{L^\infty(\R)}\norm{\wh \W}_{L^\infty(\R)} \norm{\eta(t)}_{L^2(\R)} \norm{w(t)}_{L^2(\R)},
\end{align*}
where $\eta(t)=1-\abs{u(t)}^2$. From Lemma~\ref{lem-control-energy}, we deduce
from the conservation of energy on  $(T_{\min},T_{\max})$, that there exists 
a constant $K>0$, depending on $\kappa$ and $E(\Psi_0)$, such that 
$$\norm{\eta(t)}_{L^2(\R)}\leq K,\quad  \text{for all }t\in~(T_{\min},T_{\max}).
$$
Therefore, we have  for any $\delta>0$, 
$$\frac12 \left|{\frac{d}{dt}(\norm{w(t)}_{L^2(\R)}^2+\delta)}\right|\leq
 (\norm{w(t)}^2_{L^2}+\delta)^\frac12
\left(
\norm{ \phi_0''}_{L^2(\R)}+K\norm{\wh \W}_{L^\infty(\R)}\norm{\phi_0}_{L^\infty(\R)}\right).$$
Dividing by $(\norm{w(t)}_{L^2(\R)}^2+\delta)^{\frac12}$, integrating and letting $\delta\to 0$,
we obtain \eqref{linear}, for any $t\in~(T_{\min},T_{\max})$. As mentioned above, 
this estimate implies that the solution is global.
\end{proof}

As 	explained in Section~6 in~\cite{de2010global}, Theorem~\ref{global} allows us to show that the solutions 
in the energy space are global. 
	\begin{thrm}
		\label{thm:cauchy}
		Assume that  $\W\in \boM_{3}(\R)$  is an even distribution satisfying \eqref{cond:global}. Then for every $\Psi_0 \in \mathcal{E}(\mathbb{R}),$ there exists a unique $\Psi \in C(\mathbb{R},\mathcal{E}(\mathbb{R}))$  global  solution to \eqref{ngp} with the initial condition $\Psi_0$. Moreover, the energy is  conserved, as well as the momentum as long as 
 $\inf_{x\in\R}\abs{\Psi(x,t)}>0$.
	\end{thrm}

\begin{proof}[Proof of Theorem~\ref{thm:cauchy0}]
In view of Remark~\ref{rem:estimation}, we deduce if $\W\in \boM_3(\R)$ is an even distribution, 
with $\wh \W\geq 0$ a.e.~on $\R$, and $\wh \W$ of class $C^2$ in a neighborhood of the origin, then 
$\W$ satisfies  \eqref{cond:global}, for some $\kappa\geq 0$. Therefore, we can apply Theorem~\ref{thm:cauchy}
and the conclusion follows.
\end{proof}

 The rest of the section is devoted to prove that the set $\boS_\q$ 
is orbitally stable in the energy space.
Using the Cazenave--Lions approach \cite{cazlions} and  Theorem~\ref{thm:compacite},
we obtain the following result. 
\begin{thm}
	\label{thstab}
	Assume that  $\W\in \boM_{3}(\R)$ satisfies \ref{H0} and \ref{H-coer}.
	Suppose also that  $E_{\min}$ is concave on $\mathbb{R}^+$. 
Then, $\boS_\q$ is orbitally stable 
for $(\boEE,d)$ and  for $(\boEE,d_A)$, for all $\q \in (0,\q_*)$.
Moreover,
for all $\Psi_0\in \boEE$ and for all $\varepsilon>0,$ there exists $\delta >0$ such that if 
\bq\label{stable}
 d(\Psi_0,\boS_\q)\leq \delta,\quad  \text{ then } \quad \sup_{t\in\R } \inf_{y\in \R}d_A(\Psi(\cdot-y,t),\boS_\q)\leq \varepsilon,
 \eq
where $\Psi(t)$ is the solution of \eqref{ngp} associated with the initial condition $\Psi_0$.
\end{thm}
Notice that for $u,v\in \boEE$, we have  $d(u,v)\leq d_A(u,v)$, and thus
$$
d(u,\boS_\q)=\inf_{y\in \R} d(u(\cdot-y),\boS_\q)\leq \inf_{y\in \R} d_A(u(\cdot-y),\boS_\q).
$$
Therefore, the implication in \eqref{stable} shows the orbital stability
for the distance $d$ and $d_A$.

In order to prove Theorem \ref{thstab}, we will use the following lemma.
\begin{lmm}
	\label{lmmstabi}
Let $v_{n}, v\in \mathcal{E}(\mathbb{R})$ such that
$	d(v_n,v)\to 0$. 
	Then,
	\bq
		\label{convstablemmeenergie2}
	\Vert |v_n|-|v| \Vert_{L^{\infty}(\R)} \to 0
	\quad \text{and}	\quad \Vert |v_n|^2-|v|^2 \Vert_{L^{2}(\R)} \to 0.
		\eq 
In particular, we have the continuity of the energy $E(v_n)\to E(v)$ (with respect to $d$). In addition, if $v_n,v\in \boEEE$, then we also have the continuity of the momentum
$p(v_n)\to p(v)$.
\end{lmm}
\begin{proof}
	First, we remark that since  $d(v_n,v)\to 0$, there is some $M>0$ such that $$\norm{v_n'}_{L^2(\R)}+\norm{v'}_{L^2(\R)}+\norm{v_n}_{L^2(\R)}+\norm{v}_{L^2(\R)}\leq M,$$
	 for all $n\in \N$.
	By the sharp Gagliardo--Nirenberg interpolation inequality and using that
	$||w|'|= |w'|$, for $w\in H^1_{\loc}(\R)$,  we have
	\begin{align*}
	\norm{ |v_n|-|v| }_{L^{\infty}(\R) }
			\leq \Vert |v_n|-|v| \Vert_{L^{2}(\R)}
				\Vert |v_n|'-|v|' \Vert_{L^{2}(\R)}
					&\leq 2M\Vert |v_{n}|-|v|\Vert_{L^{2}(\R)},
	\end{align*}
	so the first convergence in \eqref{convstablemmeenergie2} follows.
Similarly, we deduce the second one noticing that
	\begin{align*}
		\Vert |v|^{2}-|v_{n}|^{2} \Vert_{L^{2}(\R)} 
		\leq \left( \Vert v \Vert_{L^{\infty}(\R)} + \Vert v_n \Vert_{L^{\infty}(\R)}\right)\Vert |v|-|v_{n}| \Vert_{L^{2}(\mathbb{R})}\leq 2M\Vert |v_{n}|-|v|\Vert_{L^{2}}.
	\end{align*} 
	Therefore \eqref{convstablemmeenergie2} is proved. In particular, we have
	$v_n\to v$ in $L^2(\R)$ and $\eta_n=1-|v_n|^2\to \eta=1-|v|^2$ in $L^2(\R)$,
	so that $E(v_n)\to E(v)$. For the momentum, writing $v_n=\abs{v_n}e^{i\theta_n}$ as usual,
		we have $p(v_n)=\frac12\int_\R \eta_n \theta_n'$, so it suffices to prove that 
		$\theta_n \wto \theta$ in $L^2(\R)$ to conclude that $p(v_n)\to p(v)$, 
		where $v=\abs{v}e^{i\theta}$. To establish the weak convergence of $\theta_n$, 
		we notice that since 
			$|v_n|\to |v|$ in $L^{\infty}(\R)$, there exists $C>0$ such that 
$$\inf_{\R}{\abs{v_n}}\geq C, \quad \text{for all }n\in \N.$$
Hence,
	$$ \int_{\mathbb{R}}\theta_n'^{2} \leq  \frac{1}{C^2}\int_{\mathbb{R}}\rho_n^2\theta_n'^2 \leq \frac{2}{C^{2}}E(v_{n}).$$
	Since $E(v_{n})$ is bounded, we conclude as in Lemma~\ref{Thsolitonw} that for a subsequence, 
$\theta_{n_k}'\wto \theta'$ in $L^2(\R)$, as $k\to\infty$. Therefore, we conclude that 
$p(v_{n_k})\to p(v)$. Since the limit  does not depend on the subsequence, 
we deduce that $p(v_{n})\to p(v)$.
\end{proof}

\begin{proof}[Proof of Theorem \ref{thstab}]
	Arguing by contradiction,
we suppose that there exist $\varepsilon_{0}>0$, $(\delta_n)$,  $(t_n)$ and $(u_0^n) \subset \boE(\R)$ such that $\delta_{n} \to 0,$
	\bq
	d(u_{0}^n,\boS_\q)<\delta_n
	\label{ineqdistance}
	\eq 
	and
	\bq
	\label{ineq:absurd:Sq}
\inf_{y\in \R}	d_A(u^{n}(\cdot-y,t_n),\boS_\q) \geq \varepsilon_{0},
	\eq
	where $u^n$ denotes the solution to \eqref{ngp} with initial data $u_{0}^n$.
In particular, from \eqref{ineqdistance}	we deduce that there is   $v_n \in \boS_\q$ such that
	\bq
	\label{d:vn:u0}
	d(u_0^n,v_n)<2\delta_n.
	\eq 
Since $E(v_n)=\Emin(\q)$ and $p(v_n)=\q$, applying Theorem \ref{thm:compacite} to $(v_n)$, 
we infer that there exists $v \in \boS_\q$ and points $(a_n)$  such that, up to a subsequence,
the function $\tilde v_n(x)=v_n(x+a_n)$ satisfies 
	\begin{align}
		\label{cv:stab:unif} 
\tilde 		v_n\to  v,\ \text{ in }L^\infty_{\loc}(\R),
\quad \text { and	}\quad 
		1-|\tilde v_n|^{2} \to 1-|v|^{2}, \   
\tilde		v_n'\to v' \ \text{ in }L^2(\R).
	\end{align}
Using also the estimate \eqref{inf-u_n} in Theorem~\ref{thm:compacite}, 
we conclude that
	\bqq
\norm{ |\tilde v_n|-|v| }_{L^2(\R)}\leq 
	\frac{1}{\nu+\inf_{\R}\abs{v}}
\Vert |\tilde v_n|^2-|v|^2 \Vert_{L^2(\R)}\to 0,
	\eqq 
so that 
\bq 
\label{conv-d}
d(\tilde v_n,v)\to 0,
\eq
and also $d_A(\tilde v_n,v)\to 0$. On the other hand, by the triangle inequality and \eqref{d:vn:u0},
	\bqq
	d(u_0^n(\cdot+a_n),v)\leq d(u_0^n(\cdot+a_n),\tilde v_n)+d(\tilde v_n,v)<2\delta_n+d(\tilde v_n,v).
	\eqq
	Combining with \eqref{conv-d}, we conclude that $d(u_0^n(\cdot+a_n),v)\to 0$.
	Applying the conservation of energy in Theorem~\ref{thm:cauchy} and Lemma~\ref{lmmstabi}, 
	we thus get, for all $t\in \R$,
	\bq
		\label{E:u0=vn}
	E(u^n(t))=E(u_{0}^n)=E(u_0^n(\cdot+a_n))\to E(v)=\Emin(\q). 
	\eq 
	At this point we claim that 
	\bq\label{claim-p}
	\inf_{\R}\abs{u^n(t)}>0, \quad \text{ for all }\abs{t}\leq \abs{t_n}. 
	\eq
	Otherwise, there are values $s_n$, with $\abs{s_n}\leq \abs{t_n}$, such that   
	$\inf_{\R}\abs{u^n(s_n)}=0$. By	\eqref{E:u0=vn}, we conclude that $E(u^n(s_n))\to\Emin(\q)$
	and thus, using that $\Emin$ is strictly increasing on $(0,\q_*)$, we can find $n_0$
	such that  $E(u^n(s_n))<\Emin(\q_*)$, for all $n\geq n_0$.
This is a contradiction because, by Theorem~\ref{thm:curve}, this implies that 
$u^n(s_n)\in \boEEE$.

	In view of \eqref{claim-p}, we can proceed as before 
invoking the conservation of momentum in Theorem~\ref{thm:cauchy} and Lemma~\ref{lmmstabi}, 
to obtain
	\bq
\label{p:u0=vn}
p(u^n(t_n))=p(u_{0}^n)=p(u_0^n(\cdot+a_n))\to p(v)=\q. 
\eq 
	
By  \eqref{E:u0=vn} and \eqref{p:u0=vn}, we can apply Theorem~\ref{thm:compacite} to $(u^n(t_n))$. Then,   reasoning as before, we deduce that there exist $w \in S_\q$ and $(b_n)$ such that, up to a subsequence,
	\bq
	d_A(u^n(\cdot+b_n,t_n),w(\cdot))\to 0,
	\eq 
	which contradicts \eqref{ineq:absurd:Sq}.
\end{proof}
\section{Euler--Lagrange equations and proof of Theorem~\ref{thm-existence-general}}
In this section we establish the Euler--Lagrange equations associated with the minimization problem, 
which will allow us to complete the proof of Theorem~\ref{thm-existence-general}.
Since the energy and momentum functional are not defined on a vector space, the notion of differential is not trivial. 
For our purposes, it suffices consider the directional derivatives using only smooth functions with compact support. More precisely, 
for $u\in \boEE$  we define 
\bqq
\d E(u)[h]:=\underset{t \to 0}\lim\frac{E(u+th)-E(u)}{t} \quad \text{ and }\quad
\d p(u)[h]:=\underset{t \to 0}\lim\frac{p(u+th)-p(u)}{t},
\eqq
for all $h\in C_c^\infty(\R)$, where we also suppose that $u\in\boEEE$ for the definition of $\d p(u)$
so that  $p(u+th)$ is actually well defined for $t$ small enough.

\begin{lem}
	\label{diffmoment}
	Assume that  $\W$ satisfies \ref{H0}. 
	Then for all $ h \in C_c^\infty(\R)$, we have
	\begin{align}
		\label{marloo}
		\d E(u)[h]&=
		\int_{\mathbb{R}}\langle u',h' \rangle - \int_{\mathbb{R}}\W\ast (1-|u|^2)\langle u,h \rangle, \quad\text{ if }u\in \boEE,
		\\
		\label{marlo}
		\d p(u)[h]&=\int_{\mathbb{R}}\langle ih',u \rangle, \quad\text{ if }u\in \boEEE.
	\end{align}
	In particular, for all $c\in \R$,   $\d E(u)=c\,\d p(u)$ if and only if  $u$ satisfies \eqref{ntwc}. 
\end{lem}

Notice that the elliptic regularity theory shows that if $u$ is a solution of \eqref{ntwc}, then 
 $u$
is smooth. More precisely, the following result stated in higher dimensions
 in \cite{de2011nonexistence} applies without changes in dimension 1.
\begin{lem}[\cite{de2011nonexistence}]
	\label{regularity} 
	Let $u\in \boEE$ be a solution of \eqref{ntwc}, with $\W\in\boM_2(\R)$. Then $u$ is bounded and of class $C^\infty(\R)$.
	Moreover,  $\eta:=1-\abs{u}^2$ and $\grad{u}$ belong to $W^{k,p}(\R)$,
	for all $k\in\N$ and for all  $p\in[2, \infty)$.
\end{lem}
\begin{proof}[Proof of Lemma~\ref{diffmoment} ]
	Using \eqref{sym}, the differential in \eqref{marloo} is a straightforward consequence of the definition of $\d E$.
	To show \eqref{marlo}, let us fix $u\in \boEEE$ and $h\in C_c^\infty(\R)$. Then 
	\begin{align*}
		\d p(u)[h]&= 
		\left.\frac{d}{dt}p(u+th)\right|_{t=0} \\
		&= \frac{1}{2}\int_{\mathbb{R}}\langle ih',u \rangle \left( 1-\frac{1}{|u|^2}\right)+ 
		\frac{1}{2}\int_{\mathbb{R}}\langle iu',h \rangle \left( 1-\frac{1}{|u|^2}\right)+
		\int_{\mathbb{R}}\langle iu',u \rangle \left( \frac{\langle u,h \rangle}{|u|^4}\right)\\
		&=\int_{\mathbb{R}}\langle iu',h \rangle \left( 1-\frac{1}{|u|^2}\right)-\int_{\mathbb{R}} \frac{\langle ih,u \rangle \langle u',u \rangle}{|u|^4}+
		\int_{\mathbb{R}}\frac{\langle iu',u \rangle  \langle u,h \rangle}{|u|^4}.
	\end{align*}
	Therefore we obtain \eqref{marlo} noticing that 
	\bqq
	-\langle ih,u \rangle \langle u,u' \rangle 
	+\langle iu',u \rangle\langle u,h \rangle=
	\langle iu',h\rangle |u|^2. 
	\eqq
	The last assertion in the statement follows from the fact that if for some $v\in \boEE$ we have 
	$\int_\R\langle v,h\rangle=0$, for all $h\in  C_c^\infty(\R)$, then $v\equiv 0$.
\end{proof}
\label{sec:Euler}
\begin{thm}
	\label{thm:euler}
	Suppose that $\Emin$ is concave on $\R^+$ and that  $u\in \boS_\gq$, with $\gq>0$. Then there exists $c_\q$
	satisfying 
	\bq
	\label{cq:bound}
\Emin^+(\q) \leq c_\q \leq \Emin^-(\q),
		\eq 
such that $u$ is a solution of \eqref{ntwc} with of speed $c=c_\q$. 
\end{thm}

\begin{proof}[Proof]
	Let $u \in \boS_\gq$, so that $p(u)=\q$ and $E(u)=\Emin(\q)$. Notice that since $\q>0$,  $u$  is not a constant function.		Let $h \in C_c^{\infty}(\R)$. From the definition of $\Emin$ we have, 
for all $t>0$,
	$$ \frac{E(u+th)-E(u)}{t}\geq \frac{\Emin(p(u+th))-\Emin(\q)}{t}.$$
	If $\d p(u)[h]> 0$, then $p(u+th)\geq p(u)=\q$ for $t>0$ small enough, so that letting $t\to 0^+$, we obtain
	\bqq 
	\d E(u)[h]\geq \Emin^+(\q)\d p(u)[h].
	\eqq 
	Likewise, if $\d p(u)[h]< 0$, we get
	\bqq 
	\d E(u)[h]\geq\Emin^-(\q)\d p(u)[h].
	\eqq
	Replacing $h$ by $-h$, we obtain the following inequalities
	\bq
	\label{bound1:dE:dp}
	\Emin^+(\q) \d p(u)[h] \leq \d E(u)[h] 
	\leq \Emin^-(\q) \d p(u)[h],\quad\text{ if }\d p(u)[h]> 0,
	\eq  
	and
	\bq
	\label{bound2:dE:dp}
	\Emin^-(\q) \d p(u)[h] \leq \d E(u)[h] \leq \Emin^+(\q) \d p(u)[h],
	\quad\text{ if }\d p(u)[h]< 0.
	\eq  
Since the functionals $\d p(u),\d E(u) : C_c^\infty(\R)\to \R $ are linear, 
to establish the Euler--Lagrange equations,  it is enough to show that 
	\bq
	\label{inclusion}
	\Ker\d p(u)\subset \Ker\d E(u).
		\eq
		Indeed, by Lemma~3.2 in~\cite{brezis-book}, this implies that 
 there exists some $c_\q\in \R$ such that 
	\bq\label{eq-c}
	\d E(u)=c_\q\d p(u),
	\eq and therefore, by Lemma~\ref{diffmoment},
	 $u$ is a solution of \eqref{ntwc} with  $c=c_\q$
	
 To prove \eqref{inclusion}, let us consider $\phi \in \Ker\d p(u)$.
 	Since $u$ is nonconstant, there exists some function $\psi \in C_{c}^{\infty}(\mathbb{R})$
	such that  $\d p(u)[\psi]\neq 0$. Thus, for all $n \in \N$, we have
	$$ \d p(u)[\psi+n\phi]=\d p(u)[\psi]\neq 0.$$
	From \eqref{bound1:dE:dp} and \eqref{bound2:dE:dp}, we conclude  that $\d E(u)[\psi+n\phi]=\d E(u)[\psi]+n\d E(u)[\phi]$ is bounded. Hence $\d E(u)[\phi]=0$ i.e.\ $\phi \in \Ker \d E(u)$, which establishes
	\eqref{inclusion}.
	
 It remains to show \eqref{cq:bound}.
	Let $h_0\in C_c^{\infty}(\R)$ such that $\d p(u)[h_0]=1$.
Then \eqref{eq-c} implies that $\d E(u)[h_0]=c_\q$.
It follows from \eqref{bound1:dE:dp} that
	\bq
	\label{bound:c}
	\Emin^+(\q) \leq c_\q \leq \Emin^-(\q),
	\eq 
which finishes the proof.
\end{proof}

\begin{rem}
	It is possible to establish the Euler--Lagrange equations using an argument based on the implicit function theorem, without invoking the concavity of $\Emin$. Even thought the former argument is more general, we gave 
	the proof using the concavity because it is simpler.
\end{rem}

\begin{proof}[Proof of Theorem \ref{thm-existence-general}]
	Combining Theorems~\ref{thm:compacite},  \ref{thstab} and \ref{thm:euler}, we obtain that the set $S_\q$ is nonempty,  orbitally stable and that any $u \in \boS_\q$ is a solution of \eqref{ntwc}. 
	Using \eqref{cq:bound} and Theorem~\ref{thm:curve}-(v), we get
the properties for $c_\q$, except that $c_\q>0$. Arguing by contradiction, 
we suppose that there exists $\gp\in (0,\q_*)$ such that $c_\gp=0$. 
Thus, by \eqref{emin:fixed:bounds} and \eqref{speed0}, we get $\Emin^+(\gp)=0$.
Since 
$\Emin$ is concave, we have for all $\gr<\gs$,
\bqq
\Emin^-(\gr)\geq \Emin^+(\gr)\geq \Emin^-(\gs)\geq \Emin^+(\gs)\geq 0,
\eqq
which implies that $\Emin^-=\Emin^+=0$ on $[\gp,\infty)$, so that $\Emin$ is constant on 
$[\gp,\infty)$, which contradicts that $\Emin$ is strictly increasing on $[\gp,\q_*)$.
This completes the proof of the theorem.
\end{proof}

\section{Some numerical simulations}
\label{sec:numerics}
In this section, we numerically illustrate the properties of the minimizing curve through
some simulations. The numerical method is based on the projected gradient descent and the convolution is computed by the fast Fourier transform algorithm.  
Given $\W$ (or $\wh \W$) and some $\q>0$ close to~$0$,
we  compute the corresponding soliton $u_\q$ (i.e.\ $p(u_\q)=\q$) and its energy $E(u_\q)$.
We then increase the value of $\q>0$ until we obtain enough points to plot $\Emin$.

First, we show our results for the examples (i) and (ii) in Section \ref{intro}.
In Figures~\ref{fig:exp} and~\ref{fig:coth},   we can see $\Emin$  and the modulus of the solitons 
associated with $\q=0.05$, $\q=0.55$, $\q=1.1$ and $\q=1.5$, for the potentials 
\bq\label{W-exp}
\W_{\alpha,\beta}=\frac{\beta}{\beta-2\alpha}(\delta_0-\alpha e^{-\beta\abs{x}}), 
\eq
with $\alpha=0.05,$  $\beta=0.15$, and 
\bq\label{W-coth}
\W_{\alpha}=\frac{1}{1-\alpha}\big(\delta_0+\frac{3\alpha}{\pi}\ln(1-e^{-\pi\abs{x}})\big),
\eq
with $\alpha=0.8$. 
\begin{figure}[h]
	\begin{center}
		\begin{tabular}{lc}
{\scalebox{0.85}{\includegraphics[trim={0.5cm 0.5cm 0cm 0.5cm}, clip]{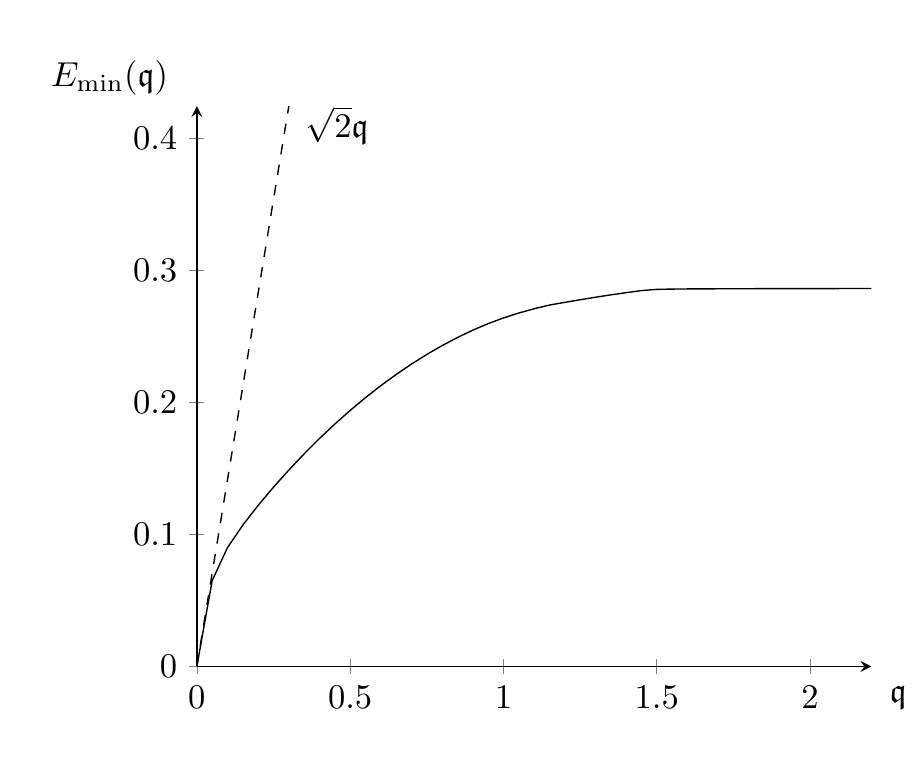}}}    
			&\hspace*{-0.7cm}
{\scalebox{0.85}{\includegraphics[trim={0cm 0.6cm 0.3cm 0.5cm}, clip]{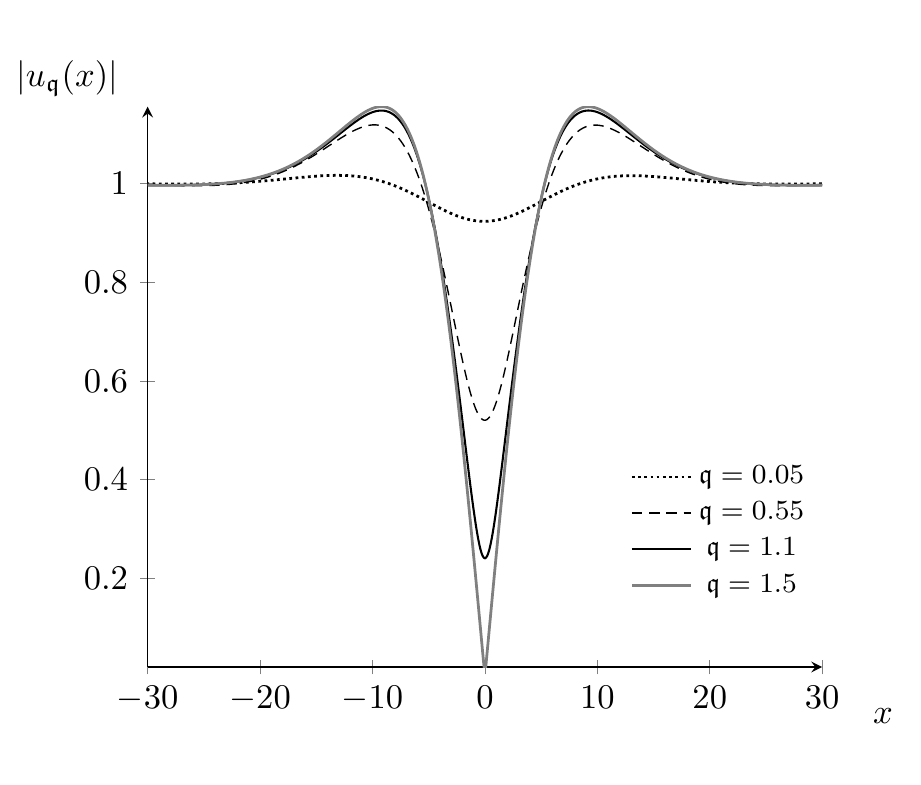}}}
		\end{tabular}
	\end{center}
	\caption{Curve $\Emin$ and solitons for the potential in \eqref{W-exp}, with $\alpha=0.05$ and $\beta=0.15$.}
	\label{fig:exp}
\end{figure}
 In both cases, we observe that  $\Emin$ is concave and that the  line $\sqrt{2}\q$ is a tangent to the curve. We notice that the shapes of the solitons in Figure~\ref{fig:coth} and the solitons in Figure~\ref{courbeth} are quite similar. On the other hand, the solitons in Figure~\ref{fig:exp} are very different, they have  values greater  than $1$ and exhibit a bump on $\R^+$. 
Notice also that the curves $\Emin$ for both potentials  seem to be constant for $\q>1.55$. 
\begin{figure}[h]
	\begin{center}
		\begin{tabular}{cl}
{\scalebox{0.85}{\includegraphics[trim={0.5cm 0.9cm 0.2cm 0.5cm}, clip]{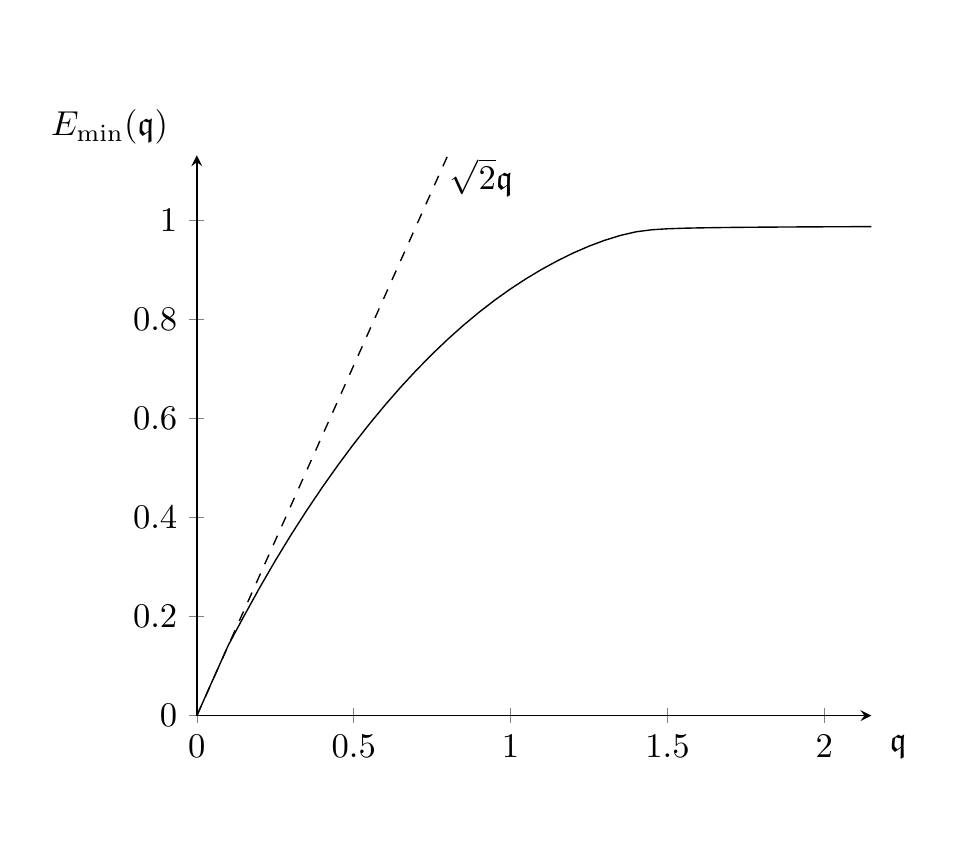}}}    
			&
			\hspace*{-0.7cm}
{\scalebox{0.85}{\includegraphics[trim={0.2cm 1cm 0.3cm 0.5cm}, clip]{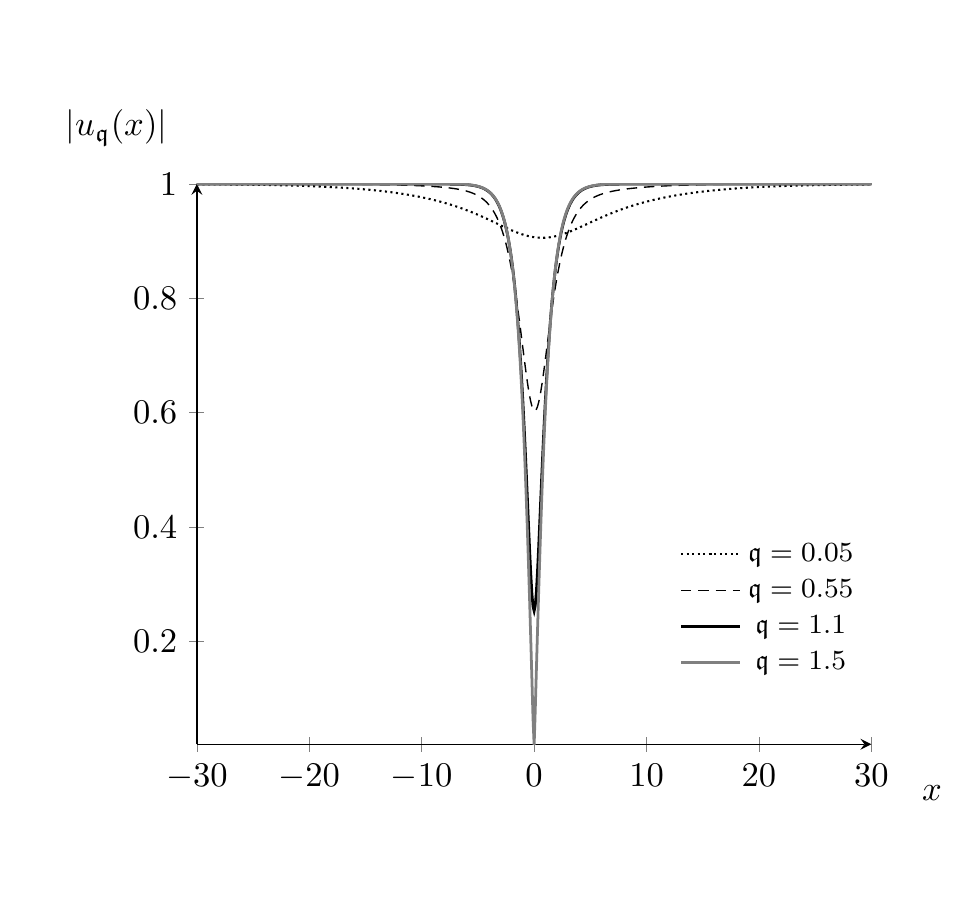}}}
	\end{tabular}
\end{center}
\caption{Curve $\Emin$ and solitons for the potential in  \eqref{W-coth},  with $\alpha=0.8$. }
\label{fig:coth}
\end{figure}

We end this section showing some numerical simulations for two interesting potentials.
The first one has been proposed in \cite{veskler2014} as simple model for interactions in a Bose--Einstein condensate. It is given  by a contact interaction $\delta_0$ and two Dirac delta functions centered at $\pm \sigma$,
\bq
\label{W-cos}
\W_\sigma=2\delta_{0}-\frac{1}2\left(\delta_{\sigma}+\delta_{-\sigma}\right).
\eq
Noticing that  $\wh \W_\sigma(\xi)=2-\cos(\sigma\xi)$, we see that for $\sigma>0$, 
$\W_\sigma$ fulfills \ref{H0}, \ref{H-coer}, and that $\wh \W_\sigma$ is analytic in $\C$, but is exponentially growing on $\H$. Thus, $\W_\sigma$ does not satisfy the assumption \eqref{decayW} in \ref{H-residue}. 
We can also check that \ref{H-residue-bis} is not fulfilled.
Nevertheless,  the results of the simulation depicted
in Figure~\ref{fig:cos} show that $\Emin$ is concave, and in that case Theorem~\ref{thm-existence-general} gives the orbital stability of the solitons illustrated in   
Figure~\ref{fig:cos}.
\begin{figure}[h]
	\begin{center}
		\begin{tabular}{cc}
{\scalebox{0.85}{\includegraphics[trim={0.5cm 1cm 0.4cm 1.1cm}, clip]{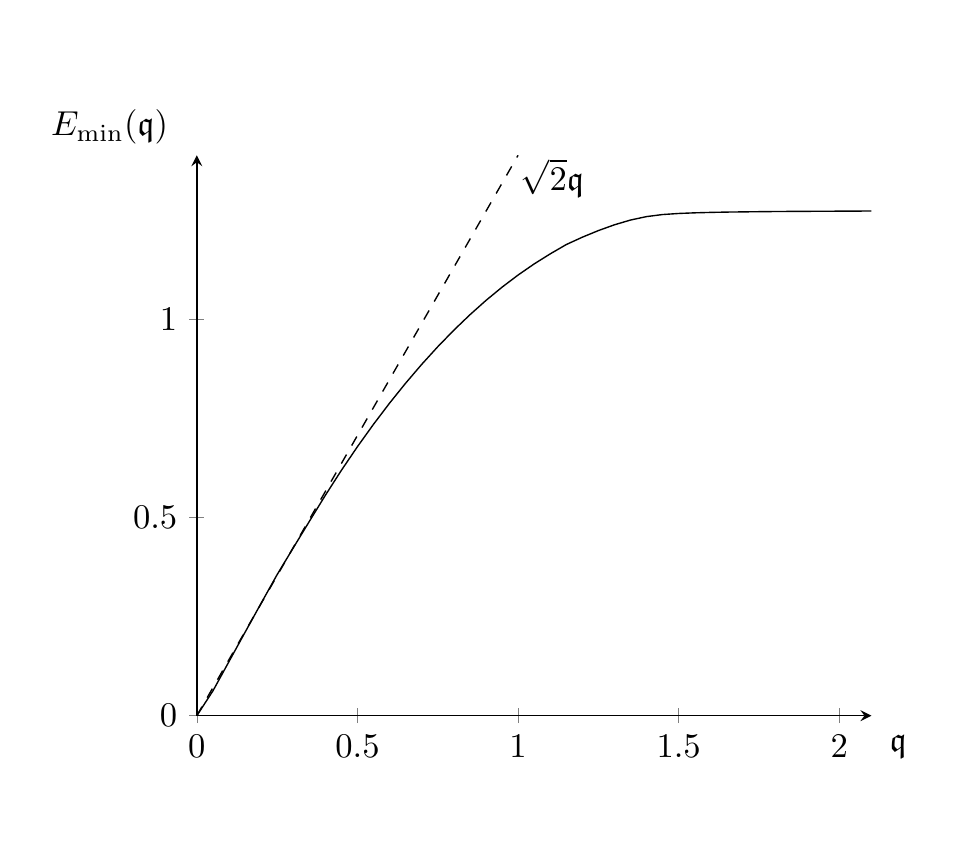}}}    
			&\hspace*{-0.7cm}
{\scalebox{0.85}{\includegraphics[trim={0.1cm 1.1cm 0.3cm 1.1cm}, clip]{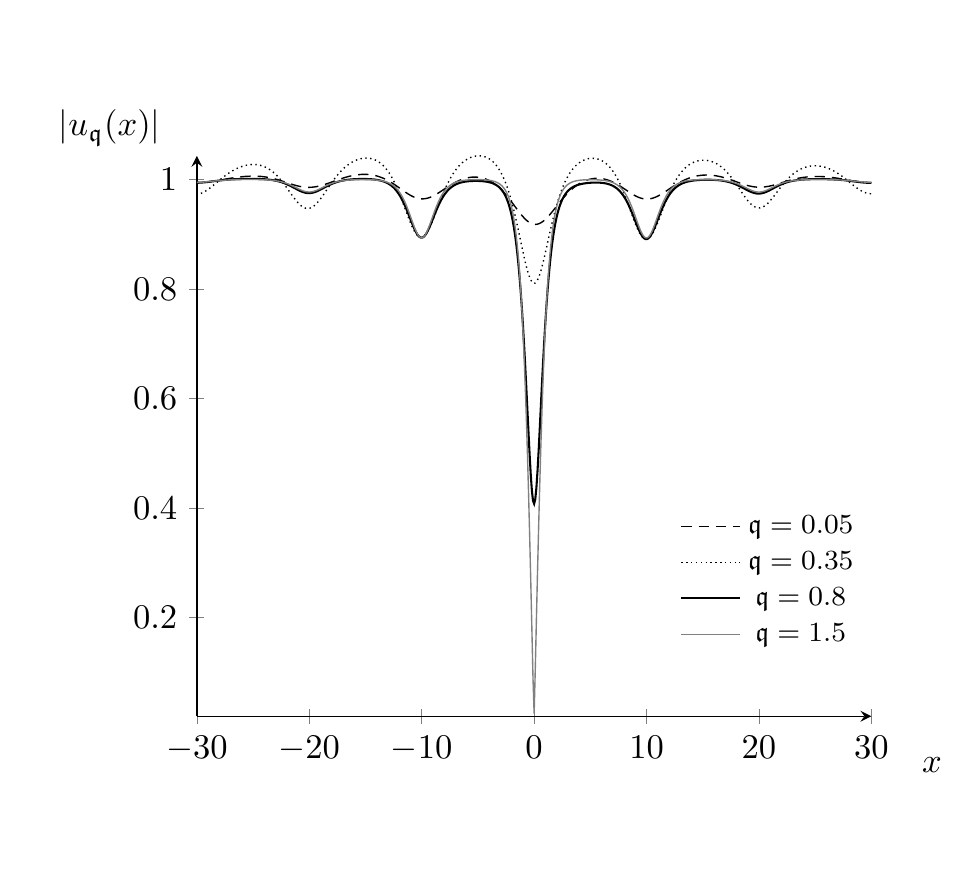}}}
			 		\end{tabular}
				\end{center}
		\caption{Curve $\Emin$ and solitons for the potential in \eqref{W-cos}, with  $\sigma=10$.}
	\label{fig:cos}
\end{figure}

Finally, we consider the potential 
\bq
\label{W-roton}
\wh \W_{a,b,c}(\xi)=(1+a\xi^2+b\xi^4)e^{-c\xi^2},
\eq
that it has been proposed in   \cite{berloff0,reneuve2018}
  to describe a quantum fluid exhibiting a roton-maxon spectrum such as Helium~4.
Indeed, as predicted by the Landau theory, in such a fluid, the dispersion curve \eqref{bogo}
cannot be monotone and it should have a local maximum and a local minimum, that are the so-called maxon and roton, respectively. In Figure~\ref{disp-curve}, we see the dispersion curve associated with potential 
  \eqref{W-roton},   with 
   $a=-36$, $b=2687$, $c=30$. In this case, there is a maxon at $\xi_m\sim 0.33$ and a roton at $\xi_r\sim 0.53$.
     \begin{figure}[h]
   	\centering
   	{\scalebox{0.85}{\includegraphics[trim={0.75cm 1.1cm 0.3cm 1.1cm}, clip]{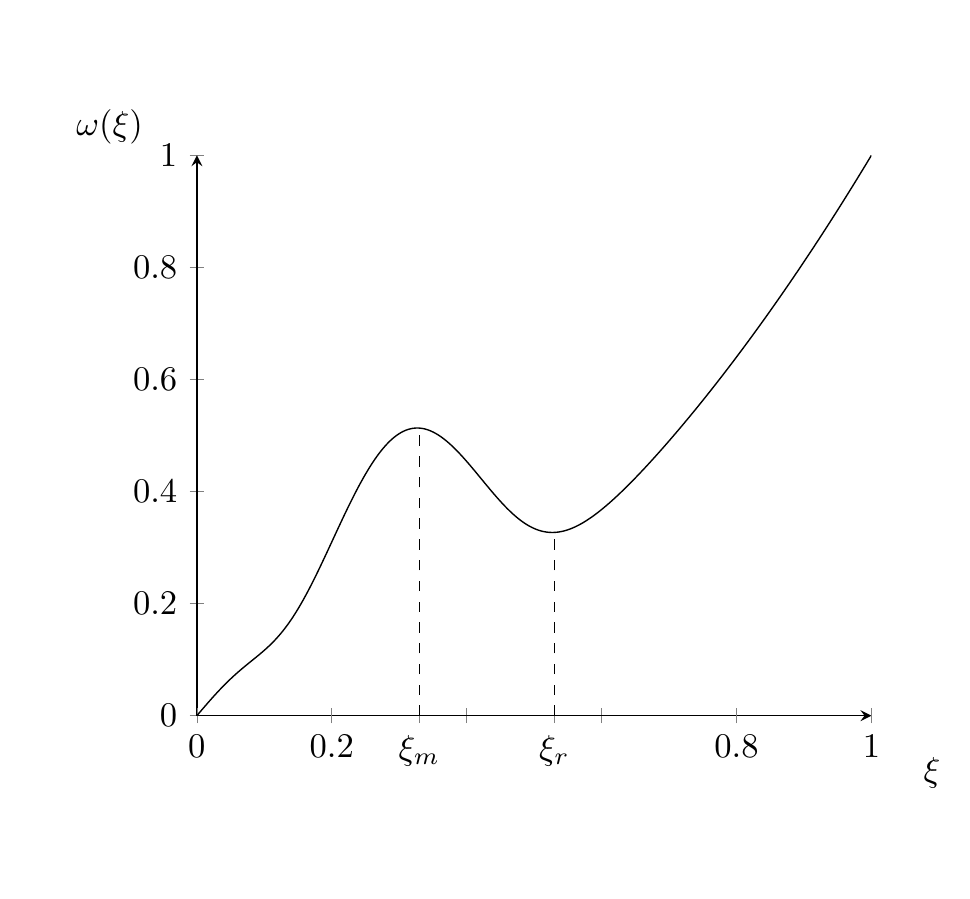}}}    
   	\caption{Dispersion curve associated with potential 
   		\eqref{W-roton}, with 	$a=-36$, $b=2687$, $c=30$. Here $\xi_m\sim 0.33$ and  $\xi_r\sim 0.53$.}
   	\label{disp-curve}
   \end{figure}
For these values,  \ref{H0} is satisfied, but not 
\ref{H-coer} nor \ref{H-residue-bis}. However, we observe 
in Figure~\ref{fig:roton} that the energy curve is still concave, 
and   that the straight line $\sqrt{2}\q$ is still a tangent to the curve. 
Moreover, we found the same critical value as before for the momentum, i.e.\ $\q_*\sim 1.55$.
   			\begin{figure}[h]
   	\begin{center}
   		\begin{tabular}{cc}
{\scalebox{0.85}{\includegraphics[trim={0.2cm 1cm 0.5cm 0.5cm}, clip]{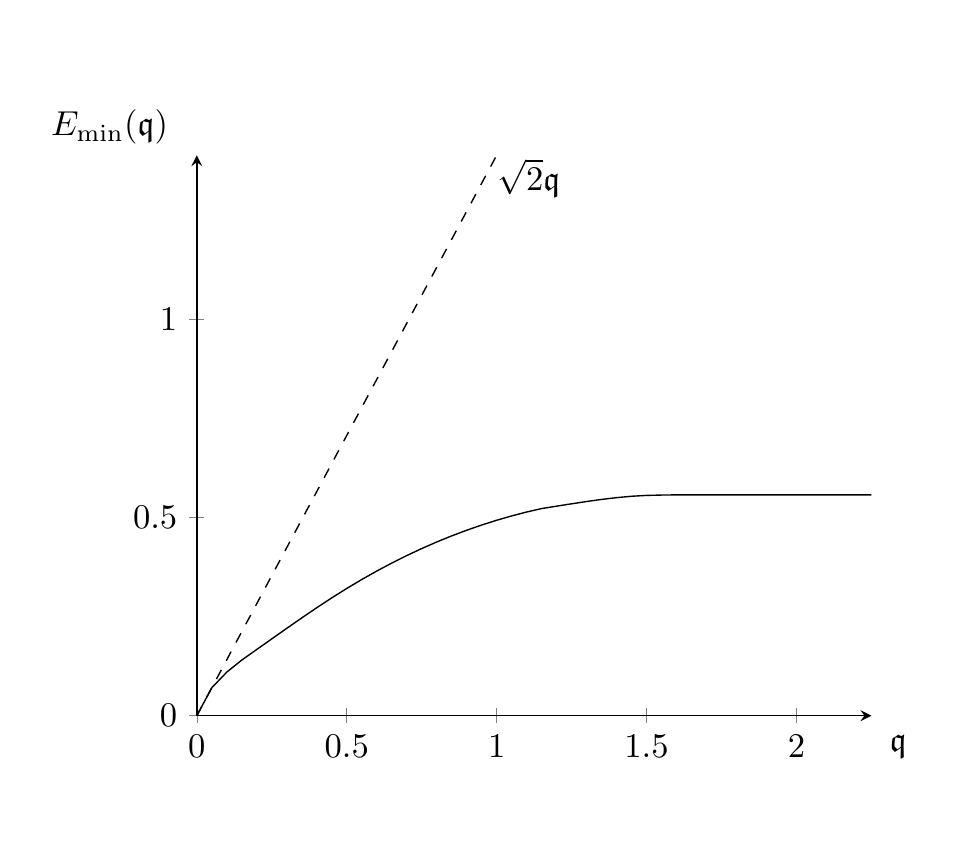}}}    
   			&\hspace*{-1cm}
{\scalebox{0.85}{\includegraphics[trim={0.2cm 1.1cm 0.3cm 0.5cm}, clip]{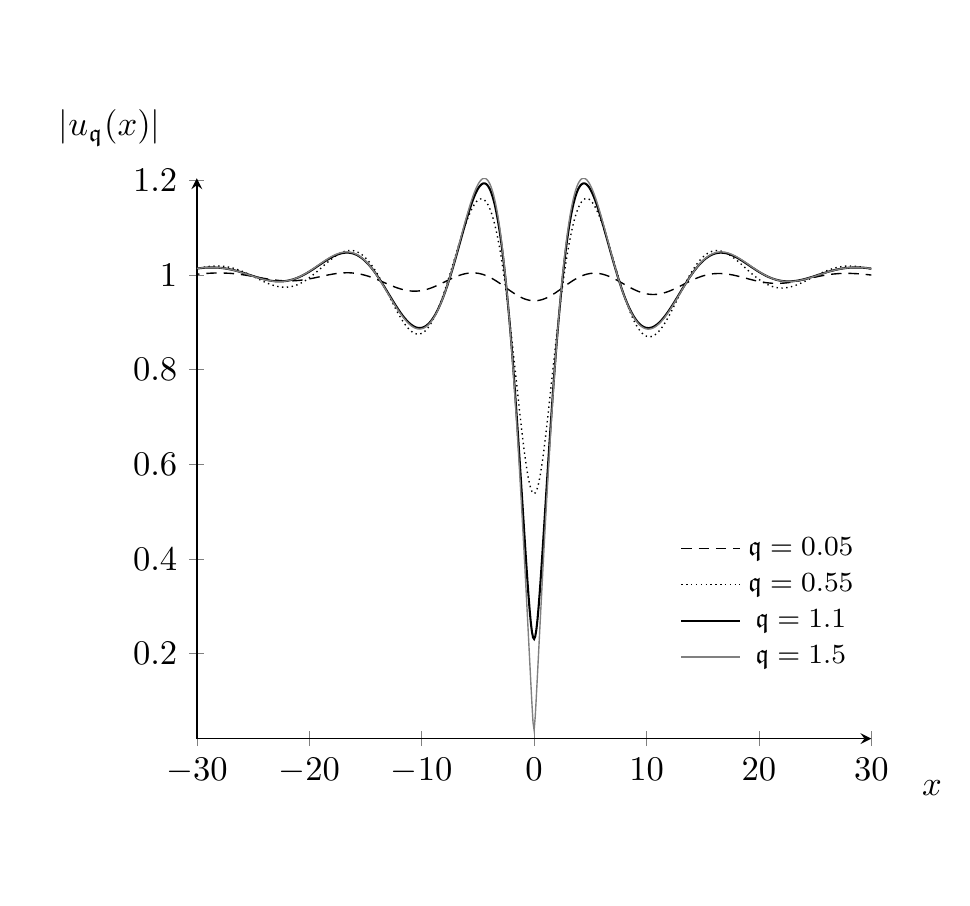}}}
   		\end{tabular}
   	\end{center}
   	\caption{Curves $\Emin$ and solitons for the potential in \eqref{W-roton}, with
  	$a=-36$, $b=2687$, $c=30$. }
   	\label{fig:roton}
   \end{figure}
\begin{merci}
	The authors acknowledge support from the Labex CEMPI (ANR-11-LABX-0007-01).
	A.~de Laire was also supported by the ANR project ODA (ANR-18-CE40-0020-01).
	 The authors are grateful to S.~De Bi\`evre and G.~Dujardin for interesting and helpful discussions.
		\end{merci}
\appendix
\section{Appendix}
\label{sec:appendix}
\begin{lmm}
	Let $R>0$ and $\mu >0.$ There exists a function $\chi \in C^{\infty}_{c}(\mathbb{R})$ such that 
for all $x\in \mathbb{R}$, 	 $0\leq \chi(x)\leq 1$, 
\bq
	\label{boundcutoff} \chi(x)=\begin{cases}
	1, &\text{ if }|x|\leq R, \\
	0, &\text{ if } |x| \geq R+\mu,
	\end{cases}
	\quad\text{and }\quad 
	|\chi'(x)|\leq 4e^{-2}e^{\frac{2}{\mu}}.
	\eq 
	
	\label{lmm:cutoff}
\end{lmm}
\begin{proof}
	Let $$ f(x):=\begin{cases}
	\exp(-\frac{1}{x}), &\text{ if }x\geq 0, \\
	0, &\text{ if }x< 0,
	\end{cases}
\quad\text{ and } 
\quad
	 \chi(x):=\frac{f(R+\mu-|x|)}{f(R+\mu-|x|)+f(|x|-R)}.
	$$
	Since 
	\bq
	\label{anex:ineq}
	f(|x|-R)+f(R+\mu-|x|)\geq f(\frac{\mu}{2})=2e^{-\frac{2}{\mu}},
	\eq the denominator of $\chi$ is always  positive, and thus $\chi$ is well defined. Moreover, 	$\chi \in C^{\infty}(\mathbb{R})$, since $f$ is smooth. Finally, for $|x|\leq R,$ we have $f(|x|-R)=0$, which implies that  $\chi(x)=1.$ For $|x|\geq R+\mu$, we have $f(R+\mu-|x|)=0$, so that $\chi(x)=0.$ 
	
	It remains to prove the bound in \eqref{boundcutoff}. Using that 
	$$ \chi'(x)= \frac{f'(R+\mu-|x|)f(|x|-R)+f'(|x|-R)f(R+\mu-|x|)}{(f(|x|-R)+f(R+\mu-|x|))^2},$$
and that $|f'(x)|\leq \frac{\exp(-1/x)}{x^2}\leq 4e^{-2},$
	we get 
	$$ |\chi'(x)|\leq \frac{8e^{-2}}{f(|x|-R)+f(R+\mu-|x|)}.$$
Combining with \eqref{anex:ineq}, we conclude that 	$|\chi'(x)|\leq 4e^{-2}e^{\frac{2}{\mu}}$.
\end{proof}

\bibliographystyle{abbrv}   
\bibliography{ref}

\begin{thebibliography}{10}

\bibitem{abid2003}
M.~Abid, C.~Huepe, S.~Metens, C.~Nore, C.~Pham, L.~Tuckerman, and M.~Brachet.
\newblock {G}ross-{P}itaevskii dynamics of {B}ose-{E}instein condensates and
  superfluid turbulence.
\newblock {\em Fluid Dynamics Research}, 33(5):509 -- 544, 2003.
\newblock Collection of Papers written by Regional Editors.

\bibitem{abram}
M.~Abramowitz and I.~A. Stegun.
\newblock {\em Handbook of mathematical functions with formulas, graphs, and
  mathematical tables}, volume~55 of {\em National Bureau of Standards Applied
  Mathematics Series}.
\newblock For sale by the Superintendent of Documents, U.S. Government Printing
  Office, Washington, D.C., 1964.

\bibitem{albert95}
J.~Albert.
\newblock Positivity properties and uniqueness of solitary wave solutions of
  the intermediate long-wave equation.
\newblock In {\em Evolution equations ({B}aton {R}ouge, {LA}, 1992)}, volume
  168 of {\em Lecture Notes in Pure and Appl. Math.}, pages 11--20. Dekker, New
  York, 1995.

\bibitem{antonelli2011}
P.~Antonelli and C.~Sparber.
\newblock Existence of solitary waves in dipolar quantum gases.
\newblock {\em Phys. D}, 240(4-5):426--431, 2011.

\bibitem{audiard2017}
C.~Audiard.
\newblock Small energy traveling waves for the {E}uler-{K}orteweg system.
\newblock {\em Nonlinearity}, 30(9):3362--3399, 2017.

\bibitem{becker2008}
C.~Becker, S.~Stellmer, P.~Soltan-Panahi, S.~D{\"o}rscher, M.~Baumert, E.-M.
  Richter, J.~Kronj{\"a}ger, K.~Bongs, and K.~Sengstock.
\newblock Oscillations and interactions of dark and dark--bright solitons in
  {B}ose-{E}instein condensates.
\newblock {\em Nature Physics}, 4(6):496, 2008.

\bibitem{jacopo2016}
J.~Bellazzini and L.~Jeanjean.
\newblock On dipolar quantum gases in the unstable regime.
\newblock {\em SIAM J. Math. Anal.}, 48(3):2028--2058, 2016.

\bibitem{berloff2008}
N.~G. Berloff.
\newblock Quantum vortices, travelling coherent structures and superfluid
  turbulence.
\newblock In {\em Stationary and time dependent {G}ross-{P}itaevskii
  equations}, volume 473 of {\em Contemp. Math.}, pages 27--54. Amer. Math.
  Soc., Providence, RI, 2008.

\bibitem{berloff0}
N.~G. Berloff and P.~H. Roberts.
\newblock Motions in a {B}ose condensate {VI}. {V}ortices in a nonlocal model.
\newblock {\em J. Phys. A}, 32(30):5611--5625, 1999.

\bibitem{bethuel2008existence}
F.~B\'ethuel, P.~Gravejat, and J.-C. Saut.
\newblock Existence and properties of travelling waves for the
  {G}ross-{P}itaevskii equation.
\newblock In {\em Stationary and time dependent {G}ross-{P}itaevskii
  equations}, volume 473 of {\em Contemp. Math.}, pages 55--103. Amer. Math.
  Soc., Providence, RI, 2008.

\bibitem{bethuel}
F.~B{\'e}thuel, P.~Gravejat, and J.-C. Saut.
\newblock Travelling waves for the {G}ross-{P}itaevskii equation. {II}.
\newblock {\em Comm. Math. Phys.}, 285(2):567--651, 2009.

\bibitem{blacksoliton}
F.~B\'ethuel, P.~Gravejat, J.-C. Saut, and D.~Smets.
\newblock Orbital stability of the black soliton for the {G}ross-{P}itaevskii
  equation.
\newblock {\em Indiana Univ. Math. J.}, 57(6):2611--2642, 2008.

\bibitem{bethuel-kdv-2009}
F.~B\'ethuel, P.~Gravejat, J.-C. Saut, and D.~Smets.
\newblock On the {K}orteweg-de {V}ries long-wave approximation of the
  {G}ross-{P}itaevskii equation. {I}.
\newblock {\em Int. Math. Res. Not. IMRN}, (14):2700--2748, 2009.

\bibitem{BGS-2015}
F.~Bethuel, P.~Gravejat, and D.~Smets.
\newblock Asymptotic stability in the energy space for dark solitons of the
  {G}ross-{P}itaevskii equation.
\newblock {\em Ann. Sci. \'{E}c. Norm. Sup\'{e}r. (4)}, 48(6):1327--1381, 2015.

\bibitem{orlandiun}
F.~Bethuel, G.~Orlandi, and D.~Smets.
\newblock Vortex rings for the {G}ross-{P}itaevskii equation.
\newblock {\em J. Eur. Math. Soc. (JEMS)}, 6(1):17--94, 2004.

\bibitem{bethuel2}
F.~B\'ethuel and J.-C. Saut.
\newblock Travelling waves for the {G}ross-{P}itaevskii equation {I}.
\newblock {\em Ann. Inst. H. Poincar\'e Phys. Th\'eor.}, 70(2):147--238, 1999.

\bibitem{bogdan1989}
M.~Bogdan, A.~Kovalev, and A.~Kosevich.
\newblock Stability criterion in imperfect {B}ose gas.
\newblock {\em Fiz. Nizk. Temp.}, 15(5):511--514, 1989.
\newblock In Russian.

\bibitem{bogo}
N.~N. Bogoliubov.
\newblock On the theory of superfluidity.
\newblock {\em J. Phys. USSR}, 11:23--32, 1947.
\newblock Reprinted in: D. Pines, The Many-Body Problem (W. A. Benjamin, New
  York, 1961), p. 292-301.

\bibitem{brezis-book}
H.~Brezis.
\newblock {\em Functional analysis, {S}obolev spaces and partial differential
  equations}.
\newblock Universitext. Springer, New York, 2011.

\bibitem{carles2008}
R.~Carles, P.~A. Markowich, and C.~Sparber.
\newblock On the {G}ross-{P}itaevskii equation for trapped dipolar quantum
  gases.
\newblock {\em Nonlinearity}, 21(11):2569--2590, 2008.

\bibitem{cazenave}
T.~Cazenave.
\newblock {\em Semilinear {S}chr\"odinger equations}, volume~10 of {\em Courant
  Lecture Notes in Mathematics}.
\newblock New York University Courant Institute of Mathematical Sciences, New
  York, 2003.

\bibitem{cazlions}
T.~Cazenave and P.-L. Lions.
\newblock Orbital stability of standing waves for some nonlinear
  {S}chr\"odinger equations.
\newblock {\em Comm. Math. Phys.}, 85(4):549--561, 1982.

\bibitem{chironexistence1d}
D.~Chiron.
\newblock Travelling waves for the nonlinear {S}chr\"odinger equation with
  general nonlinearity in dimension one.
\newblock {\em Nonlinearity}, 25(3):813--850, 2012.

\bibitem{chironstab}
D.~Chiron.
\newblock Stability and instability for subsonic traveling waves of the
  nonlinear {S}chr\"odinger equation in dimension one.
\newblock {\em Anal. PDE}, 6(6):1327--1420, 2013.

\bibitem{chironmaris}
D.~Chiron and M.~Mari{\c{s}}.
\newblock Traveling waves for nonlinear {S}chr\"odinger equations with nonzero
  conditions at infinity.
\newblock {\em Arch. Ration. Mech. Anal.}, 226(1):143--242, 2017.

\bibitem{chiron-rousset}
D.~Chiron and F.~Rousset.
\newblock The {K}d{V}/{KP}-{I} limit of the nonlinear {S}chr\"odinger equation.
\newblock {\em SIAM J. Math. Anal.}, 42(1):64--96, 2010.

\bibitem{delaire2009}
A.~de~Laire.
\newblock Non-existence for travelling waves with small energy for the
  {G}ross-{P}itaevskii equation in dimension {$N\geq 3$}.
\newblock {\em C. R. Math. Acad. Sci. Paris}, 347(7-8):375--380, 2009.

\bibitem{de2010global}
A.~de~Laire.
\newblock Global well-posedness for a nonlocal {G}ross-{P}itaevskii equation
  with non-zero condition at infinity.
\newblock {\em Comm. Partial Differential Equations}, 35(11):2021--2058, 2010.

\bibitem{de2011nonexistence}
A.~de~Laire.
\newblock Nonexistence of traveling waves for a nonlocal {G}ross-{P}itaevskii
  equation.
\newblock {\em Indiana Univ. Math. J.}, 61(4):1451--1484, 2012.

\bibitem{deLaGra1}
A.~de~Laire and P.~Gravejat.
\newblock Stability in the energy space for chains of solitons of the
  {Landau}-{Lifshitz} equation.
\newblock {\em J. Differential Equations}, 258(1):1--80, 2015.

\bibitem{delaire-gravejat-sine}
A.~de~Laire and P.~Gravejat.
\newblock The {S}ine-{G}ordon regime of the {L}andau-{L}ifshitz equation with a
  strong easy-plane anisotropy.
\newblock {\em Ann. Inst. H. Poincar\'e Anal. Non Lin\'eaire},
  35(7):1885--1945, 2018.

\bibitem{denschlag2000}
J.~Denschlag, J.~E. Simsarian, D.~L. Feder, C.~W. Clark, L.~A. Collins,
  J.~Cubizolles, L.~Deng, E.~W. Hagley, K.~Helmerson, W.~P. Reinhardt, et~al.
\newblock Generating solitons by phase engineering of a {B}ose-{E}instein
  condensate.
\newblock {\em Science}, 287(5450):97--101, 2000.

\bibitem{gallo}
C.~Gallo.
\newblock The {C}auchy problem for defocusing nonlinear {S}chr\"odinger
  equations with non-vanishing initial data at infinity.
\newblock {\em Comm. Partial Differential Equations}, 33(4-6):729--771, 2008.

\bibitem{gerard3}
P.~G\'erard.
\newblock The {G}ross-{P}itaevskii equation in the energy space.
\newblock In A.~Farina and J.-C. Saut, editors, {\em Stationary and time
  dependent {G}ross-{P}itaevskii equations. Wolfgang Pauli Institute 2006
  thematic program, January--December, 2006, Vienna, Austria}, volume 473 of
  {\em Contemporary Mathematics}, pages 129--148. American Mathematical
  Society.

\bibitem{gerard}
P.~G{\'e}rard.
\newblock The {C}auchy problem for the {G}ross-{P}itaevskii equation.
\newblock {\em Ann. Inst. H. Poincar\'e Anal. Non Lin\'eaire}, 23(5):765--779,
  2006.

\bibitem{gilbarg}
D.~Gilbarg and N.~S. Trudinger.
\newblock {\em Elliptic partial differential equations of second order}.
\newblock Classics in Mathematics. Springer-Verlag, Berlin, 2001.
\newblock Reprint of the 1998 edition.

\bibitem{ginibre1980}
J.~Ginibre and G.~Velo.
\newblock On a class of nonlinear {S}chr\"{o}dinger equations with nonlocal
  interaction.
\newblock {\em Math. Z.}, 170(2):109--136, 1980.

\bibitem{grafakos}
L.~Grafakos.
\newblock {\em Classical {F}ourier analysis}, volume 249 of {\em Graduate Texts
  in Mathematics}.
\newblock Springer, New York, second edition, 2008.

\bibitem{asymptstabblack}
P.~Gravejat and D.~Smets.
\newblock Asymptotic stability of the black soliton for the
  {G}ross-{P}itaevskii equation.
\newblock {\em Proc. Lond. Math. Soc. (3)}, 111(2):305--353, 2015.

\bibitem{gross1963hydrodynamics}
E.~P. Gross.
\newblock Hydrodynamics of a superfluid condensate.
\newblock {\em Journal of Mathematical Physics}, 4(2):195--207, 1963.

\bibitem{gustafson2007}
S.~Gustafson, K.~Nakanishi, and T.-P. Tsai.
\newblock Global dispersive solutions for the {G}ross-{P}itaevskii equation in
  two and three dimensions.
\newblock {\em Ann. Henri Poincar\'{e}}, 8(7):1303--1331, 2007.

\bibitem{gustafson2009}
S.~Gustafson, K.~Nakanishi, and T.-P. Tsai.
\newblock Scattering theory for the {G}ross-{P}itaevskii equation in three
  dimensions.
\newblock {\em Commun. Contemp. Math.}, 11(4):657--707, 2009.

\bibitem{kartashov07}
Y.~V. Kartashov and L.~Torner.
\newblock Gray spatial solitons in nonlocal nonlinear media.
\newblock {\em Opt. Lett.}, 32(8):946--948, 2007.

\bibitem{visan2012}
R.~Killip, T.~Oh, O.~Pocovnicu, and M.~Vi\c{s}an.
\newblock Global well-posedness of the {G}ross-{P}itaevskii and cubic-quintic
  nonlinear {S}chr\"{o}dinger equations with non-vanishing boundary conditions.
\newblock {\em Math. Res. Lett.}, 19(5):969--986, 2012.

\bibitem{lahaye2009}
T.~Lahaye, C.~Menotti, L.~Santos, M.~Lewenstein, and T.~Pfau.
\newblock The physics of dipolar bosonic quantum gases.
\newblock {\em Reports on Progress in Physics}, 72(12):126401, 2009.

\bibitem{lieb77}
E.~H. Lieb.
\newblock Existence and uniqueness of the minimizing solution of {C}hoquard's
  nonlinear equation.
\newblock {\em Studies in Appl. Math.}, 57(2):93--105, 1976/77.

\bibitem{linbubbles}
Z.~Lin.
\newblock Stability and instability of traveling solitonic bubbles.
\newblock {\em Adv. Differential Equations}, 7(8):897--918, 2002.

\bibitem{lions84}
P.-L. Lions.
\newblock The concentration-compactness principle in the calculus of
  variations. {T}he locally compact case. {I}.
\newblock {\em Ann. Inst. H. Poincar\'e Anal. Non Lin\'eaire}, 1(2):109--145,
  1984.

\bibitem{lopes-maris}
O.~Lopes and M.~{Mari\c{s}}.
\newblock Symmetry of minimizers for some nonlocal variational problems.
\newblock {\em J. Funct. Anal.}, 254(2):535--592, 2008.

\bibitem{luo2018}
Y.~Luo and A.~Stylianou.
\newblock Ground states for a nonlocal cubic-quartic {G}ross-{P}itaevskii
  equation.
\newblock Preprint \url{http://arxiv.org/abs/1806.00697}.

\bibitem{maris2013}
M.~Mari\c{s}.
\newblock Traveling waves for nonlinear {S}chr\"{o}dinger equations with
  nonzero conditions at infinity.
\newblock {\em Ann. of Math. (2)}, 178(1):107--182, 2013.

\bibitem{maris2016}
M.~Mari\c{s}.
\newblock On some minimization problems in {${\bf R}^N$}.
\newblock In {\em New trends in differential equations, control theory and
  optimization}, pages 215--230. World Sci. Publ., Hackensack, NJ, 2016.

\bibitem{maris-non}
M.~Mari{\c{s}}.
\newblock Nonexistence of supersonic traveling waves for nonlinear
  {S}chr\"odinger equations with nonzero conditions at infinity.
\newblock {\em SIAM J. Math. Anal.}, 40(3):1076--1103, 2008.

\bibitem{moroz2013}
V.~Moroz and J.~Van~Schaftingen.
\newblock Groundstates of nonlinear {C}hoquard equations: existence,
  qualitative properties and decay asymptotics.
\newblock {\em J. Funct. Anal.}, 265(2):153--184, 2013.

\bibitem{pecher2012}
H.~Pecher.
\newblock Global solutions for 3{D} nonlocal {G}ross-{P}itaevskii equations
  with rough data.
\newblock {\em Electron. J. Differential Equations}, pages No. 170, 34, 2012.

\bibitem{pitaevskii1961vortex}
L.~P. Pitaevskii.
\newblock Vortex lines in an imperfect {B}ose gas.
\newblock {\em Sov. Phys. JETP}, 13(2):451--454, 1961.

\bibitem{reneuve2018}
J.~Reneuve, J.~Salort, and L.~Chevillard.
\newblock Structure, dynamics, and reconnection of vortices in a nonlocal model
  of superfluids.
\newblock {\em Phys. Rev. Fluids}, 3(11):114602, 2018.

\bibitem{veskler2014}
H.~Veksler, S.~Fishman, and W.~Ketterle.
\newblock Simple model for interactions and corrections to the
  {G}ross-{P}itaevskii equation.
\newblock {\em Phys. Rev. A}, 90(2):023620, 2014.

\bibitem{vocke15}
D.~Vocke, T.~Roger, F.~Marino, E.~M. Wright, I.~Carusotto, M.~Clerici, and
  D.~Faccio.
\newblock Experimental characterization of nonlocal photon fluids.
\newblock {\em Optica}, 2(5):484--490, 2015.

\bibitem{zakharov86}
V.~Zakharov and E.~Kuznetsov.
\newblock Multi-scales expansion in the theory of systems integrable by the
  inverse scattering transform.
\newblock {\em Phys. D}, 18(1-3):455--463, 1986.

\bibitem{zhidkov2001korteweg}
P.~E. Zhidkov.
\newblock {\em Korteweg-de {V}ries and nonlinear {S}chr\"odinger equations:
  qualitative theory}, volume 1756 of {\em Lecture Notes in Mathematics}.
\newblock Springer-Verlag, Berlin, 2001.

\end{thebibliography}

\end{document}